\def\year{2019}\relax
\documentclass[letterpaper]{article} 
\usepackage{aaai19}  
\usepackage{times}  
\usepackage{helvet}  
\usepackage{courier}  
\usepackage{url}  
\usepackage{graphicx}  

\usepackage{amsmath,amssymb}
\usepackage{extarrows}
\usepackage{braket}
\usepackage{bm}
\usepackage{booktabs}
\usepackage{subfig}
\usepackage{theorem}
\usepackage{algorithm}
\usepackage{algorithmic}
\usepackage{mydef}
\usepackage[noabbrev]{cleveref}
\crefname{equation}{}{}
\usepackage{xr}
\usepackage{pdfpages}

\usepackage{threeparttable}

\newtheorem{theorem}{Theorem}
\newtheorem{proposition}{Proposition}

\newtheorem{definition}{Definition}
\newtheorem{remark}{Remark}

\newtheorem{corollary}[theorem]{Corollary}

\newtheorem{proof}{Proof}

\nocopyright
\begin{document}
%
\title{B\'ezier Simplex Fitting: Describing Pareto Fronts of Simplicial Problems with Small Samples in Multi-objective Optimization}
\author{Ken Kobayashi,$^{1,2}$ Naoki Hamada,$^{1,2}$ Akiyoshi Sannai,$^{2,3}$ Akinori Tanaka,$^{2,3}$\\  \textbf{\Large{Kenichi Bannai,$^{3,2}$ Masashi Sugiyama$^{2,4}$}}\\
$^1$Artificial Intelligence Laboratory, Fujitsu Laboratories Ltd., Japan\\
$^2$The Center for Advanced Intelligence Project, RIKEN, Japan\\
$^3$Department of Mathematics, Faculty of Science and Technology, Keio University, Japan\\
$^4$Department of Complexity Science and Engineering, Graduate School of Frontier Sciences, The University of Tokyo, Japan}
\maketitle
\begin{abstract}
Multi-objective optimization problems require simultaneously optimizing two or more objective functions.
Many studies have reported that the solution set of an $M$-objective optimization problem often forms an $(M-1)$-dimensional topological simplex (a curved line for $M=2$, a curved triangle for $M=3$, a curved tetrahedron for $M=4$, etc.).
Since the dimensionality of the solution set increases as the number of objectives grows, an exponentially large sample size is needed to cover the solution set.
To reduce the required sample size, this paper proposes a B\'ezier simplex model and its fitting algorithm.
These techniques can exploit the simplex structure of the solution set and decompose a high-dimensional surface fitting task into a sequence of low-dimensional ones.
An approximation theorem of B\'ezier simplices is proven.
Numerical experiments with synthetic and real-world optimization problems demonstrate that the proposed method achieves an accurate approximation of high-dimensional solution sets with small samples.
In practice, such an approximation will be conducted in the post-optimization process and enable a better trade-off analysis.
\end{abstract}

\section{Introduction}
A multi-objective optimization problem is a problem that minimizes multiple objective functions $f_1,\dots,f_M:X \to \R$ over a common domain $X \subseteq \R^L$:
\begin{align*}
\mbox{minimize } & \bm f(\bm x) := (f_1(\bm x), \dots, f_M(\bm x))\\
\mbox{subject to } & \bm x \in X (\subseteq \R^L).
\end{align*}
Different functions usually have different minimizers, and one needs to consider a trade-off that two solutions $\bm x,\bm y \in X$ may satisfy $f_i(\bm x) < f_i(\bm y)$ and $f_j(\bm x) > f_j(\bm y)$.
According to \emph{Pareto ordering}, i.e.,
\begin{align*}
\bm f(\bm x) \dominates \bm f(\bm y)  \xLeftrightarrow{\text{def}}
&\ f_m(\bm x) \le f_m(\bm y)~\text{for all}~m = 1, \dots, M\\
\text{and}
&\ f_m(\bm x) < f_m(\bm y)~\text{for some}~m = 1, \dots, M,
\end{align*}
the goal of multi-objective optimization is to obtain the \emph{Pareto set}
\[
X^*(\bm f) := \Set{\bm x \in X | f(\bm y) \not \dominates f(\bm x)~\text{for all}~\bm y \in X}
\]
and the \emph{Pareto front}
\[
\bm fX^*(\bm f) := \Set{\bm f(\bm x) \in \R^M | \bm x \in X^*(\bm f)}
\]
which describe the best-compromising solutions and their values of the conflicting objective functions, respectively.

In industrial applications, obtaining the whole Pareto set/front rather than a single solution enables us to compare promising alternatives and to explore new innovative designs, whose concept is variously refered to as innovization~\cite{Deb06}, multi-objective design exploration~\cite{Obayashi05} and design informatics~\cite{Chiba09}.
Quite a few real-world problems involve simulations and/or experiments to evaluate solutions~\cite{Chand15} and lack the mathematical expression of their objective functions and derivatives.
Multi-objective evolutionary algorithms are a tool to solve such problems where the Pareto set/front is approximated by a population, i.e., a finite set of sample points~\cite{Coello07}.

While the available sample size is limited due to expensive simulations and experiments, it is well-known that the dimensionality of the Pareto set/front increases as the number of objectives grows.
Describing high-dimensional Pareto sets/fronts with small samples is one of the key challenges in many-objective optimization today~\cite{Li15}.

A considerable number of real-world applications share an interesting structure: their Pareto sets and/or Pareto fronts are often homeomorphic to an $(M-1)$-dimensional simplex.
See for example~\cite{Rodriguez-Chia02,Dasgupta09,Shoval12,Mastroddi13}.
This observation has been theoretically backed up in some cases~\cite{Kuhn67,Smale73,Shoval12}.
A recent study~\cite{Lovison14} pointed out that for all $m \le M$, each $(m-1)$-dimensional face of such a simplex is the Pareto set of a subproblem optimizing $m$ objective functions of the original problem.

By exploiting this simplex structure, we considers the problem of fitting a hyper-surface to the Pareto front.
In statistics, machine learning and related fields, regression problems have been considered in Euclidean space without boundary or at most with coordinate-wise upper/lower boundaries~\cite{Gelman07,Gelman13}.
Hyper-surface models developed in such spaces are not suitable for Pareto fronts since a simplex has a non-axis-parallel boundary whose skeleton structure is described by its faces.

\begin{figure*}[t]
\centering%
\subfloat[Simplex $\Delta^J$]{\includegraphics[width=0.25\hsize]{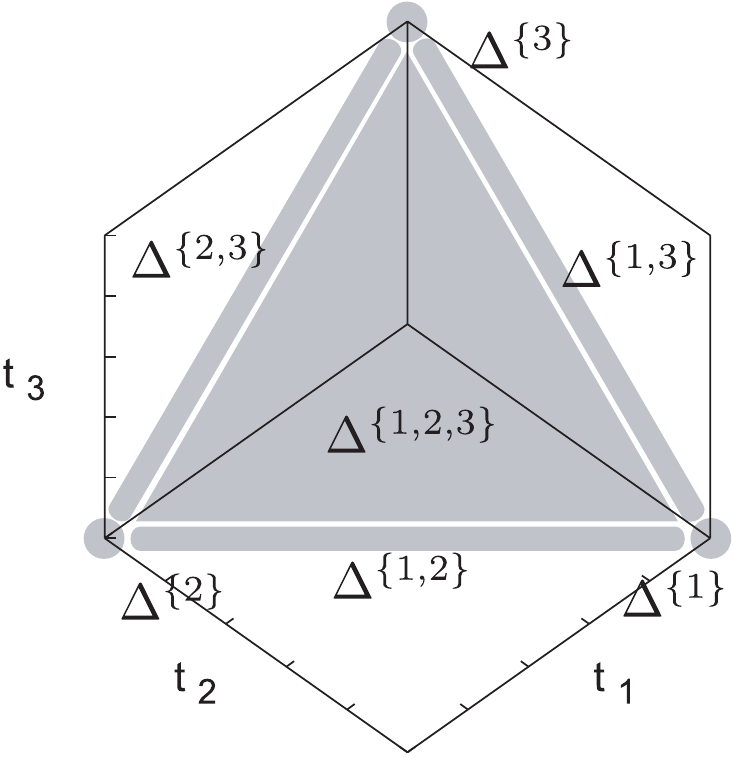}\label{fig:simplex}}
\hspace{10mm}
\subfloat[Pareto set $X^*(\bm f_J)$]{\includegraphics[width=0.25\hsize]{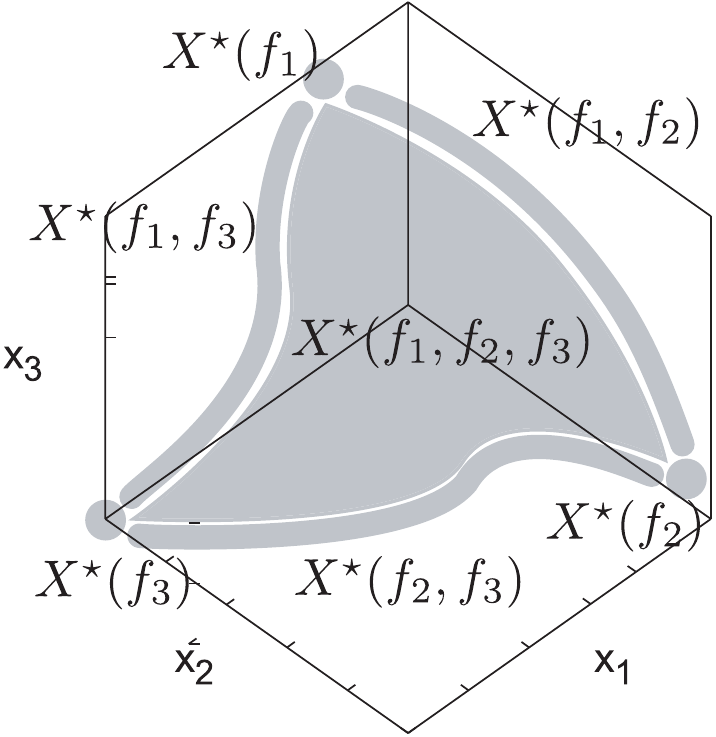}\label{fig:face-relation-X}}
\hspace{10mm}
\subfloat[Pareto front $\bm f X^*(\bm f_J)$]{\includegraphics[width=0.25\hsize]{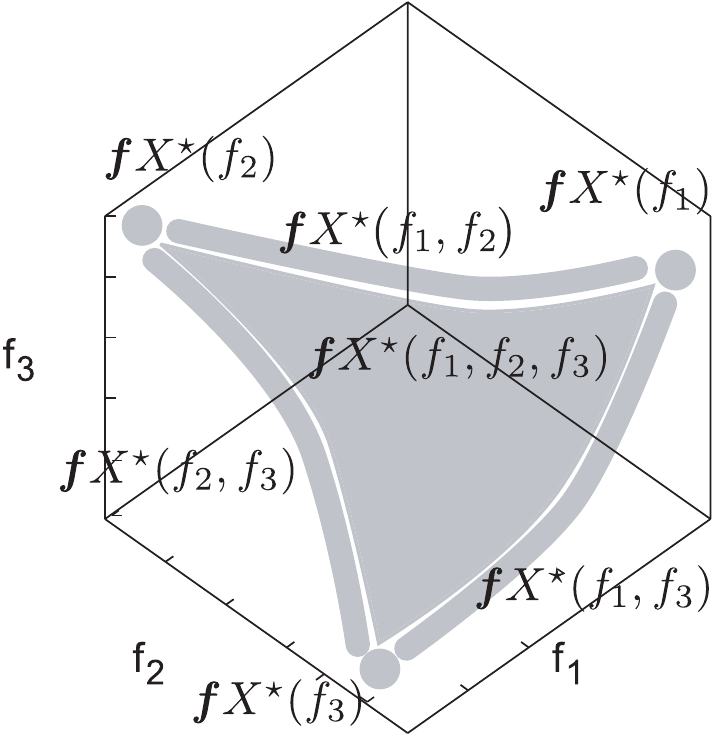}\label{fig:face-relation-F}}
\caption{Skeletons for non-empty $J \subproblemeq \Set{1, 2, 3}$ and simplicial $\bm f = (f_1, f_2, f_3)$.}\label{fig:face-relation}
\end{figure*}

This paper proposes a new model and its fitting algorithm for approximating Pareto fronts with the simplex structure.
Our contibution can be summarized as follows:
\begin{enumerate}
\item
  We define a new class of multi-objective optimization problems called the simplicial problem in which the Pareto set/front have the simplex structure discussed above.
  We propose a B\'ezier simplex model, which is a generalization of a B\'ezier curve~\cite{Farin02}.
\item
  We prove that B\'ezier simplices can approximate the Pareto set/front (as well as the objective map between them) of any simplicial problem with arbitrary accuracy.
\item
  We propose a B\'ezier simplex fitting algorithm.
  Exploiting the simplex structure of the Pareto set/front of a simplicial problem, this algorithm decomposes a B\'ezier simplex into low-dimensional simplices and fits each of them, inductively.
  This approach allows us to reduce the number of parameters to be estimated at a time.
\item
  We evaluate the approximation accuracy of the proposed method with synthetic and real-world optimization problems;
  compared to a conventional response surface model for Pareto fronts, our method exhibits a better boundary approximation while keeping the almost same quality of interior approximation.
  As a result, a five-objective Pareto front is described with only tens of sample points.
\end{enumerate}

\section{Preliminaries}\label{sec:prior-work}
Let us introduce notations for defining simplicial problems and review an existing method of B\'ezier curve fitting.

\subsection{Simplicial Problem}\label{sec:simplicial-problem}
A multi-objective optimization problem is denoted by its objective map $\bm f = (f_1, \dots, f_M): X \to \R^M$.
Let $I:=\set{1,\ldots,M}$ be the index set of objective functions and
\[
\Delta^{M-1} := \Set{(t_1, \dots, t_M) \in \R^M | 0 \le t_m, \sum_{m \in M} t_m = 1}
\]
be the \emph{standard simplex} in $\R^M$.
For each non-empty subset $J \subseteq I$, we call
\[
\Delta^J := \Set{(t_1, \dots, t_M) \in \Delta^{M-1} | t_m = 0\ (m \not\in J)}
\]
the \emph{$J$-face} of $\Delta^{M-1}$ and
\[
\bm f_J := (f_i)_{i \in J}: X \to \R^{\card{J}}
\]
the \emph{$J$-subproblem} of $\bm f$.
For each $0 \le m \le M-1$, we call
\[
\Delta^{(m)} := \bigcup_{J \subseteq I,~\card{J} = m} \Delta^J
\]
the \emph{$m$-skeleton} of $\Delta^{M-1}$.

The problem class we are interested in is as follows:
\begin{definition}[Simplicial problem]\label{def:simplicial}
A problem $\bm f: X \to \R^M$ is \emph{simplicial} if there exists a map $\bm \phi: \Delta^{M-1} \to X$ such that for each non-empty subset $J \subseteq I$, its restriction $\bm \phi|_{\Delta^J}: \Delta^J \to X$ gives homeomorphisms
\begin{align*}
\bm \phi|_{\Delta^J}: \Delta^J& \to X^*(\bm f_J),\\
\bm f \circ \bm \phi|_{\Delta^J}: \Delta^J& \to \bm fX^*(\bm f_J).
\end{align*}
We call such $\bm \phi$ and $\bm f \circ \bm \phi$ a \emph{triangulation} of the Pareto set $X^*(\bm f)$ and the Pareto front $\bm fX^*(\bm f)$, respectively.
For each non-empty subset $J \subseteq I$, we call $X^*(\bm f_J)$ the \emph{$J$-face} of $X^*(\bm f)$ and $\bm fX^*(\bm f_J)$ the \emph{$J$-face} of $\bm fX^*(\bm f)$.
For each $0 \le m \le M-1$, we call
\begin{align*}
X^{*(m)}      &:= \bigcup_{J \subseteq I,~\card{J} = m} X^*(\bm f_J),\\
\bm fX^{*(m)} &:= \bigcup_{J \subseteq I,~\card{J} = m} \bm fX^*(\bm f_J)
\end{align*}
the \emph{$m$-skeleton} of $X^*(\bm f)$ and $\bm fX^*(\bm f)$, respectively.
\end{definition}
By definition, any subproblem of a simplicial problem is again simplicial.
As shown in \Cref{fig:face-relation-X}, the Pareto sets forms a simplex.
The second condition asserts that $\bm f|_{X^*(\bm f)}: X^*(\bm f) \to \R^M$ is a $C^0$-embedding.
This means that the Pareto front of each subproblem is homeomorphic to its Pareto set as shown in \Cref{fig:face-relation-F}.
Therefore, the Pareto set/front of an $M$-objective simplicial problem can be identified with a curved $(M-1)$-simplex.
We can find its $J$-face by solving the $J$-subproblem.

The above structure appears in a broad range of applications.
In operations research, the Pareto set of the facility location problem under the $L^2$-norm is shown to be the convex hull of single-objective optima~\cite{Kuhn67}.
When the optima are in general position, the Pareto set becomes a simplex.
Similar observations are also reported under other norms~\cite{Rodriguez-Chia02}.
In economics, the Pareto set of the pure exchange economy with $M$ players is known to be homeomorphic to an $(M-1)$-dimensional simplex~\cite{Smale73}.
In hydrology, the two-objective Pareto set of a hydrologic cycle model calibration is observed to be a curve.
Its end points are single-objective optima, and the end points correspond to end points of the Pareto front curve~\cite{Dasgupta09}.
In addition, a recent study pointed out that the Pareto set of an $M$-objective convex optimization problem is diffeomorphic to an $(M-1)$-dimensional simplex and that its $(m-1)$-dimensional faces are the Pareto sets of $m$-objective subproblems for all $m \le M$~\cite{Lovison14}.

\subsection{B\'ezier Curve Fitting}\label{sec:Bezier-curve}
Since the Pareto front of any two-objective simplicial problem is a curve with two end points in $\R^3$, the B\'ezier curve would be a suitable model for describing it.

In $\R^M$, the B\'ezier curve of degree $D$ is a parametric curve, i.e., a map $\bm b: [0,1] \to \R^M$ determined by $D+1$ \emph{control points} $\bm p_0, \ldots, \bm p_D \in \R^M$~\cite{Farin02}:
\begin{equation}
\bm b(t) := \sum_{d=0}^D \binom{D}{d} {t}^d (1-t)^{(D-d)} \bm p_d \quad (0\leq t \leq 1),
\end{equation}
where $\binom{D}{d}$ represents the binomial coefficient.
The parameter $t$ moves from $t=0$ to $t=1$, giving a curve $\bm b(t)$ with two end points $\bm b(0) = \bm p_0$ and $\bm b(1) = \bm p_D$.

Given sample points $\bm x_1, \ldots, \bm x_N \in \R^M$, a B\'ezier curve can be fitted by solving the following problem~\cite{Borges02}:
\begin{equation}\label{eq:least-squares}
\begin{split}
\underset{t_n,~\bm p_d}{\text{minimize }}
& \sum_{n=1}^N \norm{\bm b(t_n) - \bm x_n}^2\\
\text{subject to }
&~0 \le t_n \le 1~(n=1,\ldots,N)\\
\end{split}
\end{equation}
where $t_n~(n=1,\ldots,N)$ and $\bm p_d~(d=0,\ldots,D)$ are variables to be optimized.
Notice that $t_n~(n=1,\ldots,N)$ are introduced to calculate residuals for each sample point.
The error function \cref{eq:least-squares} to be minimized represents the sum of squared residuals for sample points.

If one fix all control points $\bm p_d~(d=0,\ldots,D)$, the B\'ezier curve $\bm b(t)$ is determined.
The error function \cref{eq:least-squares} can be now separately optimized by minimizing $\norm{\bm b(t_n) - \bm x_n}^2$ with respect to each $t_n$.
The solution $t_n$ is the foot of a perpendicular line from a sample point $\bm x_n$ to the B\'ezier curve $\bm b([0,1])$, which satisfies
\begin{equation}\label{eq:nonlinear_equation}
\inprod{ \frac{\partial}{\partial t} \Big|_{t = t_n} \bm b(t)}{\bm b(t_n)-\bm x_n} = 0.
\end{equation}
Since \cref{eq:nonlinear_equation} is a nonlinear equation, Newton's method is used to find the solution $t_n$.

If one fix all parameters $t_n~(n=1,\ldots,N)$, the B\'ezier curve $\bm b(t_n)$ becomes a linear function with respect to the control points $\bm p_0, \ldots, \bm p_D$.
The error function \cref{eq:least-squares} can be now optimized by solving linear equations with respect to all $\bm p_d$.

\Cref{alg:borges} shows Borges and Pastva's method, which alternately adjusts parameters and control points.
This algorithm is also used to adjust some of the control points with remaining points fixed~\cite{Shao96}.

\begin{algorithm}[t]
\caption{B\'ezier curve fitting~\cite{Borges02}}
\label{alg:borges}
\begin{algorithmic}[1]
\STATE (Initialize)~Set $i \gets 1$ and initial control points $\bm p_d^{(i)}~(d=0, \ldots, D)$.
\WHILE{not converged}
    \STATE (Update parameters)~Fix control points $\bm p_d^{(i)}~(d=0, \ldots, D)$ and solve \cref{eq:nonlinear_equation} for each $n=1,\ldots,N$ using Newton's method. Then set the solutions as $t_n^{(i+1)}$.
    \STATE (Update control points)~Solve \cref{eq:least-squares} with respect to the control points and set the solutions as $\bm p_d^{(i+1)}~(d=0,\ldots,D)$.
    \STATE $i \gets i+1$.
\ENDWHILE
\RETURN $\bm p_d^{(i)}~(d=0,\ldots,D)$
\end{algorithmic}
\end{algorithm}

\section{B\'ezier Simplex Fitting}
\begin{figure}[t]
\centering%
\includegraphics[width=0.75\hsize]{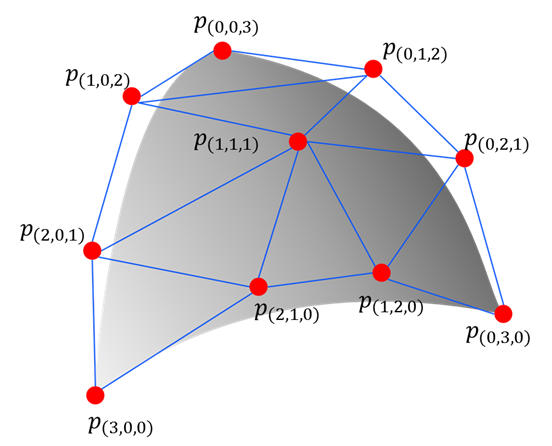}
\caption{A B\'ezier simplex for $M=3$, $D=3$.}\label{fig:Bezier-simplex}
\end{figure}
To describe the Pareto front of an arbitrary-objective simplicial problem, however, B\'ezier curve fitting is not enough, and we need to generalize it to the B\'ezier simplex.
We propose a method of fitting a B\'ezier simplex to the Pareto front of a simplicial problem.

\subsection{B\'ezier Simplex}\label{sec:Bezier-simplex}
Let $\N$ be the set of nonnegative integers and
\[
\N_D^M := \Set{(d_1,\dots,d_M)\in\N^M | \sum_{m=1}^M d_m = D}.
\]
For $\bm t :=(t_1,\ldots,t_M) \in \R^M$ and $\bm d := (d_1,\dots,d_M) \in \N^M$, we denote by $\bm t^{\bm d}$ a monomial $t^{d_1}_1 t^{d_2}_2 \cdots t^{d_M}_M$.
As shown in \Cref{fig:Bezier-simplex}, the B\'ezier simplex of degree $D$ in $\R^M$ is a map $\Delta^{M-1} \to \R^M$ determined by control points $\bm p_{\bm d} \in \R^M$ $(\bm d \in \N_D^M)$:
\begin{align}\label{eq:Bezier-simplex}
\bm b(\bm t) := \sum_{\bm d\in\N_D^M} \binom{D}{\bm d} \bm t^{\bm d} \bm p_{\bm d},
\end{align}
where $\binom{D}{\bm d}$ represents a polynomial coefficient
\[
  \binom{D}{\bm d} := \frac{D!}{d_1! d_2! \cdots d_M!}.
\]

\subsection{Approximation Theorem}
Is the B\'ezier simplex a suitable model for describing the Pareto front of a simplicial problem?
Let us check that for any simplicial problem, B\'ezier simplices can approximate the Pareto front (as well as the Pareto set and the objective map between them) with arbitrary accuracy.

We begin with a more general proposition that any continuous map on a simplex can be approximated by some B\'ezier simplex as a map:
\begin{theorem}\label{thm:approximation}
Let $\bm \phi: \Delta^{M-1} \to \R^M$ be a continuous map.
There exists an infinite sequence of B\'ezier simplices $\bm b^{(i)}: \Delta^{M-1} \to \R^M$ such that
\[
\lim_{i \to \infty} \sup_{\bm t \in \Delta^{M-1}} |\bm \phi(\bm t)-\bm b^{(i)}(\bm t)| = 0.
\]
\end{theorem}
The proof of \Cref{thm:approximation} is shown in Appendix A.
Recall \Cref{def:simplicial} that a simplicial problem $\f:X\to\R^M$ admits a triangulation of the Pareto set $\bm \phi : \Delta^{M-1} \to X^*(\f)$, which induces a triangulation of the Pareto front $\f \circ \bm \phi : \Delta^{M-1} \to \f X^*(\f)$.
These triangulations can be $\bm \phi$ in \Cref{thm:approximation}.
Addtionally, the restricted objective map $\f:X^*(\f)\to\f X^*(\f)$ has the graph
\[
G^*(\f) = \Set{(\x,\f(x))\in\R^L\times\R^M | \x\in X^*(\f)},
\]
and $\bm \phi : \Delta^{M-1} \to X^*(\f)$ also induces a triangulation of the graph
\[
\bm \phi \times (\f \circ \bm \phi): \Delta^{M-1} \to G^*(\f)
\]
since $\f:X^*(\f)\to\f X^*(\f)$ is continuous by definition and any continuous map induces a homeomorphism from the domain to the graph.
We thus get the desired result.
\begin{corollary}\label{cor:approximation}
Let $X^*$ be the Pareto set, the Pareto front or the graph of the objective map restricted to the Pareto set of a simplicial problem.
There exists an infinite sequence of B\'ezier simplices $\bm b^{(i)}: \Delta^{M-1} \to \R^M$ such that
\[
\lim_{i \to \infty} d_{\mathrm H}(X^*, B^{(i)}) = 0
\]
where $d_{\mathrm H}$ is the Hausdorff distance and $B^{(i)}$ are images of B\'ezier simplices: $B^{(i)} := \bm b^{(i)}(\Delta^{M-1})$.
\end{corollary}

\subsection{All-at-Once Fitting}\label{sec:all-at-once-fitting}
Let us consider algorithms to realize such a sequence of B\'ezier simplices.
First, we develop a straightforward generalization of the B\'ezier curve fitting method.
We call this method the \emph{all-at-once fitting}.

Given sample points $\bm x_1, \ldots, \bm x_N \in \R^M$, a B\'ezier simplex can be fitted by solving the following problem, which is a multi-dimensional analogue of the problem \cref{eq:least-squares}:
\begin{equation}\label{eq:least-squares_bezier-simplex}
\begin{split}
  \underset{\bm{t}_n,~\bm p_{\bm d}}{\text{minimize }} &\sum_{n=1}^N \norm{\bm b(\bm t_n) - \bm x_n}^2\\
  \text{subject to }
  & \bm t_n \in \Delta^{M-1}~(n=1,\ldots, N)
\end{split}
\end{equation}
where $\bm{t}_n=(t_{n1},\ldots,t_{nM})~(n=1,\ldots,N)$ and $\bm p_{\bm d}~(\bm d \in \N^M_D)$ are variables to be optimized.

As is the case of B\'ezier curve fitting, if one fix all control points $\bm p_{\bm d}~(\bm d \in \N^M_D)$, the B\'ezier simplex $\bm b(\bm t)$ is determined.
The error function \Cref{eq:least-squares_bezier-simplex} can be now separately optimized by minimizing $\norm{\bm b(\bm t_n) - \bm x_n}^2$ with respect to each $\bm t_n$.
The solution $\bm t_n$ is the foot of a perpendicular line from a sample point $\bm x_n$ to the B\'ezier simplex $\bm b(\Delta^{M-1})$, which satisfies
\begin{equation}\label{eq:nonlinear_equations}
\inprod{\frac{\partial}{\partial t_m}\Big|_{\bm t = \bm t_n} \bm b(\bm t)}{\bm b(\bm t_n) - \bm x_n} = 0~~(m=1,\ldots,M).
\end{equation}
Since \Cref{eq:nonlinear_equations} is a system of nonlinear equations, Newton's method is used to find the solution $\bm t_n$.

If one fix all parameters $\bm t_n~(n=1,\ldots,N)$, the B\'ezier simplex $\bm b(\bm t_n)$ is a linear function with respect to the control points $\bm p_{\bm d}~(\bm d \in \N^M_D)$.
The error function \Cref{eq:least-squares_bezier-simplex} can be now optimized by solving linear equations with respect to all $\bm p_{\bm d}~(\bm d \in \N^M_D)$.

As well as Borges and Pastva's method, the all-at-once fitting alternatively adjusts parameters and control points.
We describe the all-at-once fitting in \Cref{alg:allatonce}.
This algorithm is also used to adjust some of the control points with remaining control points fixed.

\begin{algorithm}[t]
\caption{B\'ezier simplex fitting (all-at-once)}
\label{alg:allatonce}
\begin{algorithmic}[1]
\STATE (Initialize)~Set $i \gets 1$ and initial control points $\bm p_{\bm d}^{(i)}~(\bm d \in \N_D^M)$.
\WHILE{not converged}
    \STATE (Update parameters) Fix control points $\bm p_{\bm d}^{(i)}~(\bm d \in \N_D^M)$ and solve \Cref{eq:nonlinear_equations} for each $n = 1, \ldots, N$ using Newton's method. Then set the solutions as $\bm t_n^{(i+1)}$.
    \STATE (Update control points)~Solve \Cref{eq:least-squares_bezier-simplex} with respect to the control points and set the solutions $\bm p_{\bm d}^{(i+1)}~(\bm d \in \N_D^M)$.
    \STATE $i \gets i+1$.
\ENDWHILE
\RETURN $\bm p_{\bm d}^{(i)}~(\bm d \in \N_D^M)$
\end{algorithmic}
\end{algorithm}

While a B\'ezier simplex has enough flexibility as described in \Cref{thm:approximation}, the required sample size for the all-at-once fitting grows quickly with respect to $M$.
Let us consider the case of fitting a B\'ezier simplex of degree $D$ to the Pareto front of an $M$-objective problem.
In this case, the number of control points to be estimated is
\[
\card{\N_D^M} = \binom{D+M-1}{D} = \frac{(D+M-1)!}{D!(M-1)!} = O(M^D).
\]
Practically, the degree of a B\'ezier simplex to be fitted is often set as $D=3$; nevertheless the number of control points to be estimated at a time becomes unreasonable for large $M$.

\subsection{Inductive Skeleton Fitting}\label{sec:inductive-skeleton-fitting}
To reduce the required sample size, we consider decomposing a B\'ezier simplex into subsimlices and fitting each subsimplex one by one from low dimension to high dimension.
This approach allows us to reduce the number of control points to be estimated at a time.
We call this method the \emph{inductive skeleton fitting}.

Let $I := \Set{1, \ldots, M}$ and for each non-empty subset $J \subseteq I$, we define
\[
\N_D^J := \Set{(d_1, \dots, d_M) \in \N_D^M | d_m = 0~(m \not\in J)}.
\]
The $J$-face of an $(M-1)$-B\'ezier simplex of degree $D$ is a map $\Delta^J \to \R^M$ determined by control points $\bm p_{\bm d}\in \R^M$ for all $\bm d\in\N_D^J$:
\begin{align}\label{eq:Bezier-simplex-face}
\bm b^J(\bm t) := \sum_{\bm d\in\N_D^J} \binom{D}{\bm d} \bm t^{\bm d} \bm p_{\bm d}.
\end{align}

For arbitrary parameters $\bm t$ satisfying
\[
\bm t \in \Delta^J \subseteq \Delta^I,
\]
all entries that are not in $J$ are 0.
It then holds that
\[
\bm b^I(\bm t) = \sum_{\bm d\in \N_D^I} \binom{D}{\bm d} \bm t^{\bm d} \bm p_{\bm d} = \sum_{\bm d\in \N_D^J} \binom{D}{\bm d} \bm t^{\bm d} \bm p_{\bm d} = \bm b^J(\bm t).
\]
This means that, for each non-empty subset $J \subset I$, the $J$-face of the $(M-1)$-B\'ezier simplex of degree $D$ is the $(\card{J}-1)$-B\'ezier simplex of the same degree and it is determined by the control points $\bm p_{\bm d}$ of the $(M-1)$-B\'ezier simplex satisfying $\bm d \in \N_D^J$.
To exploiting this structure, the inductive skeleton fitting decomposes a B\'ezier simplex into subsimplices and fits a subsimplex $\bm b^J(\bm t)$ to $X_J$ for each non-empty subset $J \subseteq I$ in ascending order of cardinality of $J$.

When a problem $\bm f=(f_1,\ldots,f_M): X \to \R^M$ is simplicial, we can provide a set of subsamples for running the inductive skeleton fitting.
Remember that the Pareto front $\bm fX^*(\bm f)$ has the same skeleton as the $(M-1)$-simplex: $\bm fX^*(\bm f_J)\subseteq \bm fX^*(\bm f_I)$ for all $\emptyset \ne J \subseteq I$.
Thus given a sample $X$ of $\bm fX^*(\bm f)$, we can decompose it into subsamples
\[
X_J :=\Set{\bm x \in X | \bm f_J(y)\not \dominates \bm f_J(x)~\text{for all}~\bm y \in X},
\]
each of which represents a sample of the $J$-face $\bm fX^*(\bm f_J)~(\emptyset \ne J \subseteq I)$.
\Cref{alg:skeleton} summarizes the inductive skeleton fitting.

\begin{algorithm}[t]
\caption{B\'ezier simplex fitting (inductive skeleton)}
\label{alg:skeleton}
\begin{algorithmic}[1]
\FOR{$m = 1, \dots, \min \Set{D,M}$}
	\FOR{$J \subseteq I$ such that $\card{J} = m$}
		\STATE Fix control points $\bm p_{\bm d}$ $(\bm d \in \bigcup_{\emptyset \ne K \subset J} \N_D^K)$ as estimated in previous steps.
		\STATE Adjust remaining control points $\bm p_{\bm d}$ $(\bm d \in \N_D^J \setminus (\bigcup_{\emptyset \ne K \subset J}\N_D^K))$ by \Cref{alg:allatonce} with sample $X_J$.
	\ENDFOR
\ENDFOR
\RETURN $\bm p_{\bm d}~(\bm d \in \N_D^M)$
\end{algorithmic}
\end{algorithm}

Unlike the all-at-once fitting, the inductive skeleton fitting allows us to reduce the number of control points to be estimated at a time.
Let us consider the case of fitting a subsimplex $\bm b^J(\bm t)$ with $\card{J} = m$.
As we described before, the subsimplex $\bm b^J(\bm t)$ has $\card{\N^m_D} = \binom{D + m - 1}{D} = \frac{(D + m - 1)!}{D! (m-1)!}$ control points.
In practice, it is sufficient to set $D$ as a small value compared to $M$.
In such a case, the inductive skeleton fitting estimates at most $D$-objective solutions, then we only have to adjust at most $\max_{m = 1, \ldots, D}\binom{D + m - 1}{D}$ control points for each step.
Notice that this number does not depend on $M$ but on $D$.
Therefore, in case of the inductive skeleton fitting, the number of control points to be estimated at a time is much smaller than  $\binom{D+M-1}{D}$ for large $M$.

\section{Numerical Experiments}\label{sec:experiments}
The small sample behavior of the proposed method is examined using Pareto front samples of varying size.\footnote{The source code is available from \url{https://github.com/rafcc/stratification-learning}.}

\subsection{Data Sets}
To investigate the effect of the Pareto front shape and dimensionality, we employed six synthetic problems with known Pareto fronts, all of which are simplicial problems.
Schaffer, ConstrEx and Osyczka2 are two-objective problems.
Their Pareto fronts are a curved line that can be triangulated into two vertices and one edge.
3-MED and Viennet2 are three-objective problems.
Their Pareto fronts are a curved triangle that can be triangulated into three vertices, three edges and one face.
5-MED is a five-objective problem.
Its Pareto front is a curved pentachoron that can be triangulated into five vertices, ten edges, ten faces, five three-dimensional faces and one four-dimensional face.
We generated Pareto front samples of 3-MED and 5-MED by AWA(objective) using default hyper-parameters~\cite{Hamada10}.
Pareto front samples of the other problems were taken from jMetal~5.2~\cite{Nebro15}.

To assess the practicality of the proposed method, we also used one real-world problem called S3TD.
This is a four-objective problem\footnote{The original problem is five-objective, but we dropped one objective since there are two objectives that are highly correlated.} of designing a silent super-sonic aircraft.
The Pareto front has not been exactly known, and we only have an inaccurate sample of 58 points obtained in the prior study~\cite{Chiba09}.
Its simpliciality is also unknown.
If it is assumed to be simplicial, then the Pareto front is a curved tetrahedron that can be triangulated into four vertices, six edges, four faces, one three-dimensional face.
The problem definitions of the synthetic and real-world problems are described in Appendix B.

For each problem, the Pareto front sample is split into a training set and a validation set.
The training set is further decomposed into subsamples for the Pareto fronts of subproblems: each $m$-objective subproblem has a subsample consisting of $N_m$ points.
Experiments were conducted on all combinations of the following subsample sizes:\footnote{We fix $N_4 = N_5 = 0$ since all methods used in our experiments do not require these subsamples to fit the Pareto front.}
\begin{align*}
N_1&=1,\\
N_2&=2,\dots,10,\\
N_3&=1,\dots,10\quad (N_4=N_5=0).
\end{align*}
An $M$-objective problem has $M$ single-objective problems, $M (M-1) / 2$ two-objective problems and $M (M-1) (M-2) / 6$ three-objective problems.
The total sample size of the training set is
\[
N = MN_1 + M(M-1)N_2/2 + M(M-1)(M-2)N_3/6.
\]
The details of making Pareto front subsamples are described in Appendix C.

\subsection{Methods}
We compared the following surface-fitting methods:
\begin{itemize}
  \item the inductive skeleton fitting~(\Cref{alg:skeleton});
  \item the all-at-once fitting~(\Cref{alg:allatonce});
  \item the response surface method~\cite{Goel07}.
\end{itemize}
The inductive skeleton fitting used each training subsample for each subproblem.
The all-at-once fitting and the response surface method used all training subsamples as a whole.
We set the initial control points $\bm p_{\bm d}^0$ $(\bm d \in \N_D^J, \ J \subset I, \ \card{J} = 1)$, which are the vertices of the B\'ezier simplex, to be the single-objective optima, and the rest of control points were set to be the simplex grid spanned by them.
Newton's method used in the inductive skeleton fitting and the all-at-once fitting employed the first and second analytical derivatives, and was terminated when the number of iterations reached 100 or when the following condition is satisfied:
\[
\sqrt{\sum_{m = 1}^M \inprod{\frac{\partial}{\partial t_m}\Big|_{\bm t = \bm t_n} \bm b(\bm t)}{\bm b(\bm t_n) - \bm x_n}^2} \le 10^{-5}.
\]
The fitting algorithm was terminated when the number of iterations reached 100 or when the following condition is satisfied:
\[
 \sqrt{\mathrm{SSR}_{i+1}}/N - \sqrt{\mathrm{SSR}_{i}}/N \le 10^{-5},
\]
where $\mathrm{SSR}_i$ is the value of the loss function \Cref{eq:least-squares_bezier-simplex} at the $i$-th iteration.

According to a prior study~\cite{Goel07}, the response surface method treated multi-dimensional objective values as a function from the first $M-1$ objective values to the last objective value by using multi-polynomial regression with cubic polynomials.
We removed cross terms of degree three so that the number of explanatory variables did not exceed the sample size.

\begin{table*}[t]
\centering
\tiny
\caption{GD and IGD (avg.\ $\pm$ s.d.\ over 20 trials) with subsample size $(1,3)$ for Schaffer, ConstrEx and Osyczka2; $(1,2,1)$ for 3-MED, Viennet2 and 5-MED; $(1,5,5)$ for S3TD. The best scores are shown in bold. (${}^{**}: p<0.05, {}^*: p< 0.1$)}\label{tbl:results}
{\tabcolsep = 1.4mm
\begin{tabular}{ccllcllllcc}
\toprule
\textbf{Problem}&\multicolumn{3}{c}{\textbf{Inductive skeleton}}&\multicolumn{3}{c}{\textbf{All-at-once}}&\multicolumn{2}{c}{\textbf{Response surface}}\\
&\multicolumn{1}{c}{Iteration}&\multicolumn{1}{c}{GD}&\multicolumn{1}{c}{IGD}&\multicolumn{1}{c}{Iteration}&\multicolumn{1}{c}{GD}&\multicolumn{1}{c}{IGD}&\multicolumn{1}{c}{GD}&\multicolumn{1}{c}{IGD}\\
\midrule
Schaffer	&3.00 	$\pm$	0.00 	&\textbf{2.50e-10	$\pm$	3.76e-11}	&\textbf{2.49e-02	$\pm$	2.49e-12}	&3.00 	$\pm$	0.00 	&2.51e-10	$\pm$	3.71e-11	&\textbf{2.49e-02	$\pm$	2.52e-12}	&5.26e-01	$\pm$	1.78e+00	&8.13e-02	$\pm$	5.54e-02\\
ConstrEx	&3.00 	$\pm$	0.00 	&2.33e-02	$\pm$	2.32e-02	&4.14e-02	$\pm$	1.24e-02	&3.00 	$\pm$	0.00 	&2.34e-02	$\pm$	2.32e-02	&4.13e-02	$\pm$	1.24e-02	&\textbf{1.47e-02	$\pm$	9.15e-03}$^{**}$	&\textbf{2.48e-02	$\pm$	9.24e-03}$^{**}$\\
Osyczka2	&3.45 	$\pm$	1.12 	&\textbf{6.02e-02	$\pm$	2.72e-02}	&\textbf{8.27e-02	$\pm$	3.49e-02}	&3.20 	$\pm$	0.51 	&6.08e-02	$\pm$	2.86e-02	&8.33e-02	$\pm$	3.70e-02	&2.79e-01	$\pm$	3.34e-01	&1.01e-01	$\pm$	4.22e-02\\
\midrule
3-MED	&2.45 	$\pm$	0.50 	&\textbf{3.99e-01	$\pm$	9.77e-01}$^{** }$	&\textbf{6.16e-02	$\pm$	2.51e-02}$^{**}$	&3.00 	$\pm$	0.00 	&1.02e+00	$\pm$	2.84e+00	&1.05e-01	$\pm$	6.43e-02	&1.39e+00	$\pm$	2.38e-01	&6.75e-02	$\pm$	5.37e-03\\
Viennet2	&2.65 	$\pm$	0.48 	&\textbf{2.51e+00	$\pm$	3.73e+00}$^{** }$	&\textbf{6.47e-02	$\pm$	9.91e-03}	&3.68 	$\pm$	0.98 	&2.24e+09	$\pm$	9.40e+09	&2.42e-01	$\pm$	1.02e-01	&3.10e+00	$\pm$	2.58e+00	&6.89e-02	$\pm$	1.73e-02\\
5-MED	&2.60 	$\pm$	0.49 	&\textbf{2.55e-01	$\pm$	3.13e-01}$^{**}$	&\textbf{7.94e-02	$\pm$	1.40e-02}$^{**}$	&3.00 	$\pm$	0.00 	&4.19e+00	$\pm$	1.24e+01	&1.66e-01	$\pm$	4.01e-02	&4.89e+00	$\pm$	7.52e-01	&1.02e-01	$\pm$	4.89e-03\\
\midrule
S3TD	&3.25 	$\pm$	1.58 	&\textbf{2.04e-01	$\pm$	4.04e-02}$^{*}$	&1.37e-01	$\pm$	3.15e-02	&36.20 	$\pm$	33.81 	&5.58e-01	$\pm$	5.99e-01	&6.31e-01	$\pm$	6.36e-01	&7.17e-01	$\pm$	2.14e-01	&\textbf{9.88e-02	$\pm$	6.67e-03}$^{**}$\\
\bottomrule
\end{tabular}
}
\end{table*}

\subsection{Performance Measures}\label{sec:evaluation}
To evaluate how accurately an estimated hyper-surface approximates the Pareto front, we used the generational distance (GD)~\cite{Veldhuizen99} and the inverted generational distance (IGD)~\cite{Zitzler03a}:
\begin{align*}
\mathrm{GD}(X, Y) &:=\frac{1}{\card{X}} \sum_{\bm x \in X} \min_{\bm y \in Y} \norm{\bm x-\bm y},\\
\mathrm{IGD}(X, Y)&:=\frac{1}{\card{Y}} \sum_{\bm y \in Y} \min_{\bm x \in X} \norm{\bm x-\bm y},
\end{align*}
where $X$ is a finite set of points sampled from an estimated hyper-surface and $Y$ is a validation set.

\begin{figure}[t]
\centering%
\subfloat[GD assesses false positive.]{\includegraphics[width=0.5\hsize]{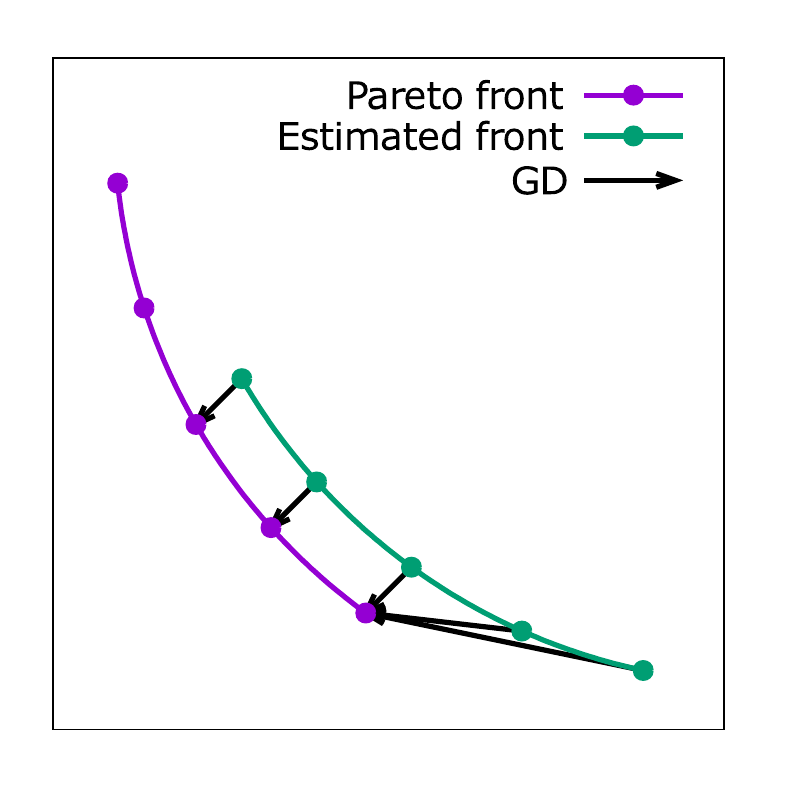}\label{fig:GD}}
\subfloat[IGD assesses false negative.]{\includegraphics[width=0.5\hsize]{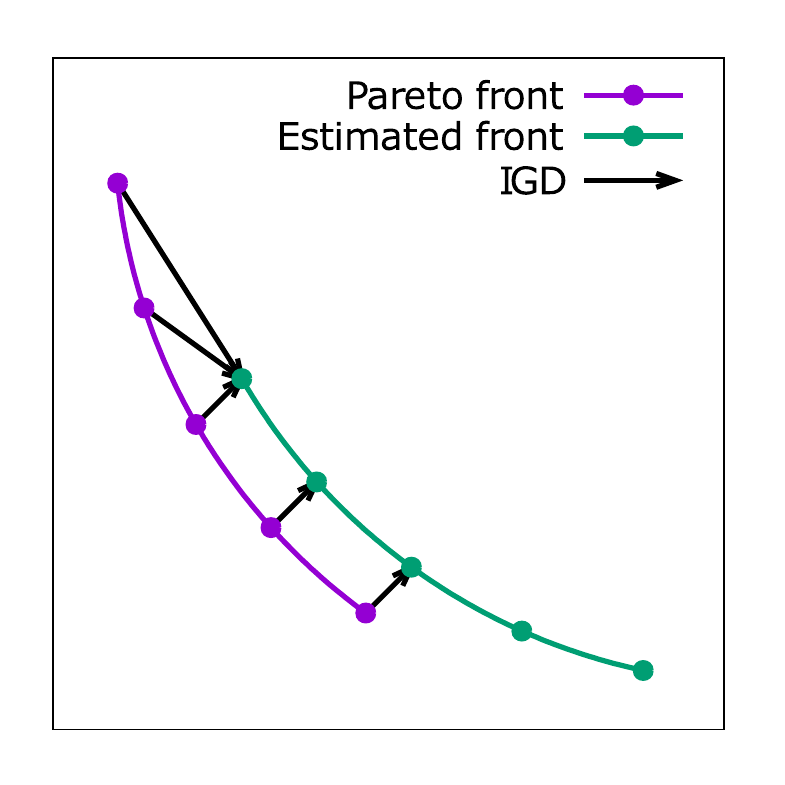}\label{fig:IGD}}
\caption{GD and IGD.\label{fig:GD-IGD}}
\end{figure}

\Cref{fig:GD-IGD} depicts what GD and IGD assess.
These values can be viewed as the avarage length of arrows in each plot.
\Cref{fig:GD} implies that GD becomes high when the estimated front has a false positive area which is far from the Pareto front.
Conversely, \Cref{fig:IGD} tells that IGD becomes high when the Pareto front has a false negative area which is not covered by the estimated front.
Thus, we can say that the estimated hyper-surface is close to the Pareto front if and only if both GD and IGD are small.

For the all-at-once fitting and the inductive skeleton fitting, we sampled the estimated B\'ezier simplex as follows:
\[
X := \Set{\bm b(\bm t) | \bm t \in \Delta^{M-1},\ t_m \in \Set{0, \frac{1}{20}, \frac{2}{20}, \ldots, 1}}.
\]
For the response surface method, we sampled the estimated response surface as follows:
\begin{align*}
X := \{ & (x_1,\ldots,x_{M-1},r(x_1,\ldots,x_{M-1}))\\
        & | x_1,\ldots,x_{M-1} \in \{0,\frac{1}{20},\frac{2}{20},\ldots,1\} \},
\end{align*}
where $r(x_1,\ldots,x_{M-1})$ is an estimated polynomial function.
We repeated experiments 20 times\footnote{We set the number of trials on the basis of the power analysis. According to our preliminary experiments, the effect size of GD/IGD values was approximately $d=0.6$. In this case, the sample size $n=20$ is required to assume the power $>0.8$ with significance level $p=0.05$.} with different training sets and computed the average and the standard deviations of their GDs and IGDs.

\subsection{Results}
For each problem and method, the average and the standard deviation of the GD and IGD are shown in \Cref{tbl:results}.
In \Cref{tbl:results}, we highlighted the best score of GD and IGD out of all methods for each problem and added the results of Mann-Whitney's one-tail U-test to check the best score was smaller than that of the other methods significantly.
In case of conducting multiple tests, we corrected p-value by Holm's method.

For B\'ezier simplex fitting methods, the number of iterations until termination is also shown.
The table shows that the fitting was finished in approximately three iterations in all cases except for the all-at-once fitting on S3TD.

\paragraph{Inductive skeleton vs.\ all-at-once}
For all the two-objective problems, both methods obtained almost identical GD and IGD values.
For three- and five-objective problems, the inductive skeleton fitting significantly outperformed the all-at-once fitting in both GD and IGD\@.
Especially in Viennet2, the all-at-once fitting exhibited unstable behavior in GD while the inductive skeleton fitting was stable.

\paragraph{Inductive skeleton vs.\ response surface}
The inductive skeleton fitting achieved better IGDs in Schaffer, Osyczka2, 3-MED, and 5-MED but worse values in ConstrEx and Viennet2.
In terms of the average GD, the inductive skeleton fitting was better in all the problems except ConstrEx.

\section{Discussion}\label{sec:discussion}
In this section, we discuss approximation accuracy, required sample size, practicality for real-world data and objective map approximation of our method.
\begin{figure*}[t]
\centering%
\subfloat[Inductive skeleton]{\includegraphics[width=0.27\hsize]{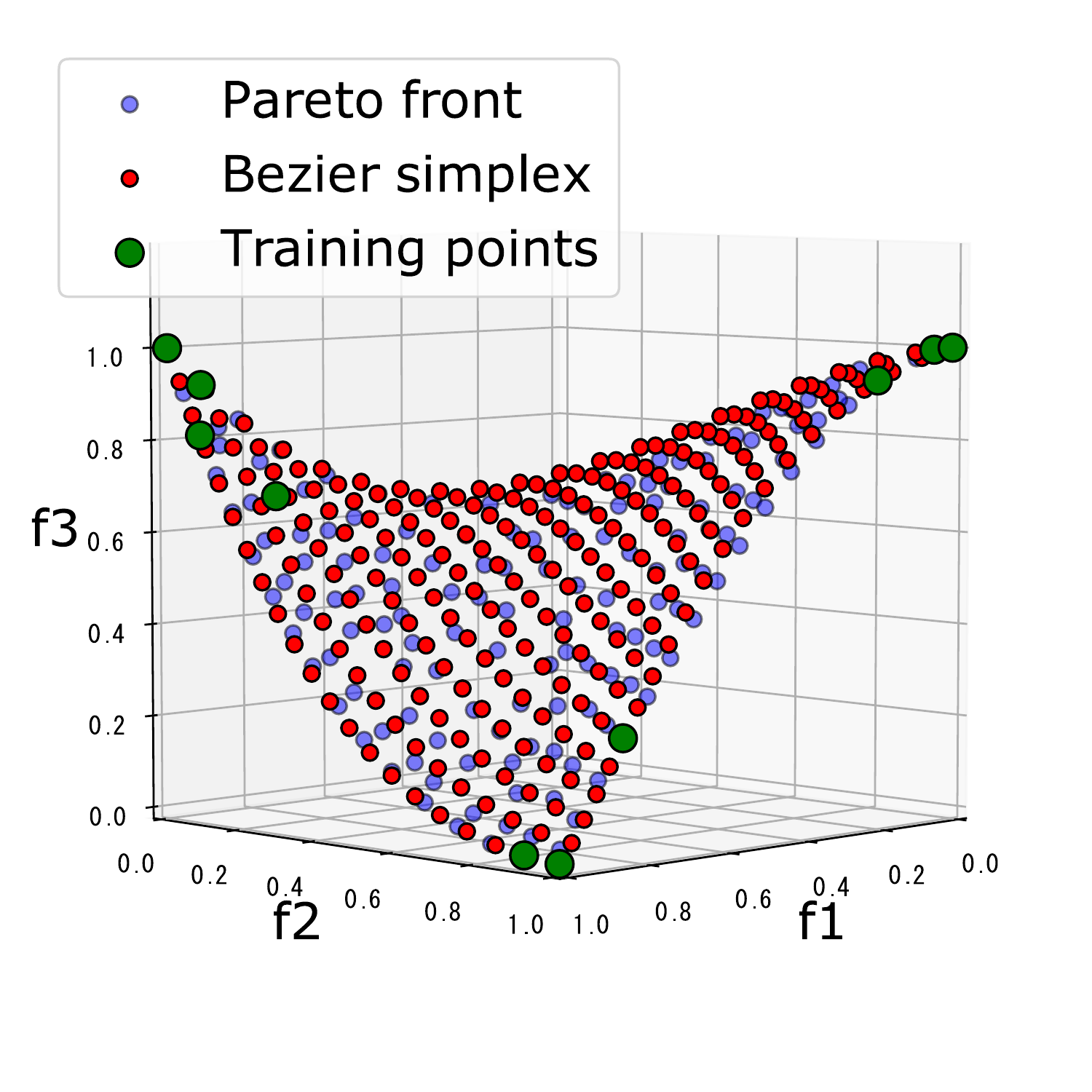}}
\hspace{10mm}
\subfloat[All-at-once]{\includegraphics[width=0.27\hsize]{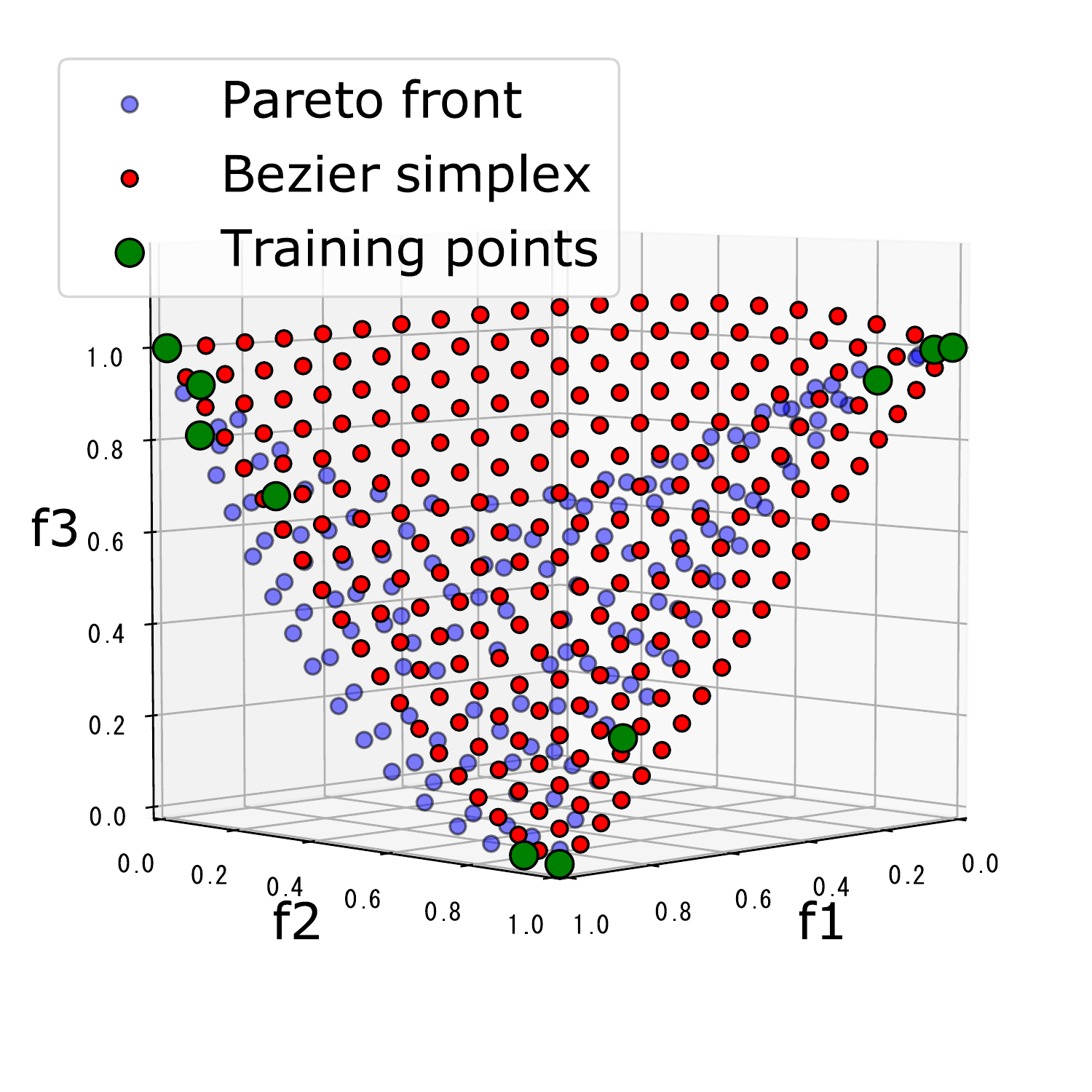}}
\hspace{10mm}
\subfloat[Response surface]{\includegraphics[width=0.27\hsize]{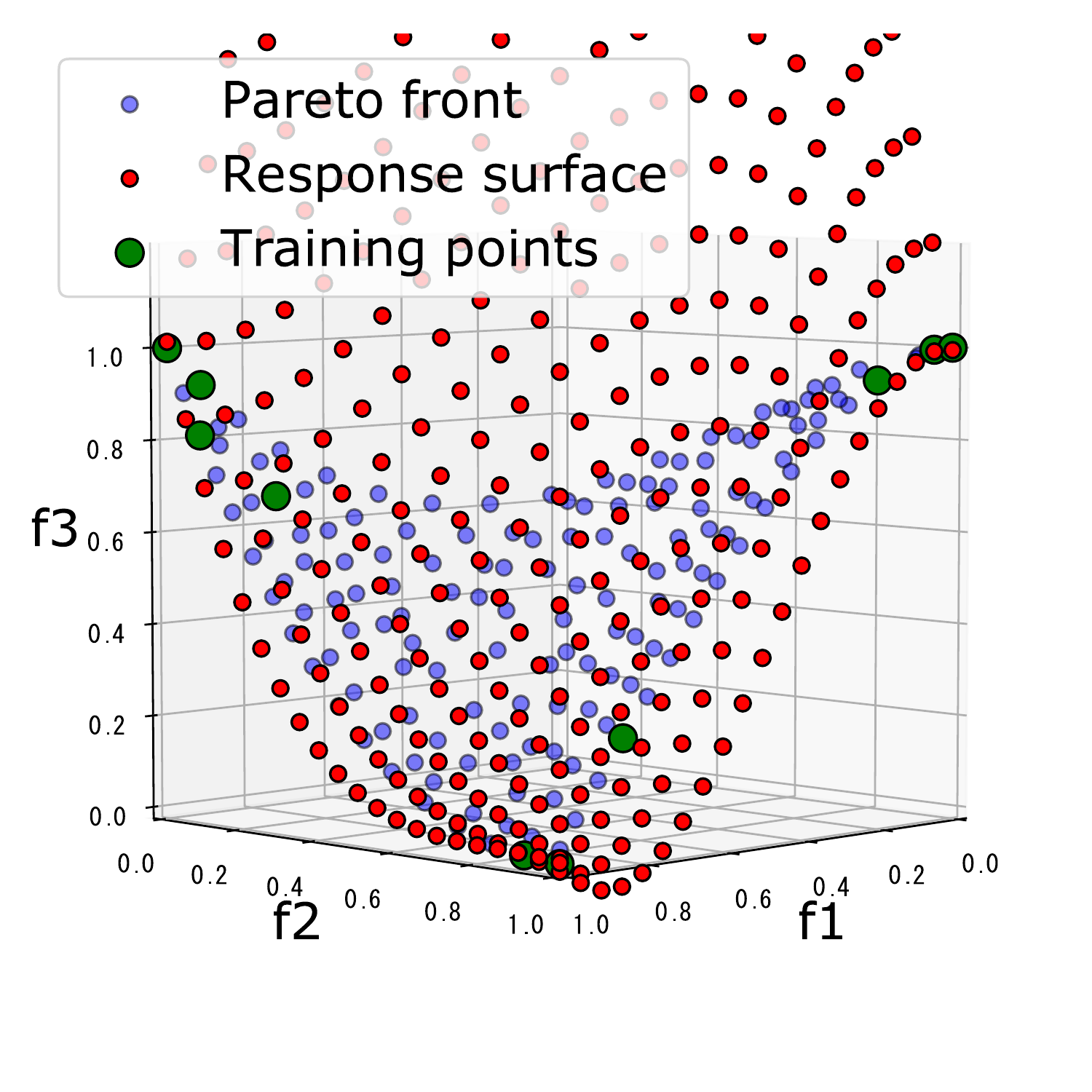}}
\caption{B\'ezier triangles for 3-MED with sample size $(1, 2, 1)$.}\label{fig:3-MED}
\end{figure*}
\begin{figure*}[t]
\centering%
\includegraphics[width=0.49\hsize]{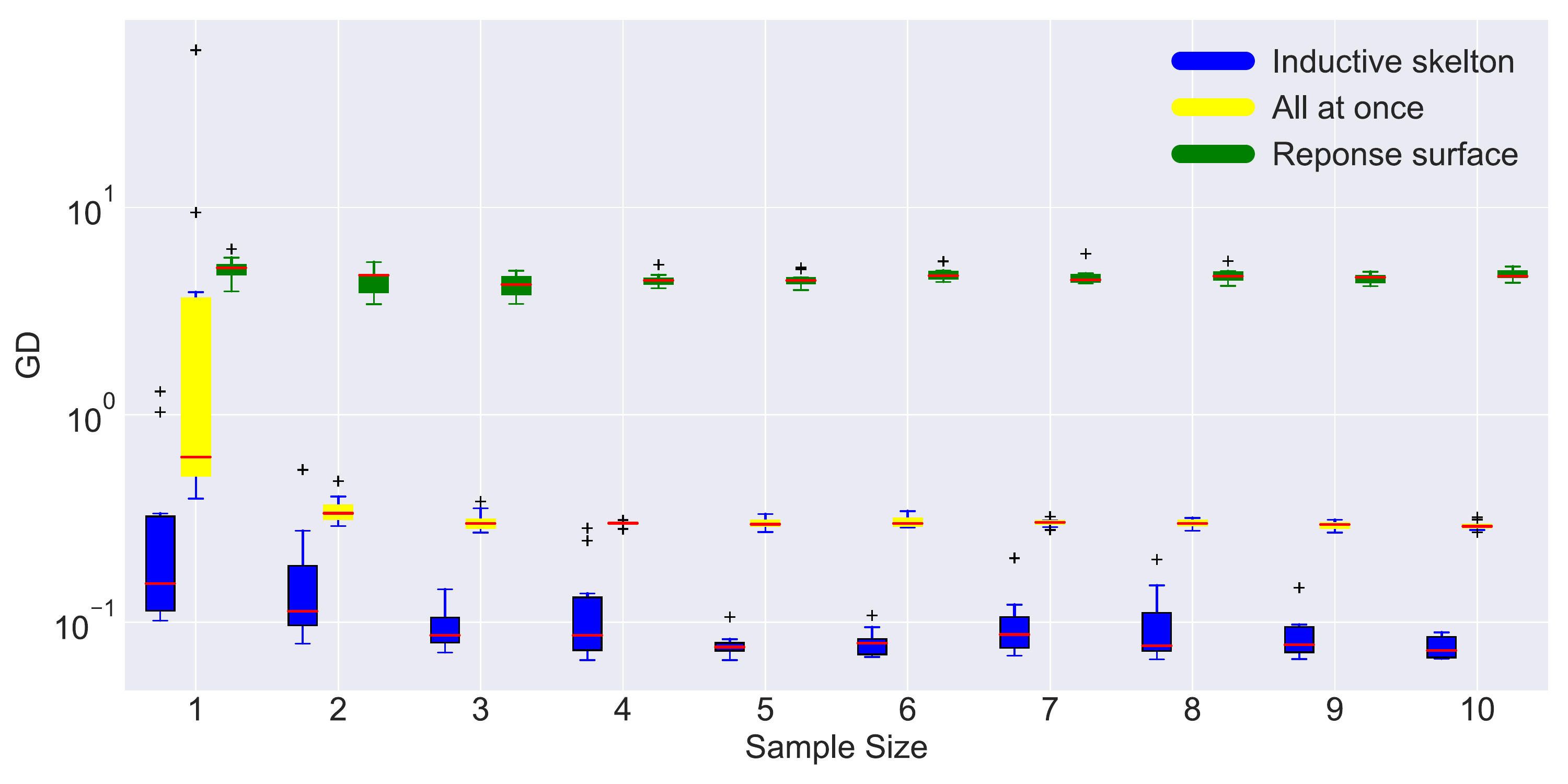}
\includegraphics[width=0.49\hsize]{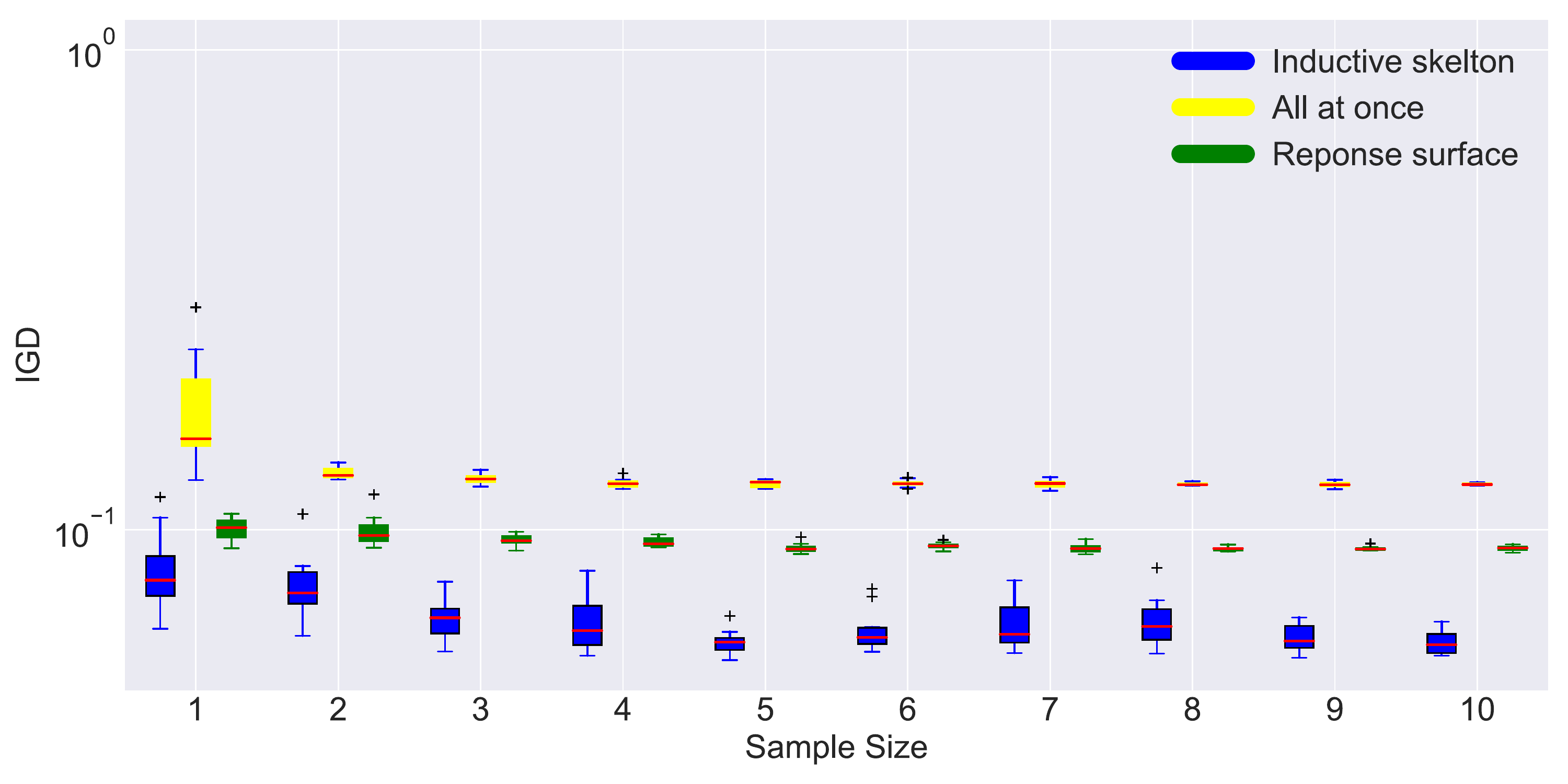}
\caption{Sample size $N_3$ vs.\ GD/IGD on 5-MED with sample size $(1,2,N_3)$ (boxplot\ over ten trials).}\label{fig:sample-size-MED5}

\includegraphics[width=0.49\hsize]{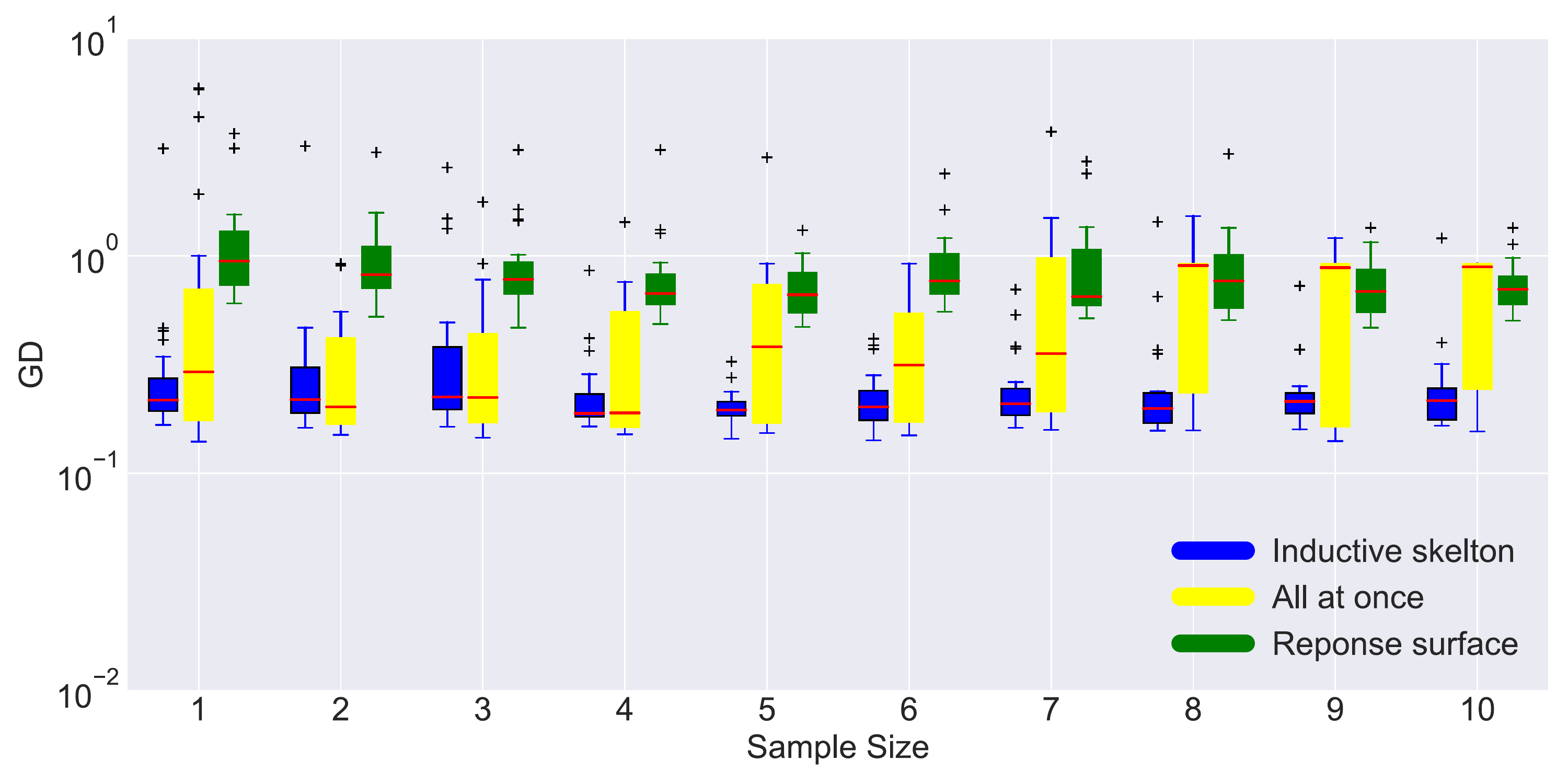}
\includegraphics[width=0.49\hsize]{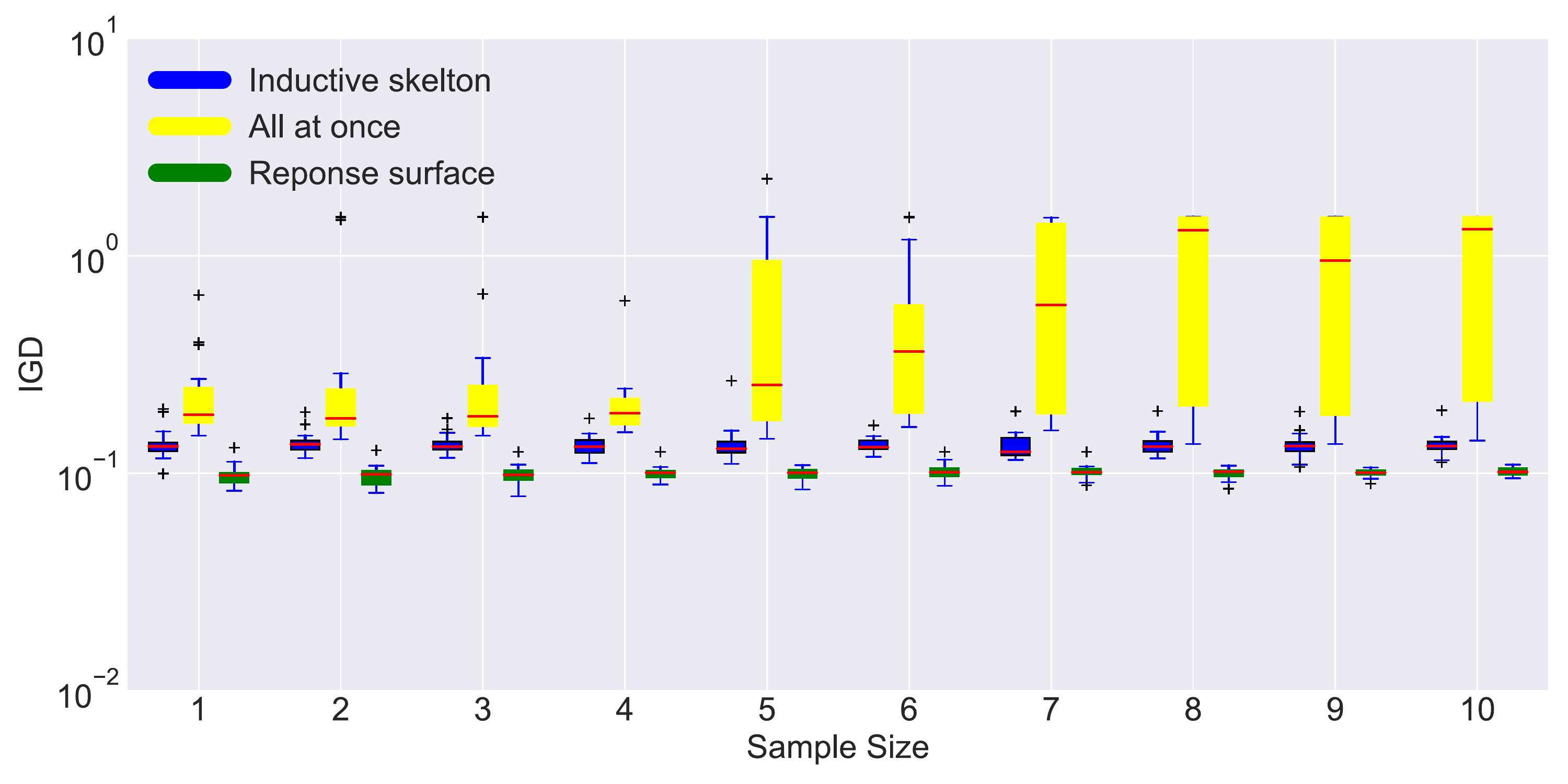}
\caption{Sample size $N_3$ vs.\ GD/IGD on S3TD with sample size $(1,5,N_3)$ (boxplot\ over ten trials).}\label{fig:sample-size-s3td}
\end{figure*}

\subsection{Approximation Accuracy}
For three- and five-objective problems, the inductive skeleton fitting achieved slightly better IGDs.
The inductive skeleton fitting adjusts a small subset of control points at each step with already-adjusted control points of the faces.
This reduces the number of parameters estimated at a time, which seems to prevent over-fitting.

More significant differences can be seen in GD\@.
\Cref{fig:3-MED} shows the cause of these defferences:
The all-at-once fitting obtained an overly-spreading B\'ezier simplex while the inductive skeleton fitting found an exactly-spreading B\'ezier simplex.
Minimizing the squared loss \Cref{eq:least-squares_bezier-simplex} only imposes that all sample points are close to the B\'ezier simplex, which leads to a good IGD\@.
However, it does not impose that all B\'ezier simplex points are close to the sample points, which is needed to achieve a good GD\@.
The inductive skeleton fitting stipulates that each face of the B\'ezier simplex must be close to each face of the Pareto front (i.e., the Pareto front of each subproblem), which minimizes GD\@.

Similarly to the all-at-once fitting, the response surface has poor GDs.
As we do not know the true Pareto set, it is difficult for grid sampling to obtain an exactly-spreading surface.

ConstrEx and Viennet2 are the only problems where the inductive skeleton fitting was worse than the response surface method.
ConstrEx is a non-smooth curve that cannot be fitted by a single B\'ezier curve (see \Cref{fig:Schaffer} in Appendix D).
This type of Pareto fronts leads to a challenging problem: developing a method for gluing multiple B\'ezier simplices to express a non-smooth surface.
Viennet2 is a smooth surface but its curvature is severely sharp (see \Cref{fig:Viennet2} in Appendix D).
Although Viennet2 is defined by quadratic functions (see Appendix B), its Pareto front cannot be fitted by the B\'ezier simplex of degree three.
This means that setting the degree of the B\'ezier simplex grater than or equal to the maximum degree in the problem definition does not ensure approximation accuracy.
It would be useful to develop a way to understand the required degree of the B\'ezier simplex from the problem definition.

\subsection{Required Sample Size}
To achieve good accuracy, how many sample points does the B\'ezier simplex require?
\Cref{fig:sample-size-MED5} shows transitions of GD and IGD on 5-MED when the sample size of each three-objective subproblem (two-dimensional face) varies $N_3 = 1, \dots, 10$ with fixed $N_1 = 1$ and $N_2 = 2$.
Surprisingly, both GD and IGD have already converged at $N_3 = 4$.

In case of $D=3$ for 5-MED, which is a five-objective ($M=5$) problem, the number of all control points to be esitimated is $\binom{3+5-1}{5} = 35$.
While the all-at-once fitting fits them simultaneously, in case of the induuctive skeleton method, we only have to fit two and one control points for each two- and three-objective subproblem respectively.
This result demonstrates that the inductive skeleton fitting allows us to fit Pareto fronts with small samples by reducing the number of control points to be estimated at a time.

\subsection{Practicality for Real-World Data}
Now we discuss the practicality of the inductive skeleton fitting.
Firstly, we focus on the performance of S3TD, which is a real-world problem.
\Cref{tbl:results} indicates that the average IGD of the inductive skeleton fitting on S3TD was worse than that of the response surface method.
However, the IGD itself is an order of $10^{-1}$, which small in absolute sense considering the number of the objective functions ($M=4$) and the sample size of the validation set ($N=58$) on S3TD.
The average GD of the inductive skeleton method was smaller than that of the other methods and the difference was significant with significance level $p=0.1$.
\Cref{fig:sample-size-s3td} shows transitions of GD and IGD on S3TD when the sample size of each three-objective subproblem (two-dimensional face) varies $N_3 = 1, \dots, 10$ with fixed $N_1 = 1$ and $N_2 = 5$.
We can observe that both GD and IGD already converged at $N_3 = 4$ as \Cref{fig:sample-size-MED5}.
These results suggest that the inductive skeleton fitting can work properly in real-world problems.

In real-world problems, it seems to be difficult to know that the Pareto set/front is a simplex and in some cases, it is degenerated.
However, a statistical test has been proposed to judge whether the Pareto set/front of a given problem is a simplex or not~\cite{Hamada18}.
This test tells us whether we can apply the inductive skeleton fitting to real-world problems.
Furthermore, even if the Pareto front is degenerated, \Cref{thm:approximation} states the Bezier simplex model can fit degenerate Pareto sets/fronts.

\subsection{Objective Map Approximation}
As well as the Pareto front approximation discussed and demonstrated so far, \Cref{thm:approximation} ensures that our method can be used for approximating the Pareto set, $X^*(\f)$, and the restricted objective map, $\f: X^*(\f)\to\f X^*(\f)$.
Such approximations may provide richer information about the problem.
To check the performance of graph approximation, we made a sample of solution-objective pairs of 5-MED $X = \Set{(\x,\f(\x))\in \R^5 \times \R^5 | \x\in X^*(\f)}$ and applied the inductive skeleton fitting to it.
Other settings were the same as the Pareto front approximation.
\Cref{tbl:results_2} shows the GD and IGD values where the results of the all-at-once fitting are just for baseline.
As the inductive skeleton fitting achieved GD and IGD of $10^{-1}$ order of magnitude, which means that it accurately approximates the graph of the restricted objective map in the abslute sense.
In 5-MED, the graph of the restricted objective map on the Pareto set is a four-dimensional topological simplex in a 10-dimensional space.
Although the codimension is six times higher than the case of Pareto front approximation, we have similar GD and IGD to ones of 5-MED in \Cref{tbl:results}.
This fact implies that the approximation error mainly depends on the intrinsic dimensionality rather than the ambient dimensionality.

\begin{table*}[h]
\centering
\tiny
\caption{GD and IGD (avg.\ $\pm$ s.d.\ over 20 trials) with sample size $(1,2,1)$ for the graph of 5-MED.}\label{tbl:results_2}
\begin{tabular}{ccccccccccc}
\toprule
\multicolumn{3}{c}{\textbf{Inductive skeleton}}&\multicolumn{3}{c}{\textbf{All-at-once}}\\
Iteration&GD&IGD&Iteration&GD&IGD\\
\midrule
2.84 $\pm$ 0.36 & 3.36e-1 $\pm$ 9.48e-2 & 2.05e-1 $\pm$ 3.45e-2 & 3.00 $\pm$ 0.00 & 1.38e-00 $\pm$ 2.00e-00 & 2.79e-01 $\pm$ 8.21e-02 \\
\bottomrule
\end{tabular}
\end{table*}

\section{Conclusions}
In this paper, we have proposed the B\'ezier simplex model and its fitting algorithm for approximating Pareto fronts of multi-objective optimization problems.
An approximation theorem of the B\'ezier simplex has been proven.
Numerical experiments have shown that our fitting algorithm, the inductive skeleton fitting, obtains exactly-spreading B\'ezier simplices over synthetic Pareto fronts of different shape and dimensionality.
It has been also observed that a real-world problem with four objectives but only 58 points is accurately fitted.
The proposed model and its fitting algorithm drastically reduce the sample size required to describe the entire Pareto front.

The current algorithm has some drawbacks.
The inductive skeleton fitting requires that each subproblem must have a non-empty sample, which may be too demanding for some practice.
Furthermore, the loss function for fitting does not take into account of the sampling error of the Pareto front.
To get a more stable approximation under wild samples, we plan to extend the (deterministic) Bezier simplex to a probabilistic model.

We believe that the potential applications of the method are not limited to the use in post-optimization process.
It will be introduced in evolutionary algorithms and dynamically evolved to accelerate search.
We also expect that this model can be applied to regression problems that have complex boundary conditions such as learning from multi-labeled data.

\section{Acknowledgements}
We wish to thank Dr. Yuichi Ike for checking the proof and making a number of valuable suggestions.

\bibliographystyle{aaai}
\bibliography{ref}

\appendix
\section{Appendix A: Proof of \Cref{thm:approximation}}
In this section, we prove that B\'ezier simplices can approximate pareto sets of simplicial problems in any accuracy.
Let $\bm b: \Delta^{M-1} \to \R^M$ be a B\'ezier simplex.
We say this function $\bm b: \Delta^{M-1} \to \R^M$ B\'ezier maps and the projections to its coordinates $pr \circ \bm b: \Delta^{M-1} \to \mathbb{R}$ B\'ezier functions.
Namely,
\[
\bm b = (b_1, \ldots, b_M)
\]
where $b_1, \ldots, b_M$ are B\'ezier functions.

\let\temp\thetheorem
\renewcommand{\thetheorem}{\ref{thm:approximation}}

\begin{theorem}\label{thm:main}
Let $\bm \phi: \Delta^{M-1} \to \R^M$ be a $C^0$-embedding.
There exists an infinite sequence of B\'ezier simplices $\bm b^{(i)}: \Delta^{M-1} \to \R^M$ such that
\[
\lim_{i \to \infty} \sup_{\bm t \in \Delta^{M-1}} |\bm \phi(\bm t)-\bm b^{(i)}(\bm t)| = 0.
\]
\end{theorem}

\let\thetheorem\temp
\addtocounter{theorem}{-1}

\begin{remark}
By definition, if $X^*$ is the Pareto set of a simplicial problem, then there is a homeomorphism $\bm \phi : \Delta^{M-1} \to X^*$.
Therefore, \Cref{thm:main} implies that Pareto sets of simplicial problems can be approximated by B\'ezier simplices.
\end{remark}

To prove this theorem, we use the following theorem by Stone-Weierstrass.
Let us prepare some notations before stating Stone-Weierstrass theorem.
Let $X$ be a compact metric space.
Let $C(X)$ be the set of continuous functions over $X$.
We regard $C(X)$ as a Banach algebra with the sup-norm,
\[
\norm{f}_{\infty} = \sup_{\bm t \in X} \abs{f(\bm t)}.
\]
A subalgebra of $C(X)$ is a vector subspace $S$ such that for any $f, g \in S$, $f \cdot g \in S$.
We say that a subalgebra $S$ separates points if for any different points $\bm x, \bm y \in X,$ there exists $f \in S$ so that $f(\bm x) \neq f(\bm y)$.
We say that $S$ is unital if $S$ contains $1$.

\begin{theorem}[Stone-Weierstrass]\label{thm:stone}
Let $S$ be a unital subalgebra of $C(X)$.
Assume $S$ separates any two points in $X$.
Then $S$ is dense in $C(X)$, namely $\bar{S}=C(X)$.
\end{theorem}

We use Stone-Weierstrass theorem in the case where $X=\Delta^{M-1}$ and $S$ is the set of B\'ezier functions.
To confirm the assumptions of Stone-Weierstrass theorem, we need the following propositions.

\begin{proposition}\label{thm:alg}
Let $S$ be the set of B\'ezier functions of any degree.
Then $S$ is a unital subalgebra of $C(\Delta^{M-1})$.
\end{proposition}

\begin{proof}
We can easily see that the product of B\'ezier functions of degree $D_1$ and $D_2$ is a B\'ezier function of degree $D_1 + D_2$.
Hence, it is enough to show that the linear combination of B\'ezier functions is a B\'ezier function again.
Since it is clear that  the linear combination of B\'ezier functions of the same degree is a B\'ezier function, it is enough to show that B\'ezier functions of degree $D_1$ can be represented by B\'ezier functions of degree $D_2$, if $D_1\leq D_2$.
To see this, let us denote by  $S_D$ the set of the B\'ezier functions whose degree is equal to $D$.
We can regard $S_D$ as a linear subspace of $\mathbb{R}[t_1, \ldots ,t_M]$.
Let $S'_D$ be the set of polynomials of degree lower than or equal to $D$ in $\mathbb{R}[t_1, \ldots ,t_M]$ (the degree is the total degree by $t_1, \ldots ,t_M$).
By the definition of B\'ezier functions, $S_D$ is contained in $S'_D$.
Since $t_1^{D_1}, \ldots, t_M^{D_M}$ are linearly independent over $\mathbb{R}$, the dimension of $S_D$ is equal to $\binom{M+D-1}{M-1}$.
However, the dimension of $S'_D$ is equal to $\binom{M+D-1}{M-1}$.
Therefore, $S_D$ is equal to $S'_D$.
Hence, $ S_{D_2} = S'_{D_2} \supset S'_{D_1} = S_{D_1}$.
$S$ is a subalgebra of $C(\Delta^{M-1})$.
Clearly $S'_0$ contains $1$.
This implies that $S$ is unital.
\end{proof}

\begin{proposition}\label{thm:sep}
Let $S$ be the set of B\'ezier functions of any degree.
Then $S$ separates points.
\end{proposition}

\begin{proof}
Take two different points $\bm s, \bm t \in \Delta^{M-1}$.
Consider the B\'ezier functions of degree $1$.
By the definition of B\'ezier functions, B\'ezier functions of degree $1$ define arbitrary hyperplanes (or equivalently linear equations).
We can find a hyperplane $\bm b$ which contains $\bm s$ and does not contain $\bm t$.
In this case, $\bm b(\bm s) \neq 0 = \bm b(\bm t)$.
Hence $S$ separates points.
\end{proof}

Now we prove the main theorem.

\begin{proof}[proof of \Cref{thm:main}]
The Hausdorff distance is defined as
\[
d_{\mathrm H}(X,Y) := \max\Set{\sup_{\bm x \in X} \inf_{\bm y \in Y} d(\bm x, \bm y),\,\sup_{\bm y \in Y} \inf_{\bm x \in X} d(\bm x,\bm y)}.
\]
Hence, if there is a B\'ezier simplex $\bm b: \Delta^{M-1} \to \R^M$ such that for any $\bm t \in \Delta^{M-1}$, $\abs{\bm \phi(\bm t) - \bm b(\bm t)}$ is enough small, we obtain the assertion.
Let $S$ be the set of B\'ezier functions of arbitrary degree.
By \Cref{thm:alg,thm:sep}, $S$ is a unitary algebra which separates points on $\Delta^{M-1}$.
By \Cref{thm:stone}, $S$ is dense in $C(\Delta^{M-1})$.
Hence we can find B\'ezier functions $b_i$ such that $\norm{pr_i \circ \bm \phi - b_i}_{\infty}$ is enough small, where $pr_i$ denote the projection to the $i$-th coordinate.
Put $\bm b := (b_1, \ldots, b_M)$.
Then
\begin{align*}
\sup_{\bm t \in \Delta^{M-1}} \abs{\bm \phi(\bm t) - \bm b(\bm t)}
 & \le \sum_{i=1}^M \sup_{\bm t \in \Delta^{M-1}} \abs{pr_i \circ \bm \phi - b_i}\\
 & =   \sum_{i=1}^M \norm{pr_i \circ \bm \phi - b_i}_{\infty}
\end{align*}
holds.
If the right hand side is enough small, $\sup_{\bm t \in \Delta^{M-1}} \abs{\bm \phi(\bm t) - \bm b(\bm t)}$ is also enough small.
\end{proof}

\section{Appendix B: Problem Definition}
\paragraph{Schaffer}
is a one-variable two-objective problem defined by:
\begin{align*}
\text{minimize }
&f_1(x) = x^2,\\
&f_2(x) = {(x-2)}^2\\
\text{subject to }&-100000 \le x \le 100000.
\end{align*}

\paragraph{ConstrEx}
is a two-variable two-objective problem defined by:
\begin{align*}
\text{minimize }
&f_1(x) = x_1,\\
&f_2(x) = \frac{1+x_2}{x_1}\\
\text{subject to }&x_2 + 9x_1 -6 \ge 0,\\
&-x_2 + 9x_1 -1 \ge 0,\\
&0.1 \le x_1 \le 1,\\
&0 \le x_2 \le 5.
\end{align*}

\paragraph{Osyczka2}
is a six-variable two-objective problem defined by:
\begin{align*}
\text{minimize }
&f_1(x) = -25{(x_1-2)}^2 - {(x_2-2)}^2\\
&\qquad\qquad - {(x_3-1)}^2 - {(x_4-4)}^2 - {(x_5-1)}^2,\\
&f_2(x) = \sum_{i=1}^6 x_i^2\\
\text{subject to }&x_1+x_2-2\ge0,\\
&6-x_1-x_2\ge0,\\
&2-x_2+x_1\ge0,\\
&2-x_1+3x_2\ge0,\\
&4-{(x_3-3)}^2-x_4\ge0,\\
&{(x_5-3)}^2+x_6-4\ge0,\\
&0 \le x_1,x_2,x_6\le 10,\\
&1 \le x_3,x_5\le 5,\\
&0 \le x_4\le 6.
\end{align*}

\paragraph{Viennet2}
is a two-variable three-objective problem defined by:
\begin{align*}
\text{minimize }
&f_1(x) = \frac{{(x_1-2)}^2}{2} + \frac{{(x_2+1)}^2}{13}+3,\\
&f_2(x) = \frac{{(x_1+x_2-3)}^2}{36} + \frac{{(-x_1+x_2+2)}^2}{8}-17,\\
&f_3(x) = \frac{{(x_1+2x_2-1)}^2}{175} + \frac{{(2x_2-x_1)}^2}{17}-13\\
\text{subject to }&-4\le x_1,x_2\le4.
\end{align*}

\paragraph{$M$-MED}
is an $M$-variable $M$-objective problem defined by:
\begin{align*}
\text{minimize }
&f_m (x) = \paren{\frac{1}{\sqrt{2}} \norm{x - e_m}}^{p_m}~&&(m=1,\ldots,M)\\
\text{subject to }
&-5.12 \le x_i\le 5.12~&&(m=1,\ldots,M)\\
\text{where }
&p_m = \exp\left(\frac{2(m-1)}{M-1} - 1\right)&&(m=1,\ldots,M),\\
& e_m = (0,\dots,0,\underbrace{1}_{m\text{-th}},0,\dots,0)&&(m=1,\ldots,M).
\end{align*}

\paragraph{S3TD}
is a 58-variable four-objective problem of designing a super-sonic airplane.
The original problem has five objectives:
\begin{enumerate}
\item Minimizing pressure drug
\item Minimizing fliction drug
\item Maximizing lift
\item Minimizing sonic boom intensity
\item Minimizing structural weight
\end{enumerate}
All objectives are converted into minimization.
We removed $f_1$ since $f_1$ and $f_4$ are strongly correlated and $f_4$ is the primary objective in this project.

\section{Appendix C: Data Description}
The description of data that we used in the numerical experiments is summarized in \Cref{tbl:datasets}:
\Cref{tbl:datasets} shows the sample size of each subproblem.
To make a data set for each problem, we chose mutually exclusive subsamples of the Pareto fronts of the subproblems.

\begin{table}[h]
\centering
\small
\caption{Sample size on the Pareto front for each subproblem.}\label{tbl:datasets}
\begin{tabular}{cccc}
\toprule
\textbf{Problem}&1-obj&2-obj&3-obj\\
\midrule
Schaffer&1 $(\times2)$&201 $(\times1)$&--\\
ConstrEx&1 $(\times2)$&679 $(\times1)$&--\\
Osyczka2&1 $(\times2)$&570 $(\times1)$&--\\
\midrule
3-MED&1 $(\times3)$&17 $(\times3)$&153 $(\times1)$\\
Viennet2&1 $(\times3)$&\begin{tabular}{c}
    811 $(f_1,f_2)$, \\ 1310 $(f_2,f_3)$,\\ 1618 $(f_3,f_1)$
\end{tabular}&8122 $(\times1)$\\
\midrule
5-MED&1 $(\times5)$&15 $(\times10)$&105 $(\times10)$\\
\midrule
S3TD&1 $(\times4)$&\begin{tabular}{c}
					19 $(f_1,f_2)$,\\

					9 $(f_1,f_3)$,\\
					21 $(f_1,f_4)$,\\
					6 $(f_2,f_3)$,\\
					13 $(f_2,f_4)$,\\
					25 $(f_3,f_4)$
				    \end{tabular}
				&\begin{tabular}{c}
					37 $(f_1,f_2,f_3)$,\\
					24 $(f_1,f_2,f_4)$,\\
					25 $(f_1,f_3,f_4)$,\\
					30 $(f_2,f_3,f_4)$
				    \end{tabular}\\
\bottomrule
\end{tabular}
\end{table}

\section{Appendix D: All Results}
\subsection{The Estimated B\'ezier Simplices and Response Surfaces for Pareto Front}
\Cref{fig:Schaffer} shows pairplots of the estimated B\'ezier simplices and response surfaces for two-objective problems with sample size $(1,3)$.
For three- and five-objective problems, the results of 3-MED, Viennet2 and 5-MED with sample size $(1,2,2)$ are shown in \Cref{fig:3-MED_appendix,fig:Viennet2,fig:5-MED}, respectively.
For a real-world problem, the results of S3TD, which is four-objective problem with sample size $(1,5,5)$ is shown in \Cref{fig:S3TD}.

\subsection{Transitions of GD and IGD}
\Cref{fig:sample-size-2D} shows boxplots of GD and IGD when we changed the sample size on the edges for two-objective problems.
The results of 3-MED, Viennet2 and 5-MED are shown in \Cref{fig:sample-size-MED,fig:sample-size-Viennet,fig:sample-size-MED5_appendix}, respectively.

\begin{figure*}[t]
\centering%
\subfloat[Inductive skeleton]{%
\includegraphics[width=0.3\hsize]{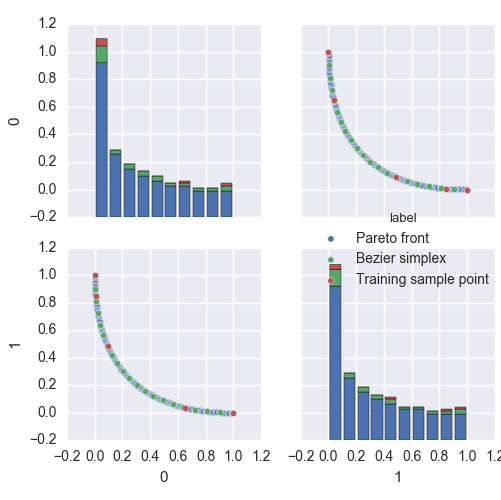}
\includegraphics[width=0.3\hsize]{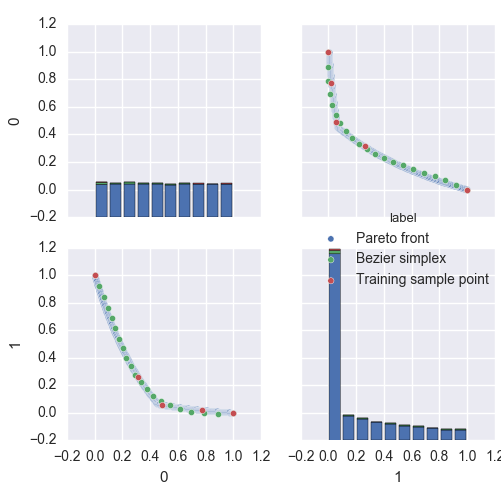}
\includegraphics[width=0.3\hsize]{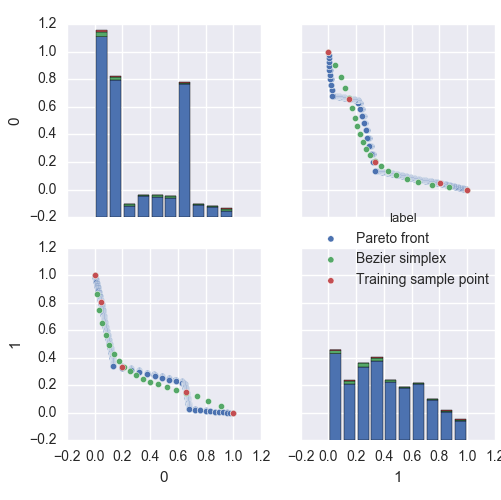}}\\
\subfloat[All-at-once]{%
\includegraphics[width=0.3\hsize]{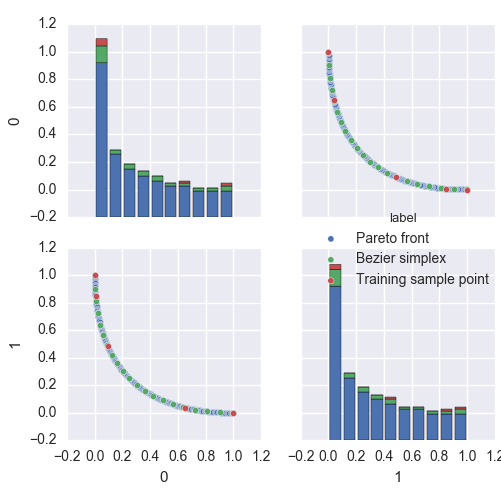}
\includegraphics[width=0.3\hsize]{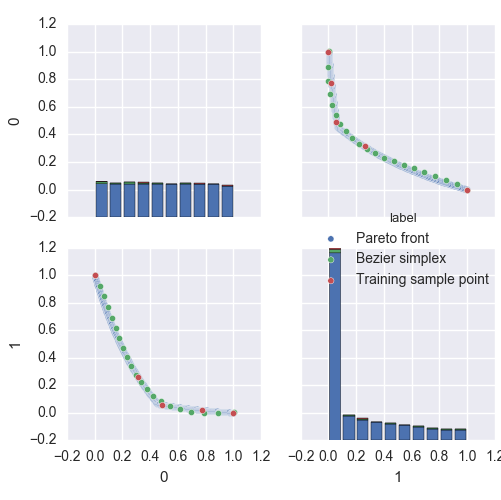}
\includegraphics[width=0.3\hsize]{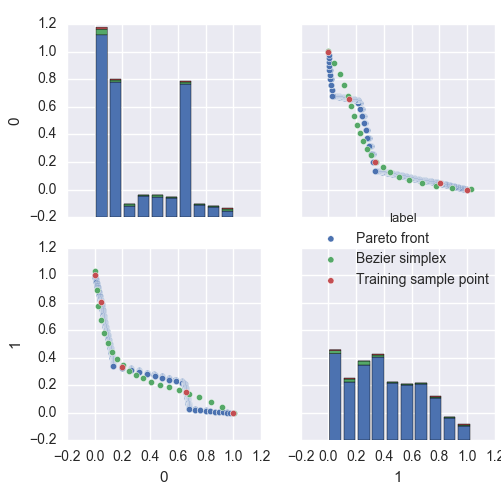}}\\
\subfloat[Response surface]{%
\includegraphics[width=0.3\hsize]{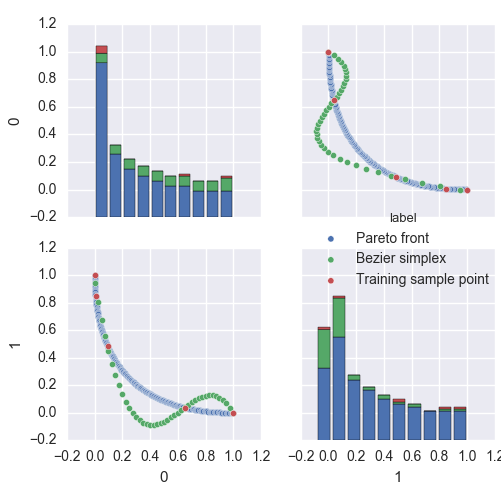}
\includegraphics[width=0.3\hsize]{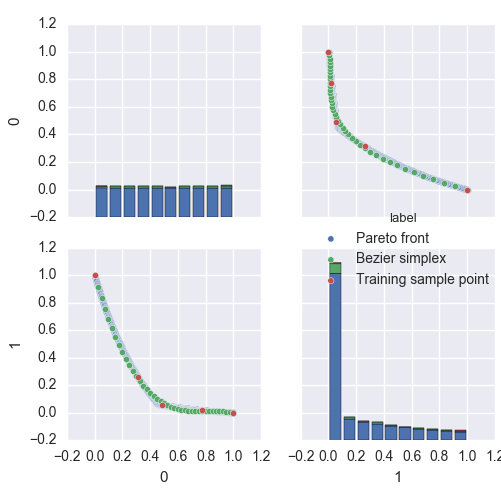}
\includegraphics[width=0.3\hsize]{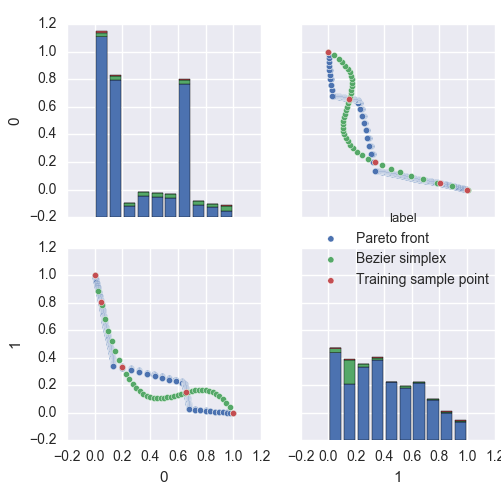}}\\
\caption{B\'ezier curves and a response surface for Schaffer (left), ConstrEx (center) and Osyczka2 (right) with sample size $(1, 3)$.}\label{fig:Schaffer}
\end{figure*}

\begin{figure*}[t]
\centering%
\subfloat[Inductive skeleton]{\includegraphics[width=0.4\hsize]{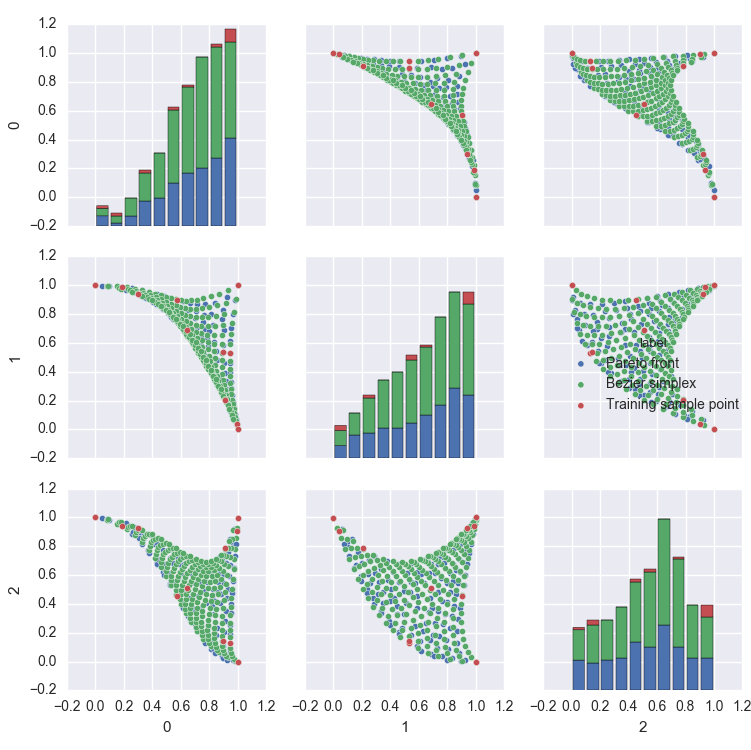}}\\
\subfloat[All-at-once]{\includegraphics[width=0.4\hsize]{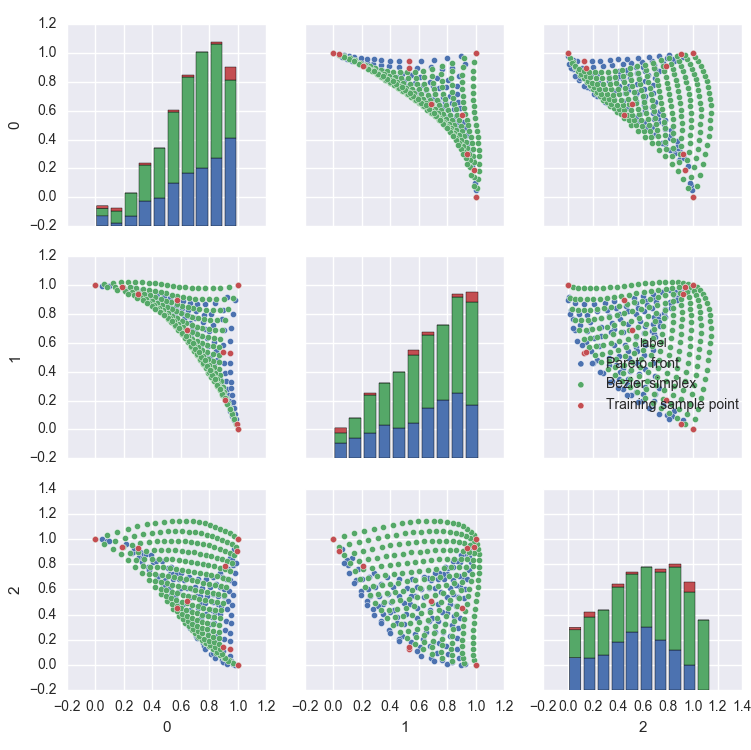}}\\
\subfloat[Response surface]{\includegraphics[width=0.4\hsize]{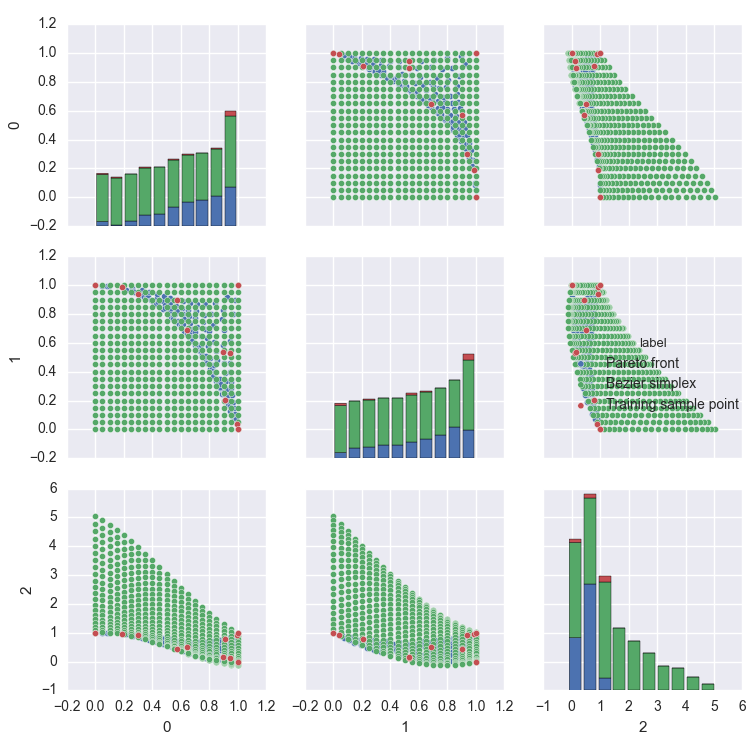}}
\caption{B\'ezier triangles and a response surface for 3-MED with sample size $(1, 2, 2)$.}\label{fig:3-MED_appendix}
\end{figure*}

\begin{figure*}[t]
\centering%
\subfloat[Inductive skeleton]{\includegraphics[width=0.4\hsize]{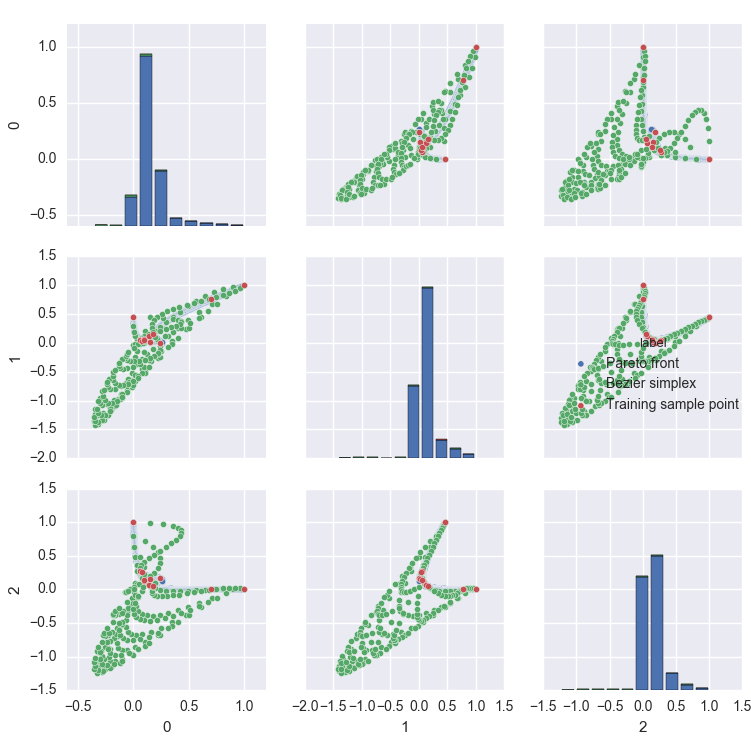}}\\
\subfloat[All-at-once]{\includegraphics[width=0.4\hsize]{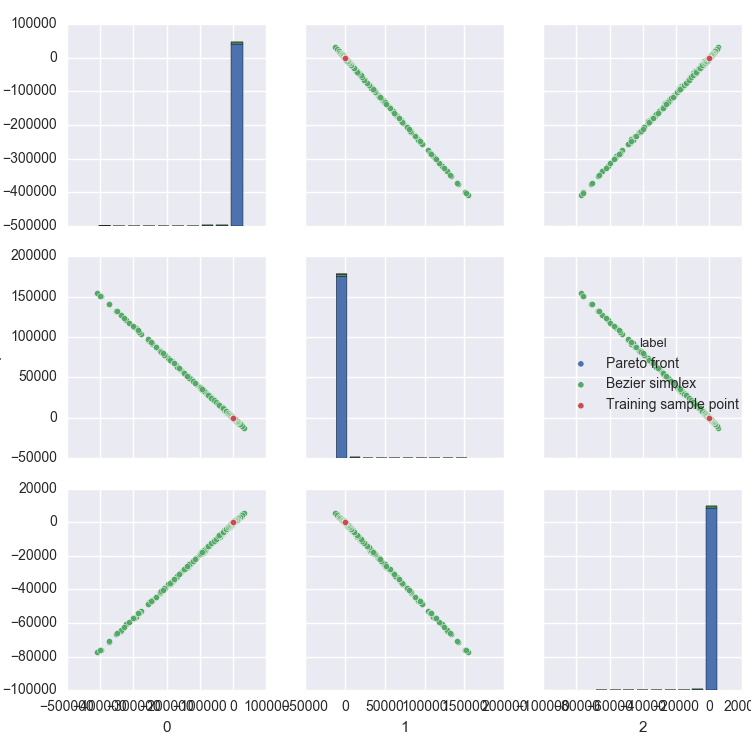}}\\
\subfloat[Response surface]{\includegraphics[width=0.4\hsize]{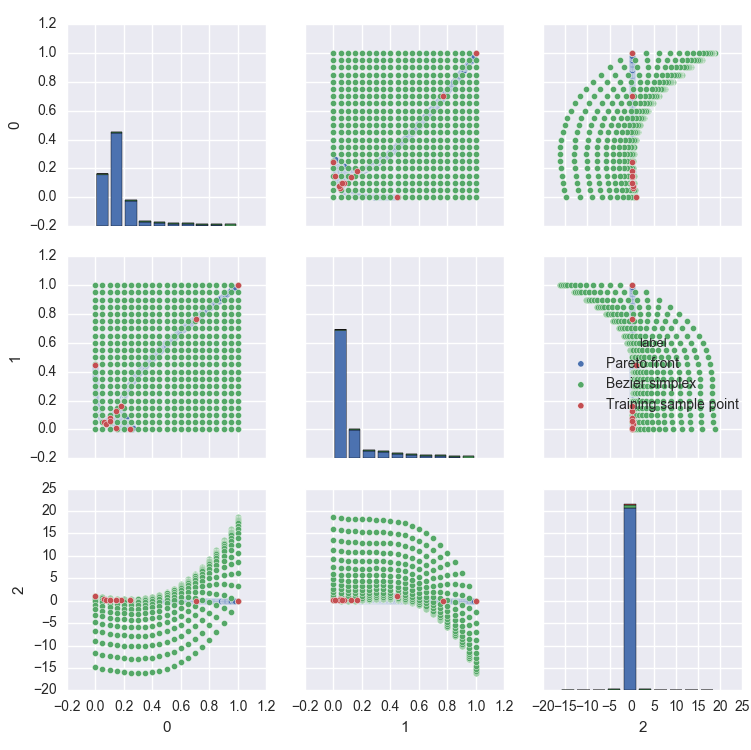}}
\caption{B\'ezier triangles and a response surface for Viennet2 with sample size $(1, 2, 2)$.}\label{fig:Viennet2}
\end{figure*}

\begin{figure*}[t]
\centering%
\subfloat[Inductive skeleton]{\includegraphics[width=0.4\hsize]{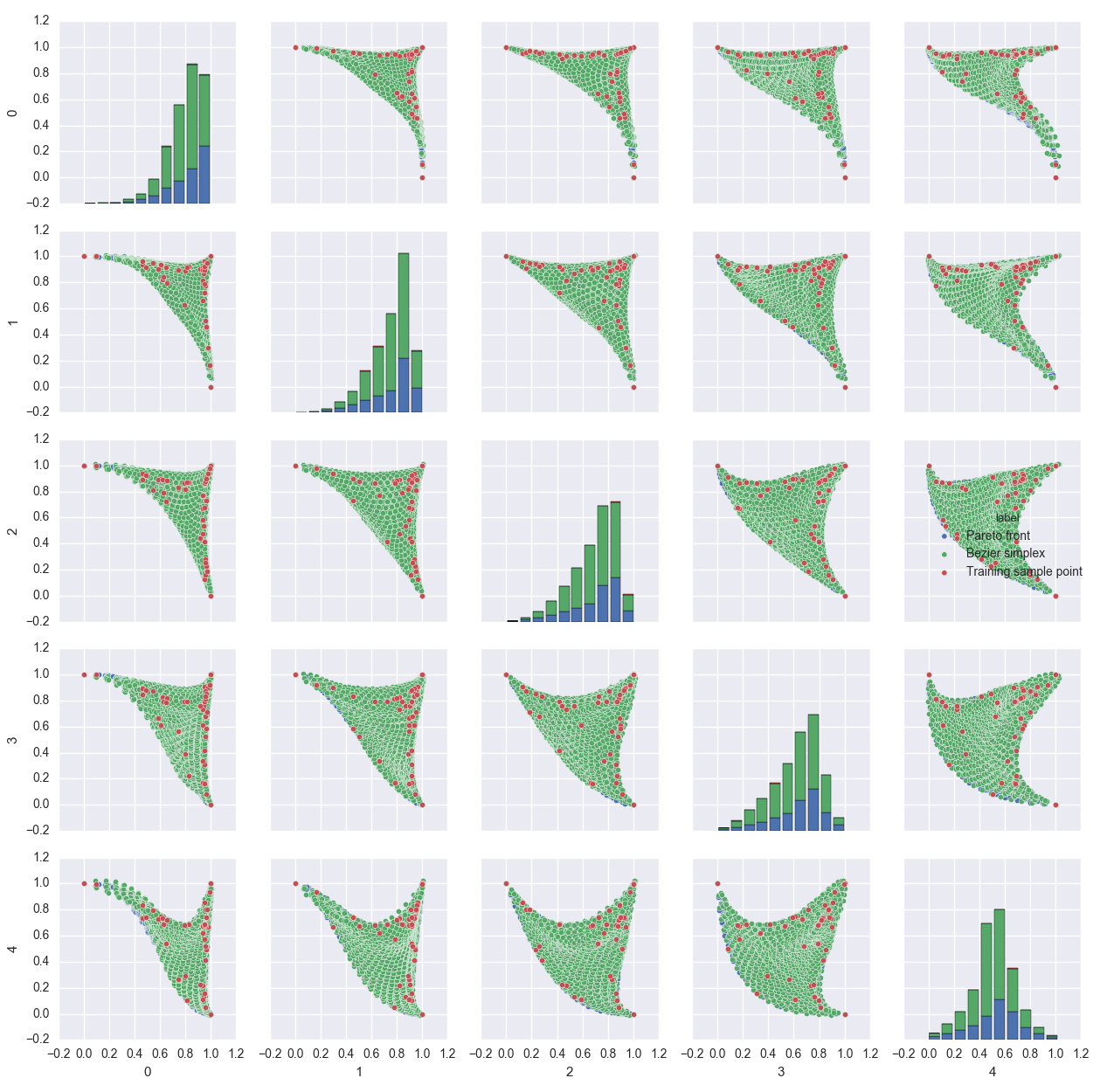}}\\
\subfloat[All-at-once]{\includegraphics[width=0.4\hsize]{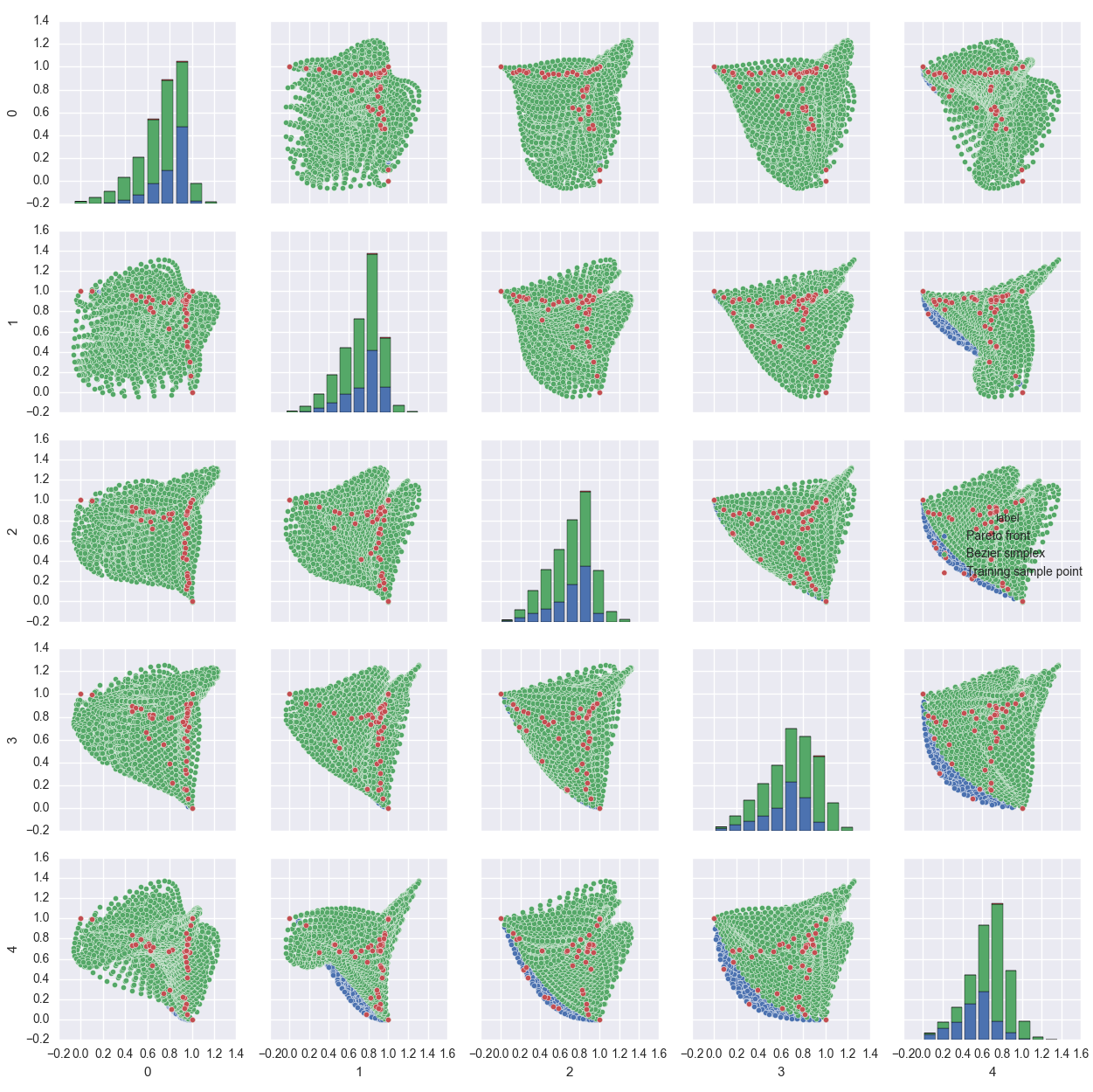}}\\
\subfloat[Response surface]{\includegraphics[width=0.4\hsize]{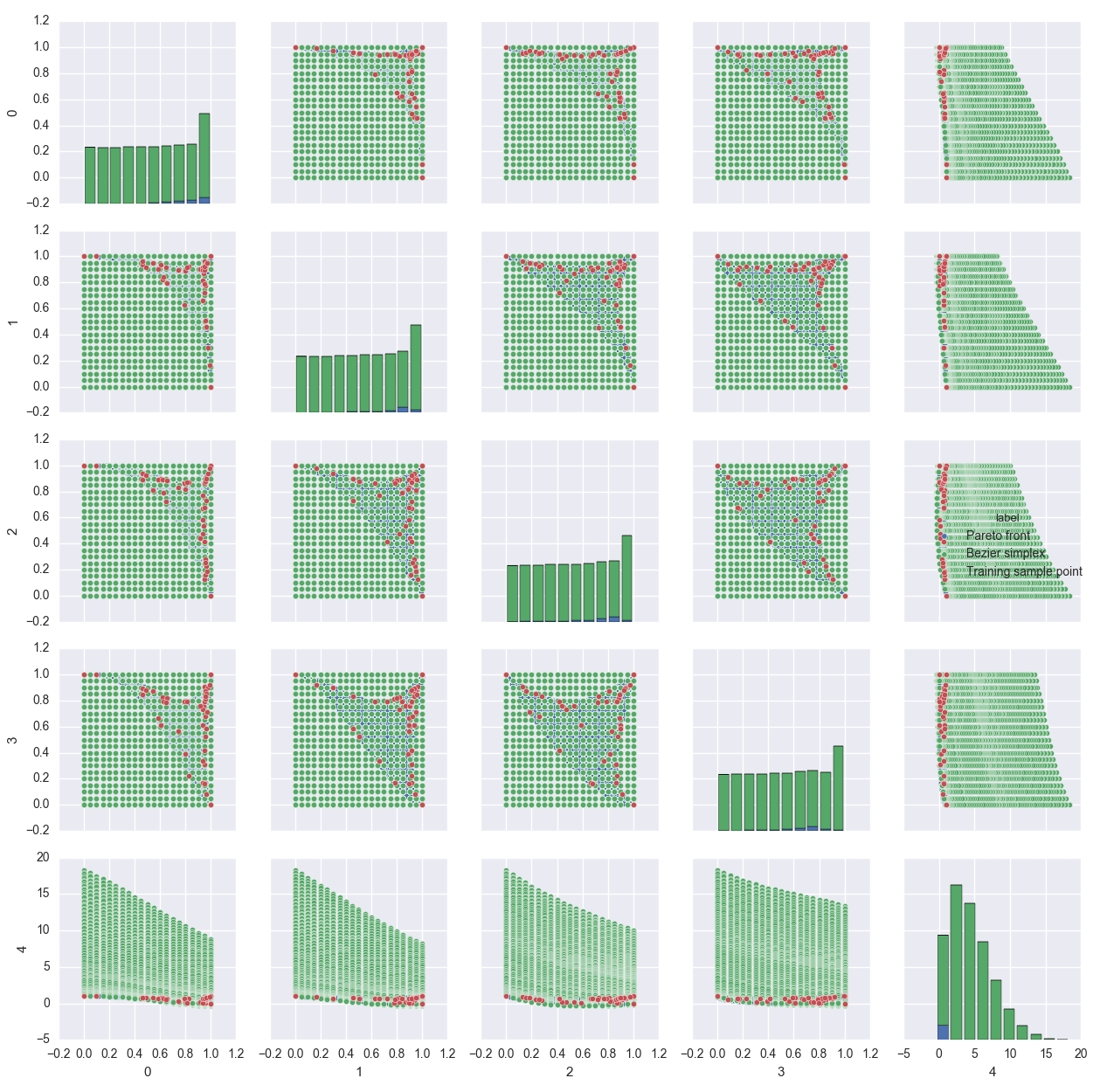}}
\caption{B\'ezier triangles and a response surface for 5-MED with sample size $(1, 2, 2)$.}\label{fig:5-MED}
\end{figure*}

\begin{figure*}[t]
\centering%
\subfloat[Inductive skeleton]{\includegraphics[width=0.4\hsize]{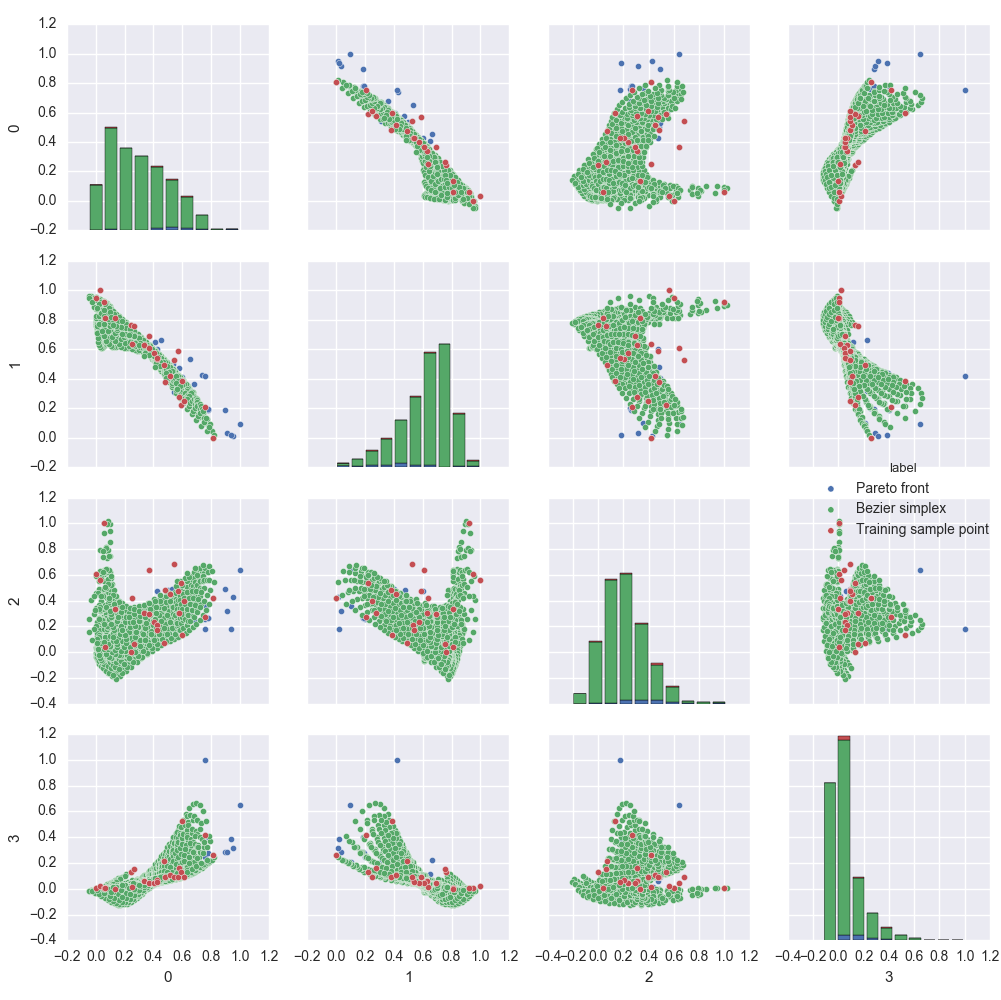}}\\
\subfloat[All-at-once]{\includegraphics[width=0.4\hsize]{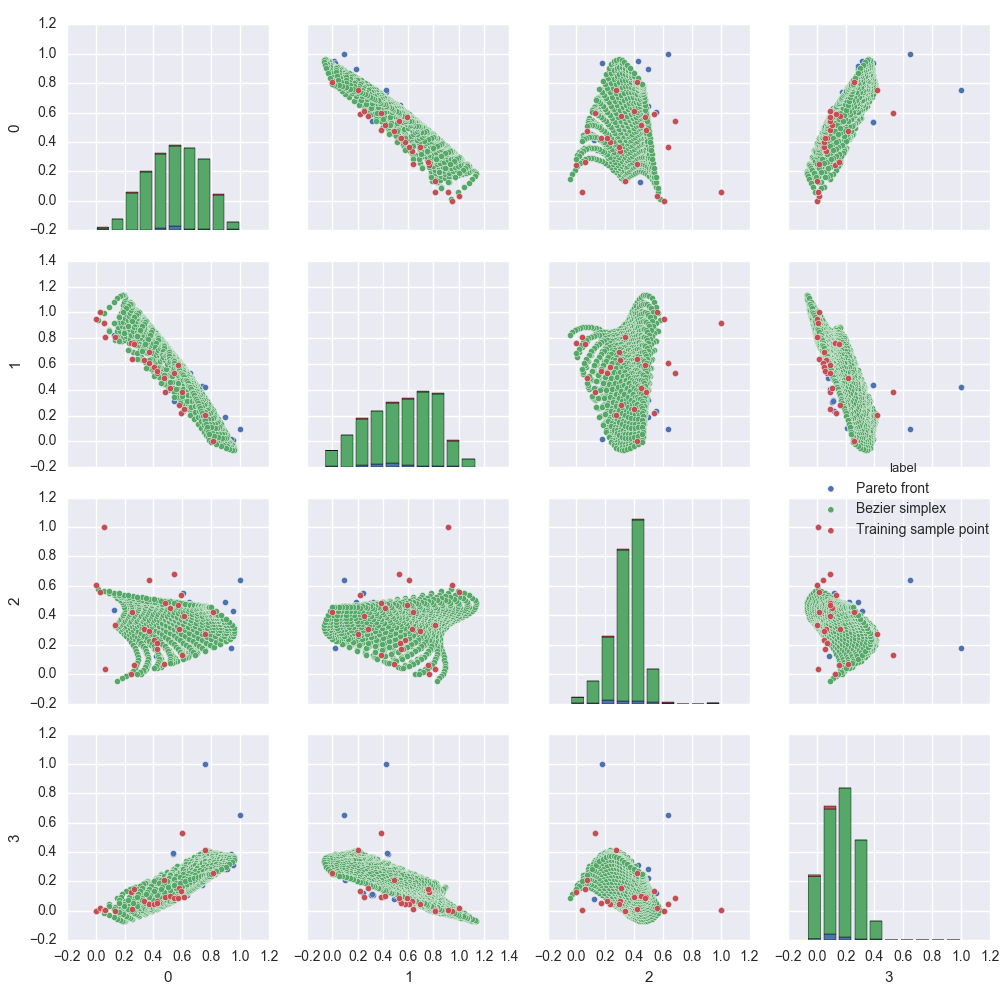}}\\
\subfloat[Response surface]{\includegraphics[width=0.4\hsize]{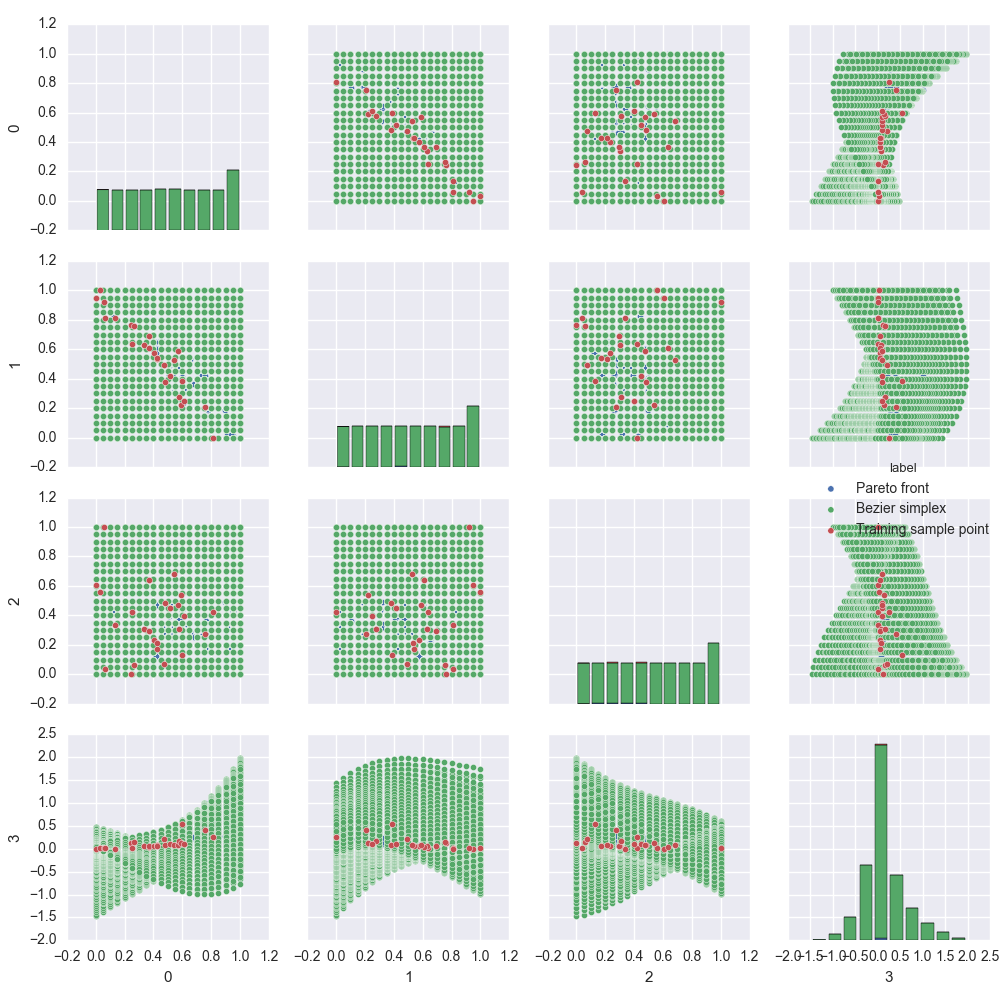}}
\caption{B\'ezier triangles and a response surface for S3TD with sample size $(1, 5, 5)$.}\label{fig:S3TD}
\end{figure*}

\begin{figure*}[t]
\centering%
\subfloat[Schaffer with sample size $(1,N_2)$]{%
\includegraphics[width=0.49\hsize]{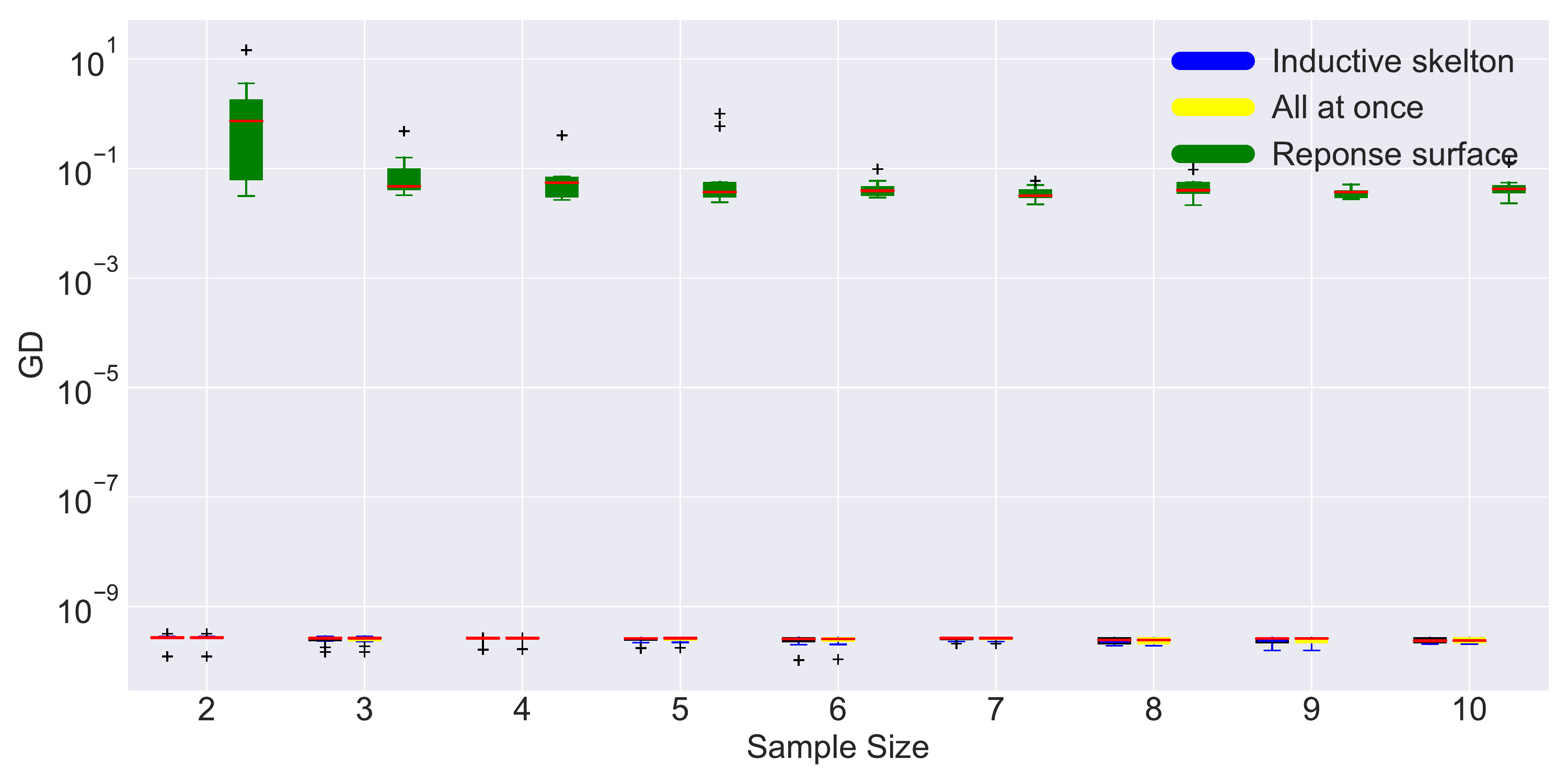}
\includegraphics[width=0.49\hsize]{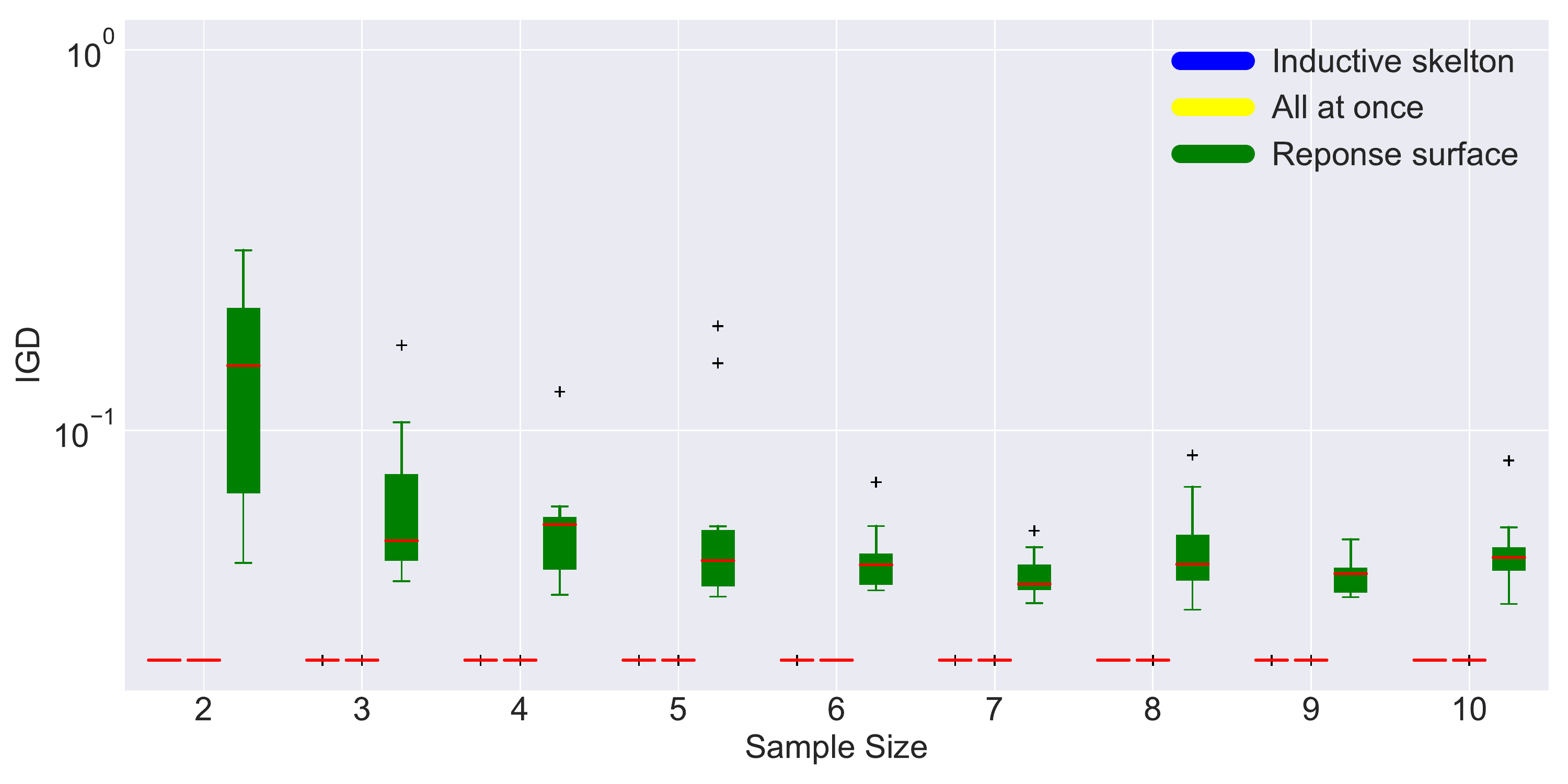}}\\
\subfloat[ConstrEx with sample size $(1,N_2)$]{%
\includegraphics[width=0.49\hsize]{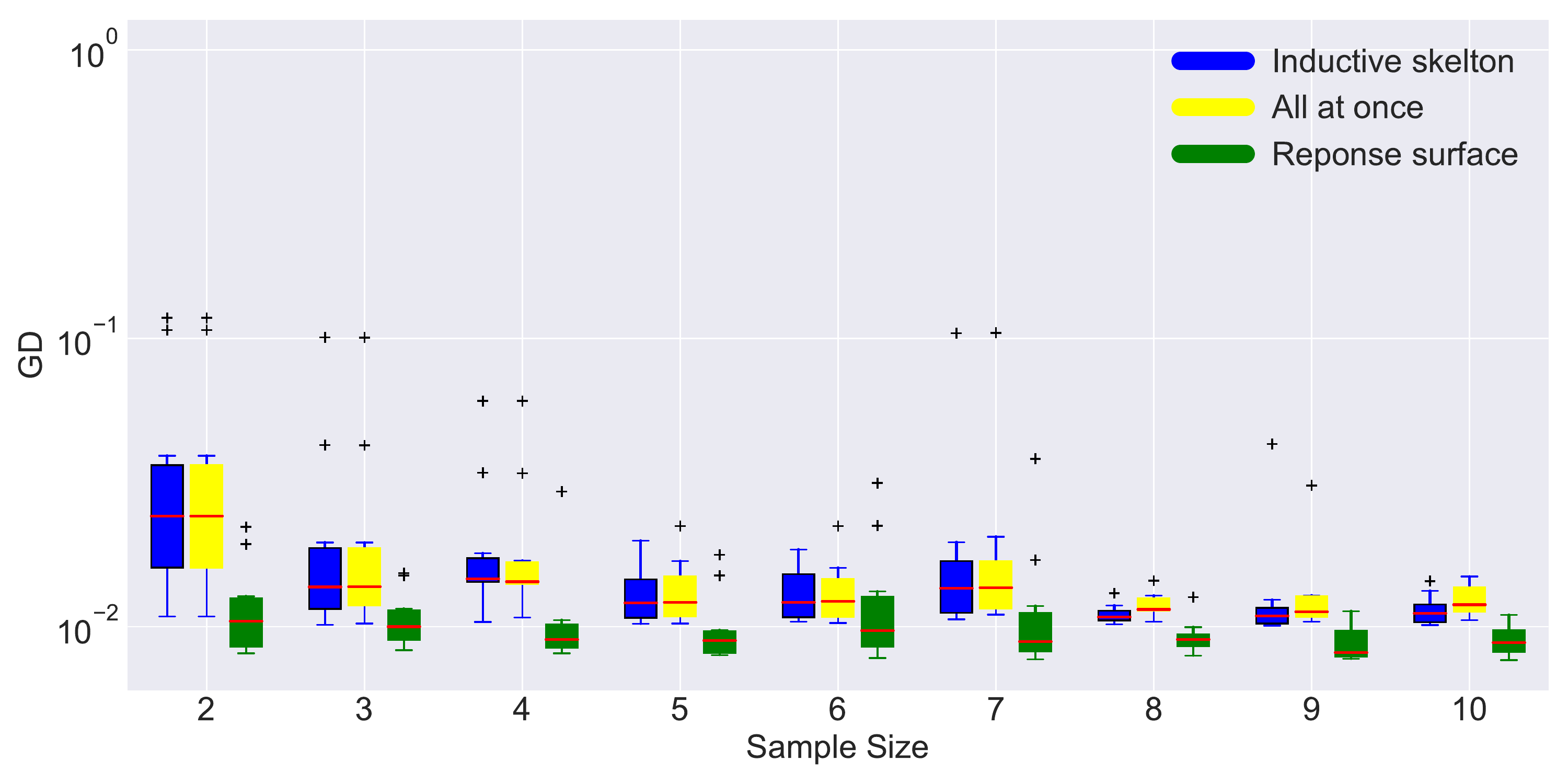}
\includegraphics[width=0.49\hsize]{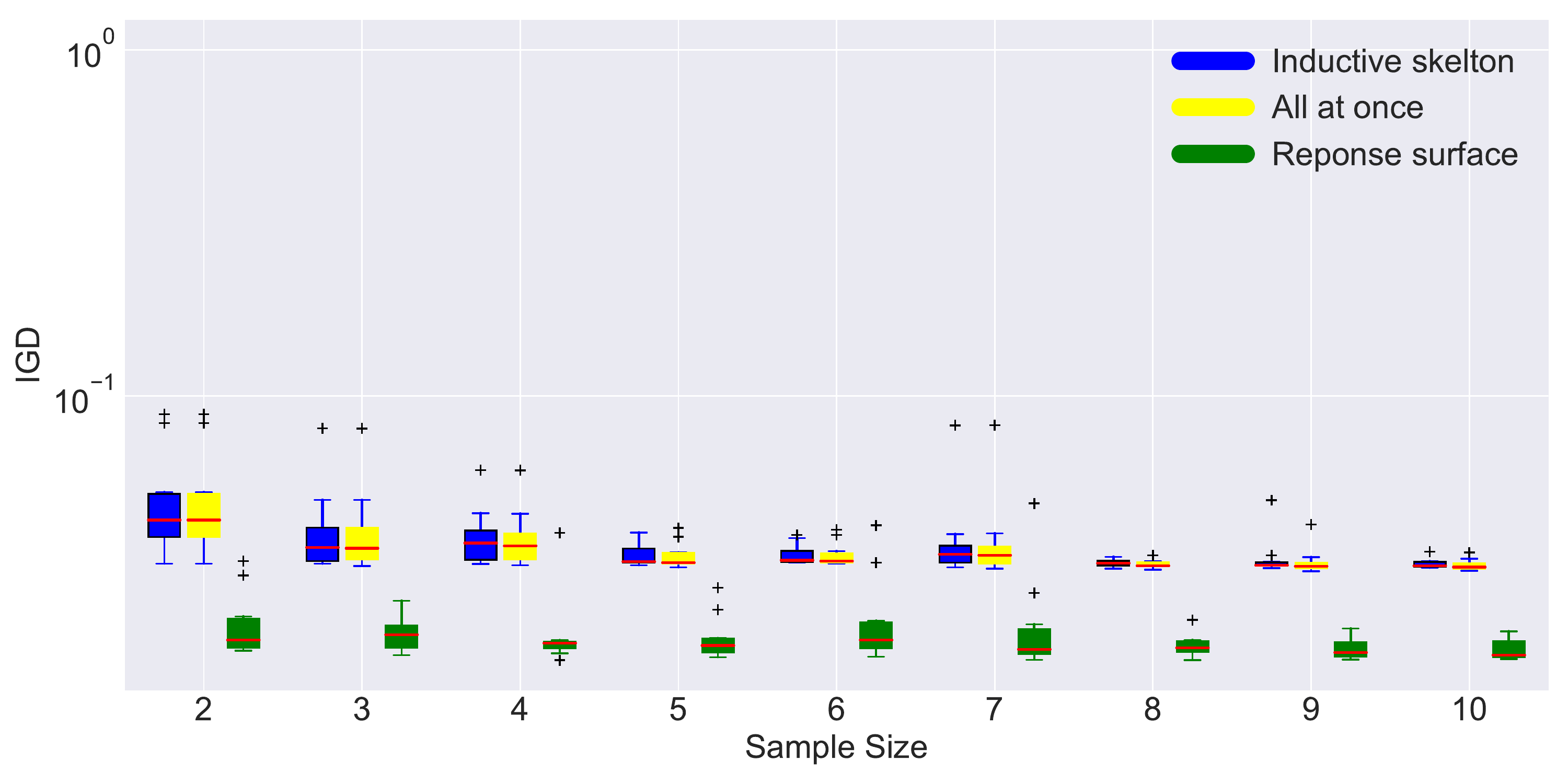}}\\
\subfloat[Osyczka2 with sample size $(1,N_2)$]{%
\includegraphics[width=0.49\hsize]{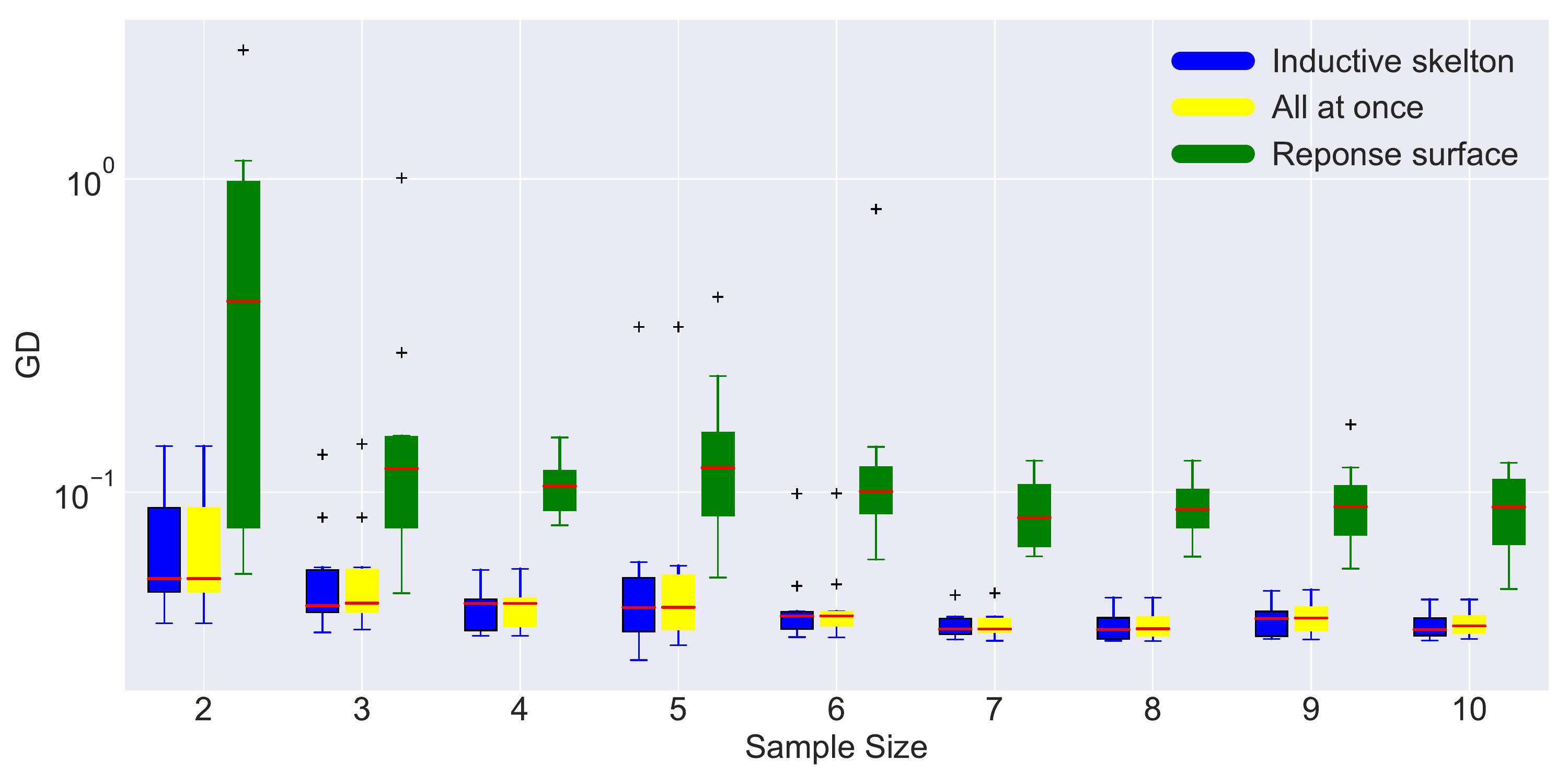}
\includegraphics[width=0.49\hsize]{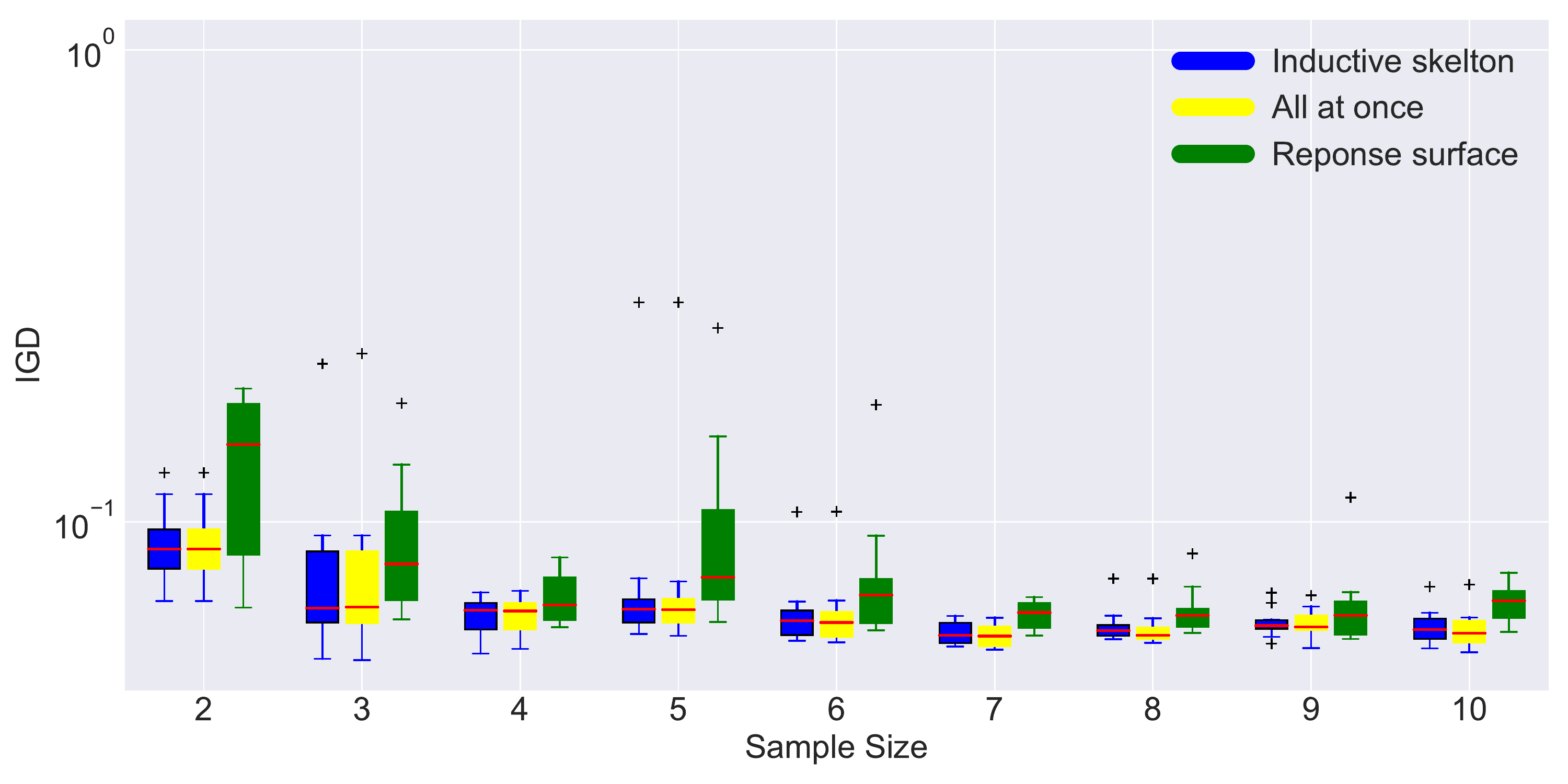}}
\caption{Sample size $N_2$ vs.\ GD/IGD (boxplot over ten trials).}\label{fig:sample-size-2D}
\end{figure*}

\begin{figure*}[t]
\centering%
\subfloat[Sample size $(1,2,N_3)$]{%
\includegraphics[width=0.49\hsize]{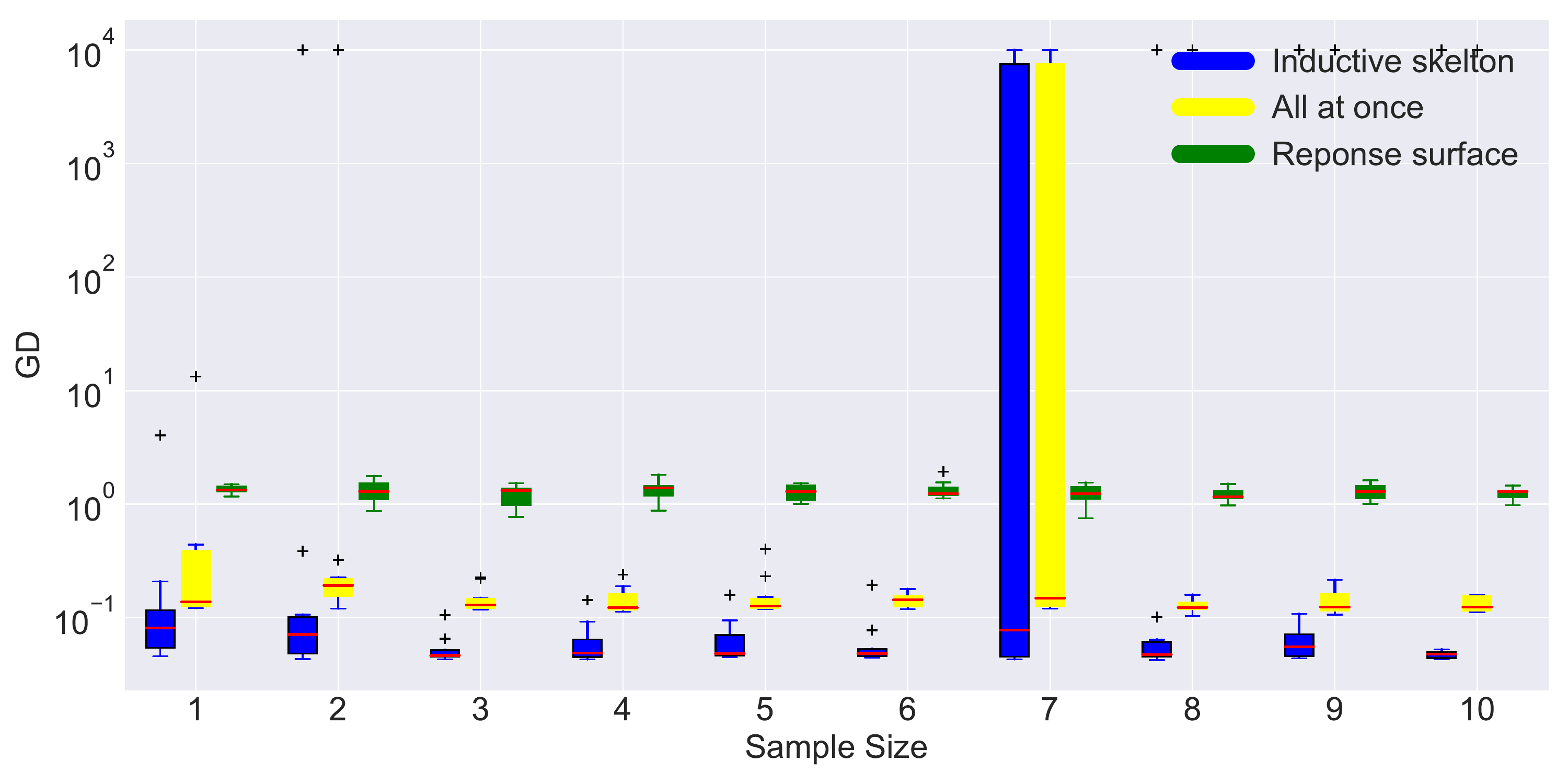}
\includegraphics[width=0.49\hsize]{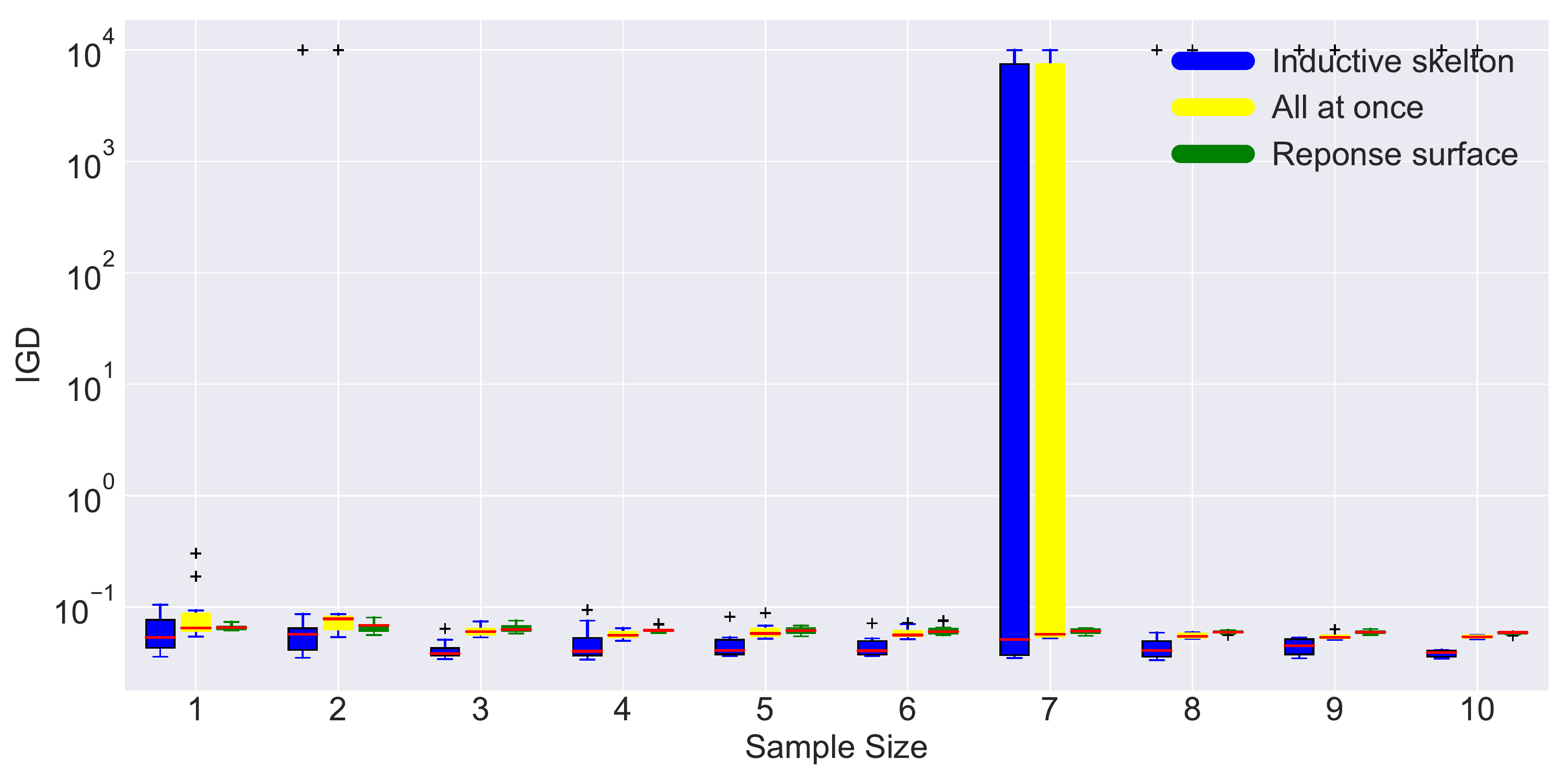}}\\
\subfloat[Sample size $(1,3,N_3)$]{%
\includegraphics[width=0.49\hsize]{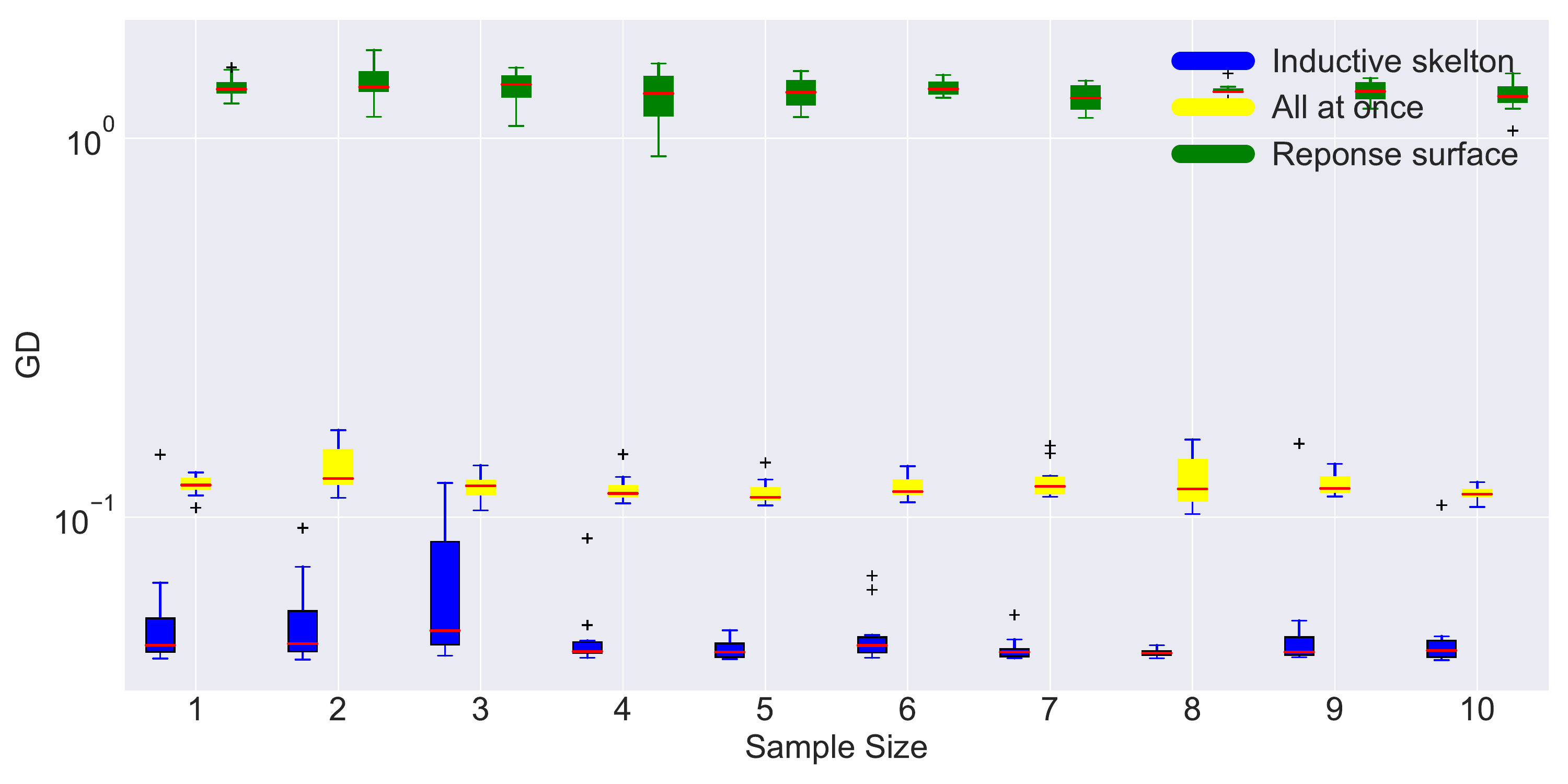}
\includegraphics[width=0.49\hsize]{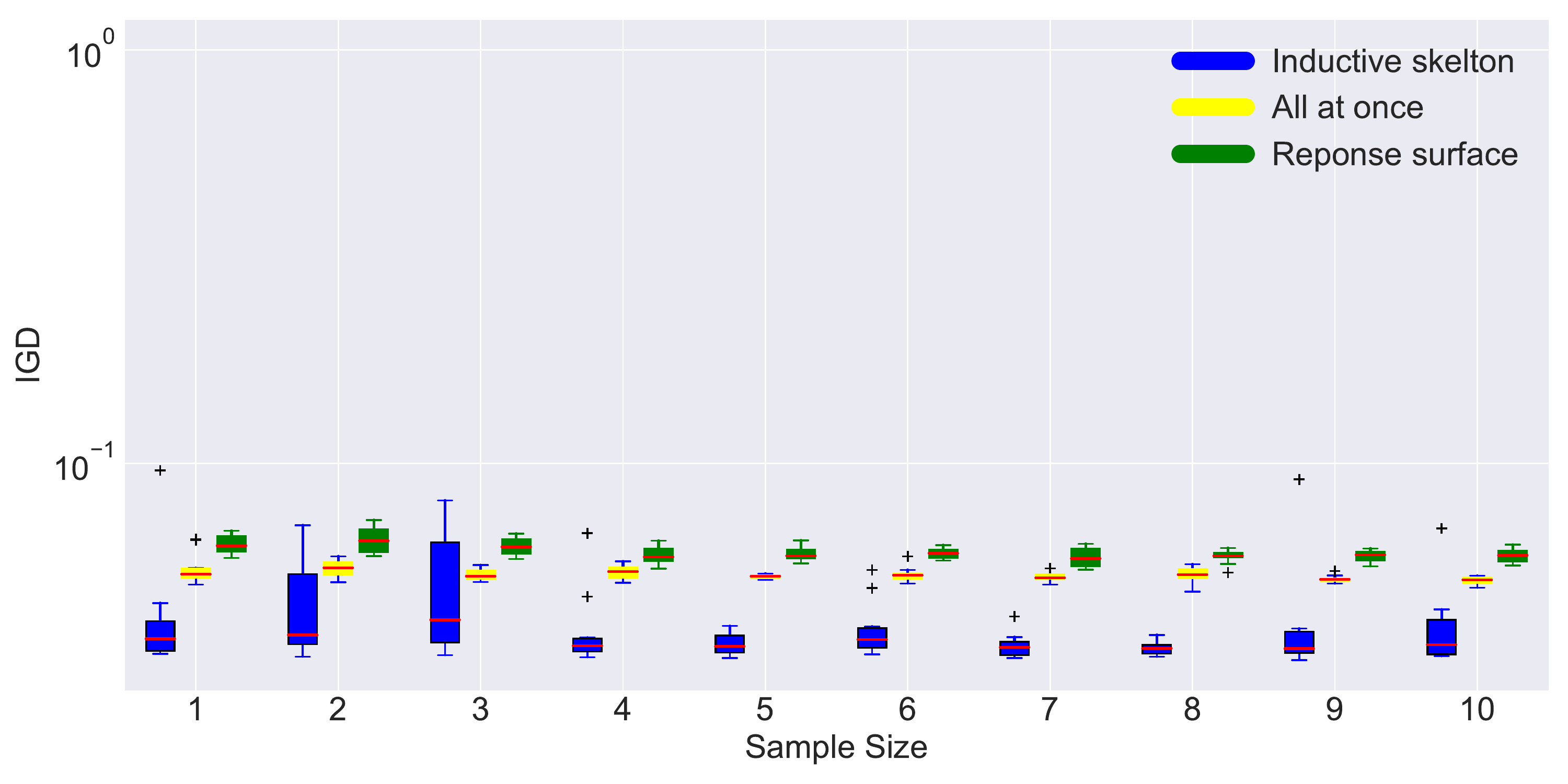}}\\
\subfloat[Sample size $(1,4,N_3)$]{%
\includegraphics[width=0.49\hsize]{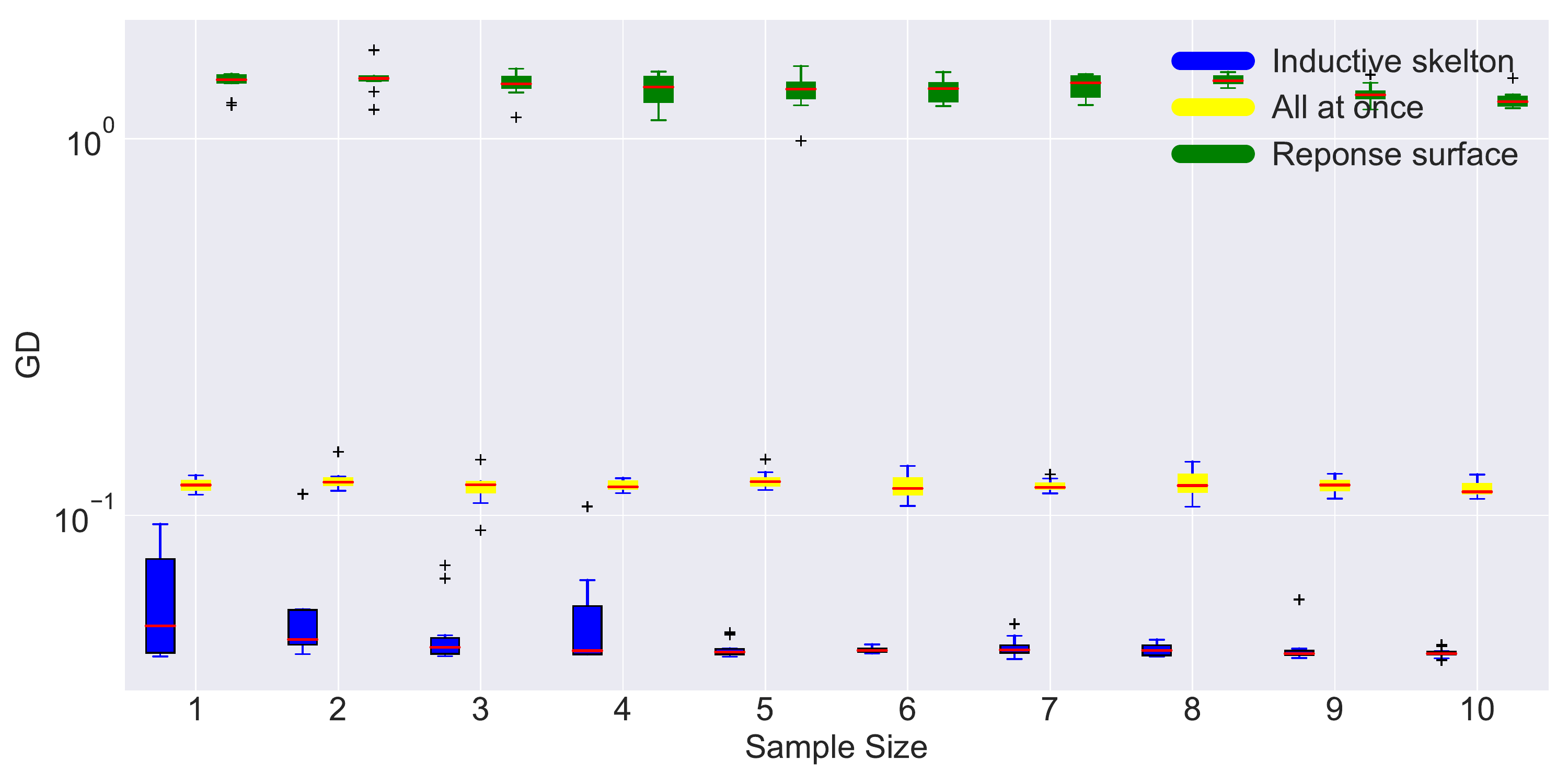}
\includegraphics[width=0.49\hsize]{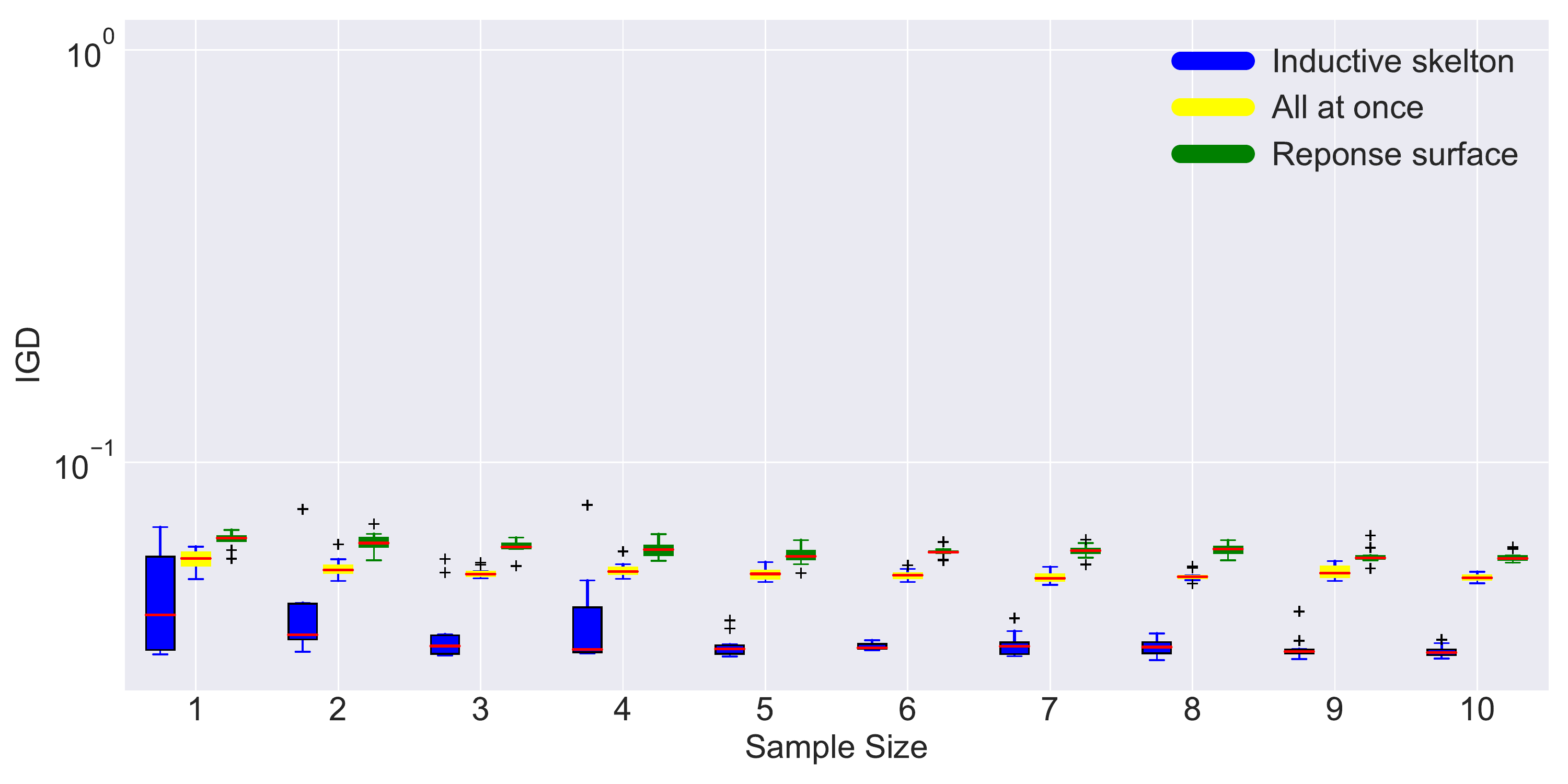}}
\caption{Sample size $N_3$ vs.\ GD/IGD on 3-MED (boxplot over ten trials).}\label{fig:sample-size-MED}
\end{figure*}

\begin{figure*}[t]
\centering%
\subfloat[Sample size $(1,2,N_3)$]{%
\includegraphics[width=0.49\hsize]{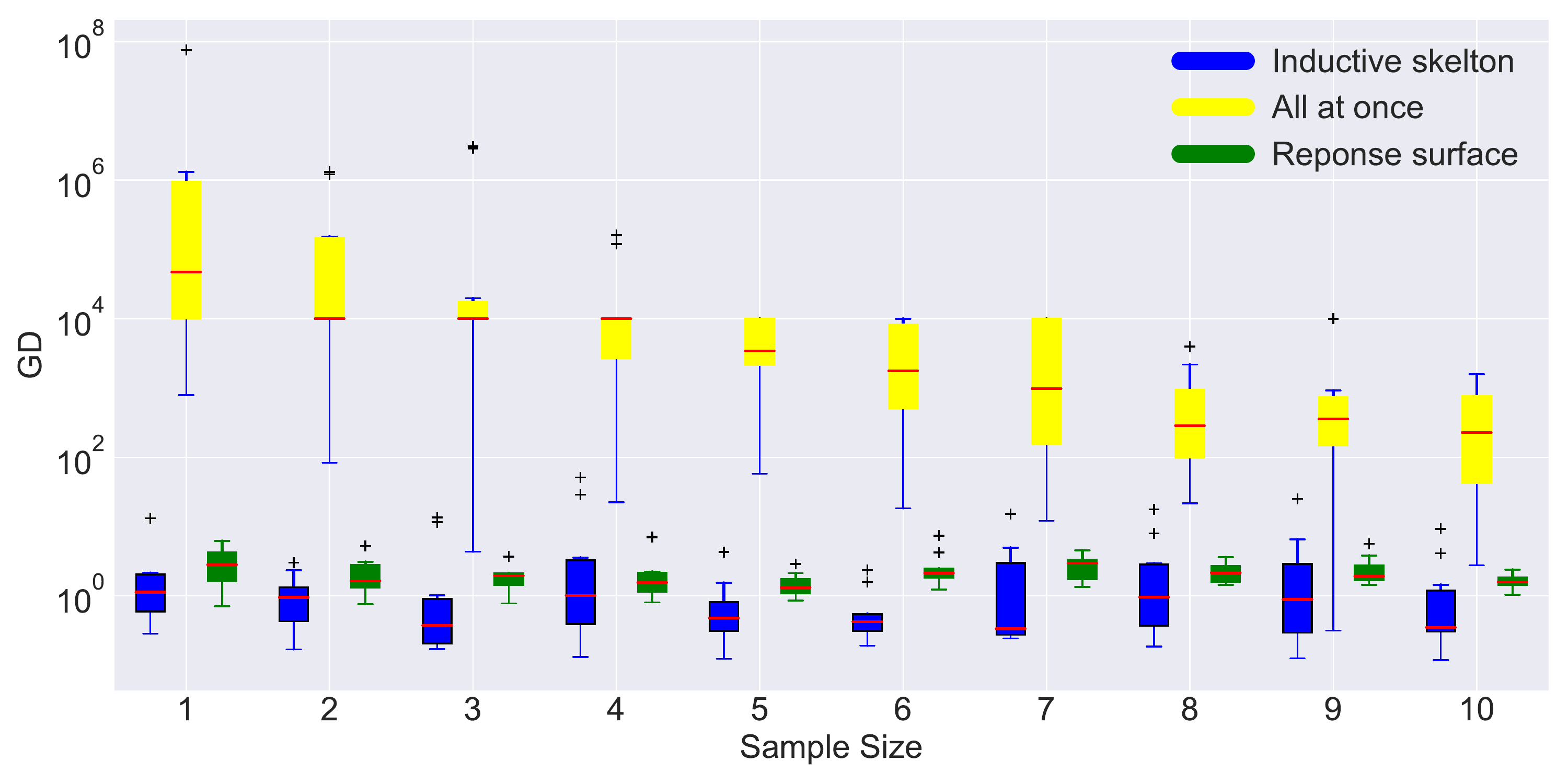}
\includegraphics[width=0.49\hsize]{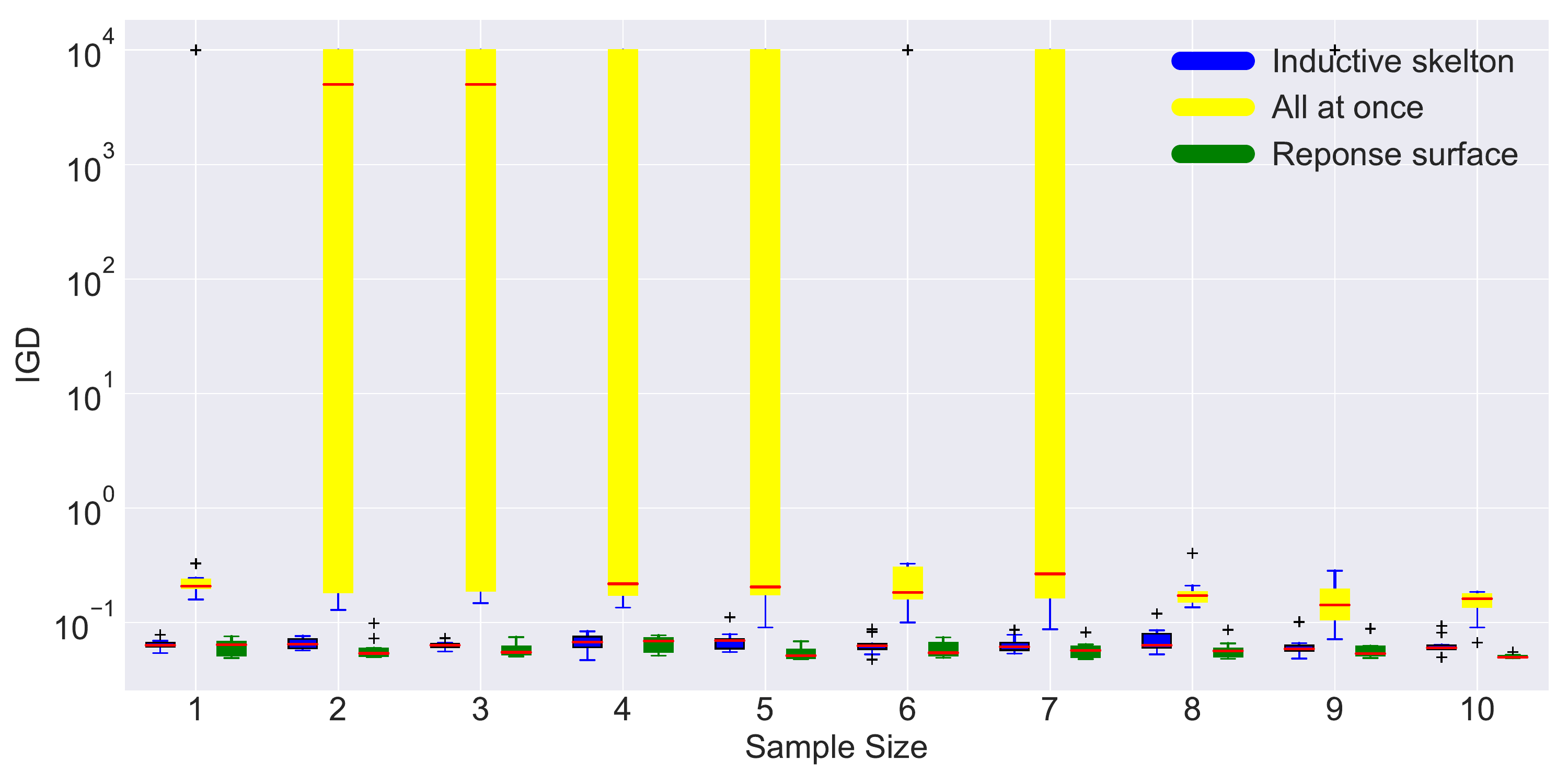}}\\
\subfloat[Sample size $(1,3,N_3)$]{%
\includegraphics[width=0.49\hsize]{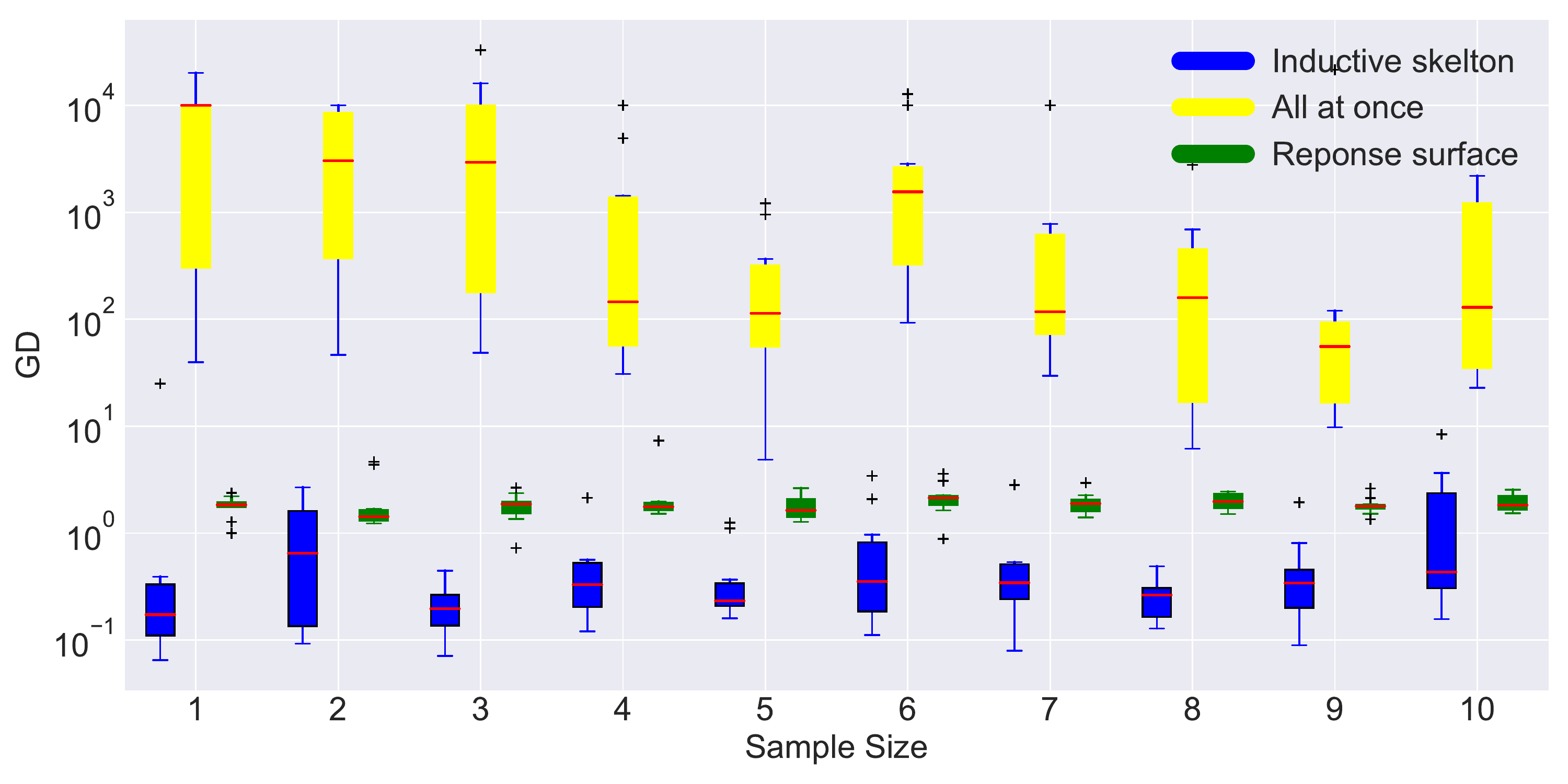}
\includegraphics[width=0.49\hsize]{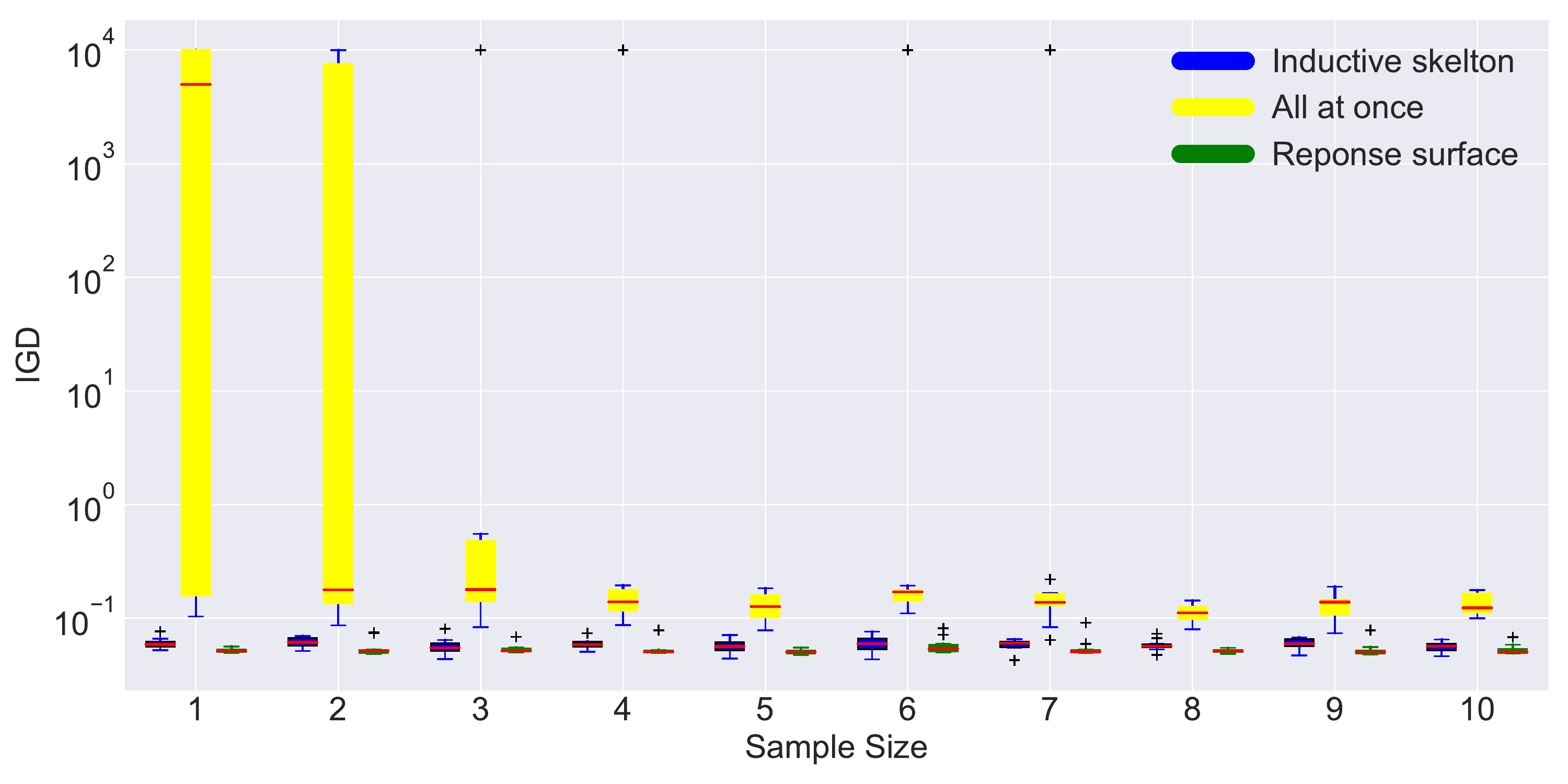}}\\
\subfloat[Sample size $(1,4,N_3)$]{%
\includegraphics[width=0.49\hsize]{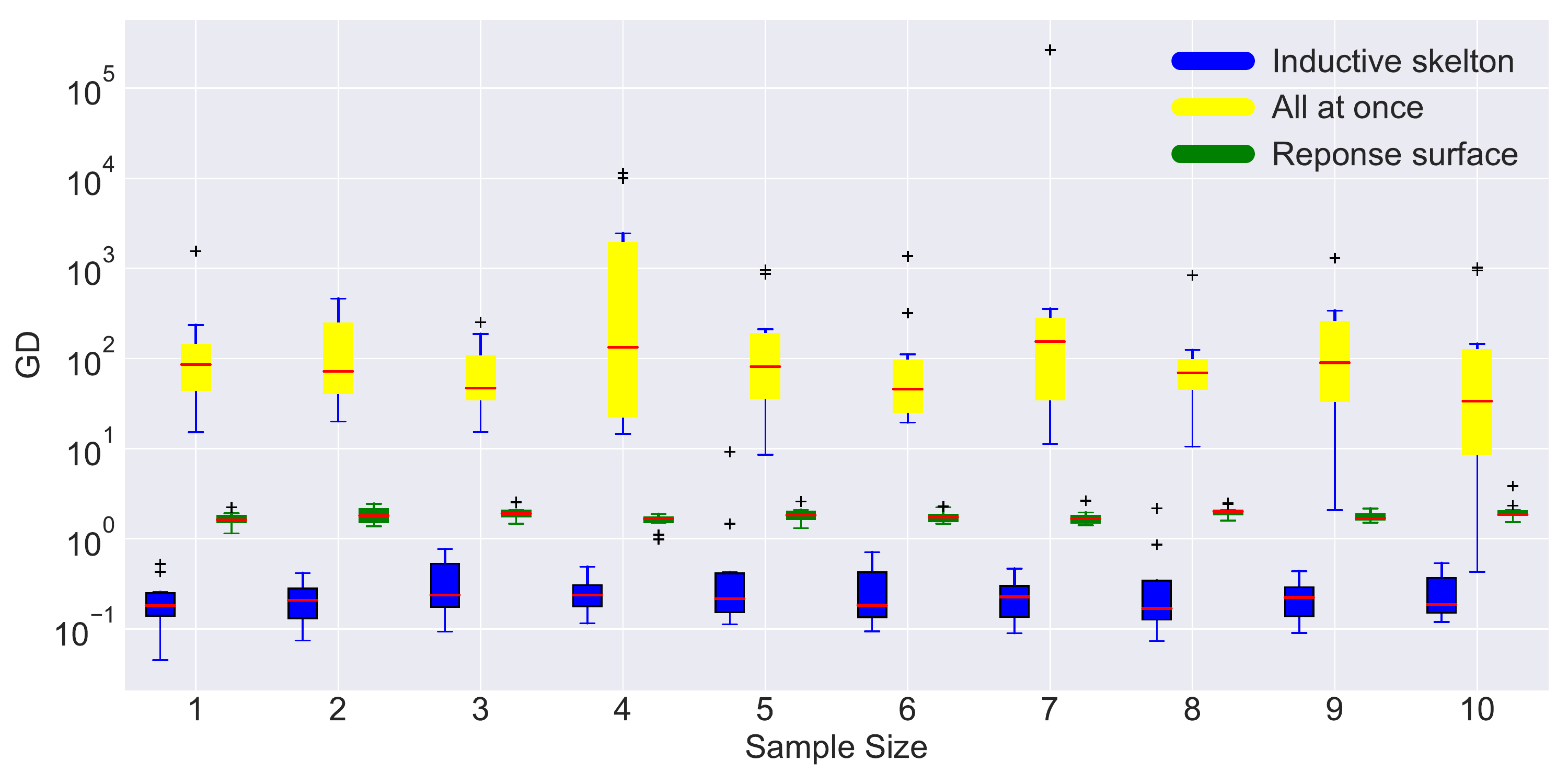}
\includegraphics[width=0.49\hsize]{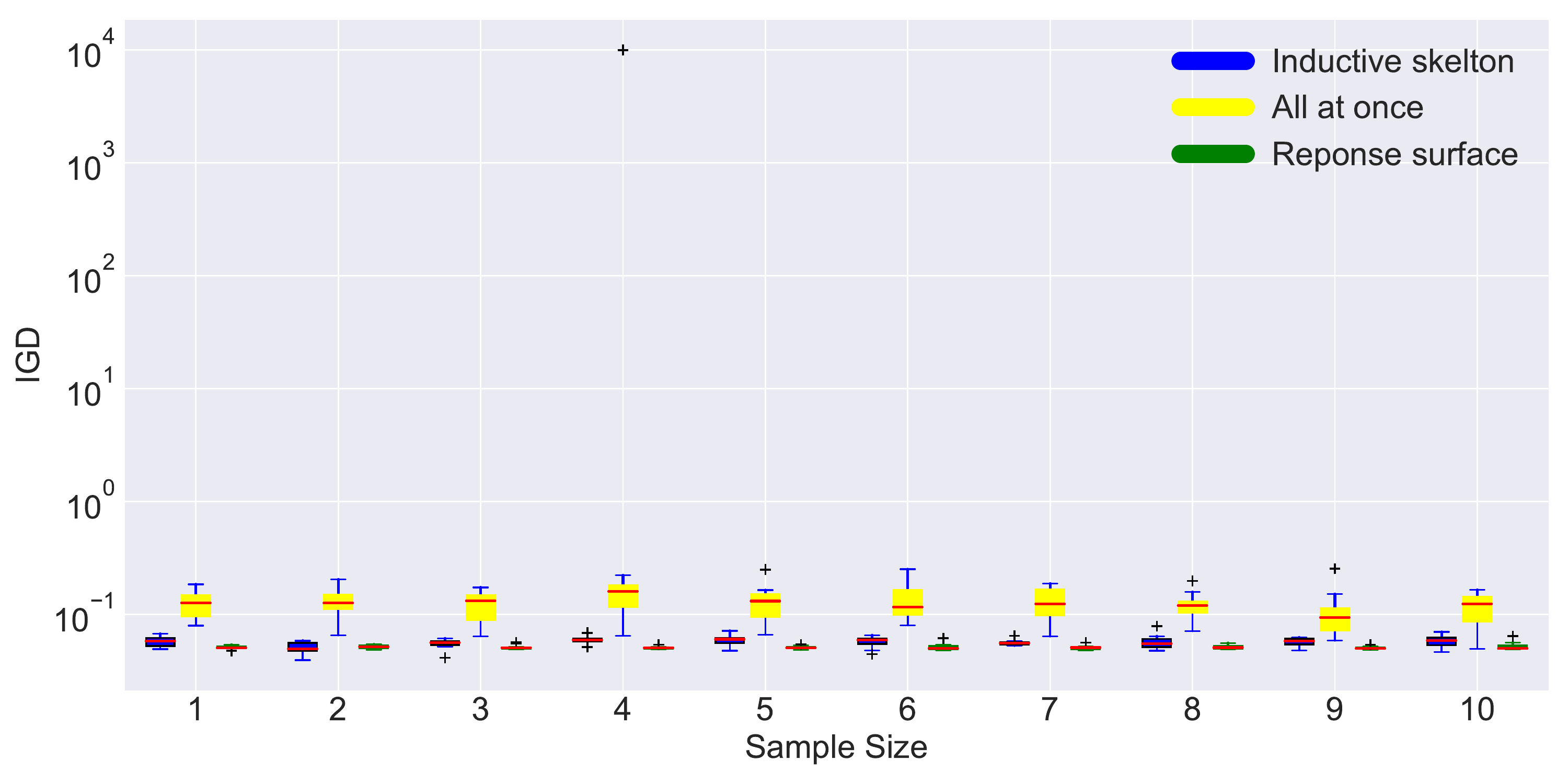}}
\caption{Sample size $N_3$ vs.\ GD/IGD on Viennet2 (boxplot over ten trials).}\label{fig:sample-size-Viennet}
\end{figure*}

\begin{figure*}[t]
\centering%
\subfloat[Sample size $(1,2,N_3)$]{%
\includegraphics[width=0.49\hsize]{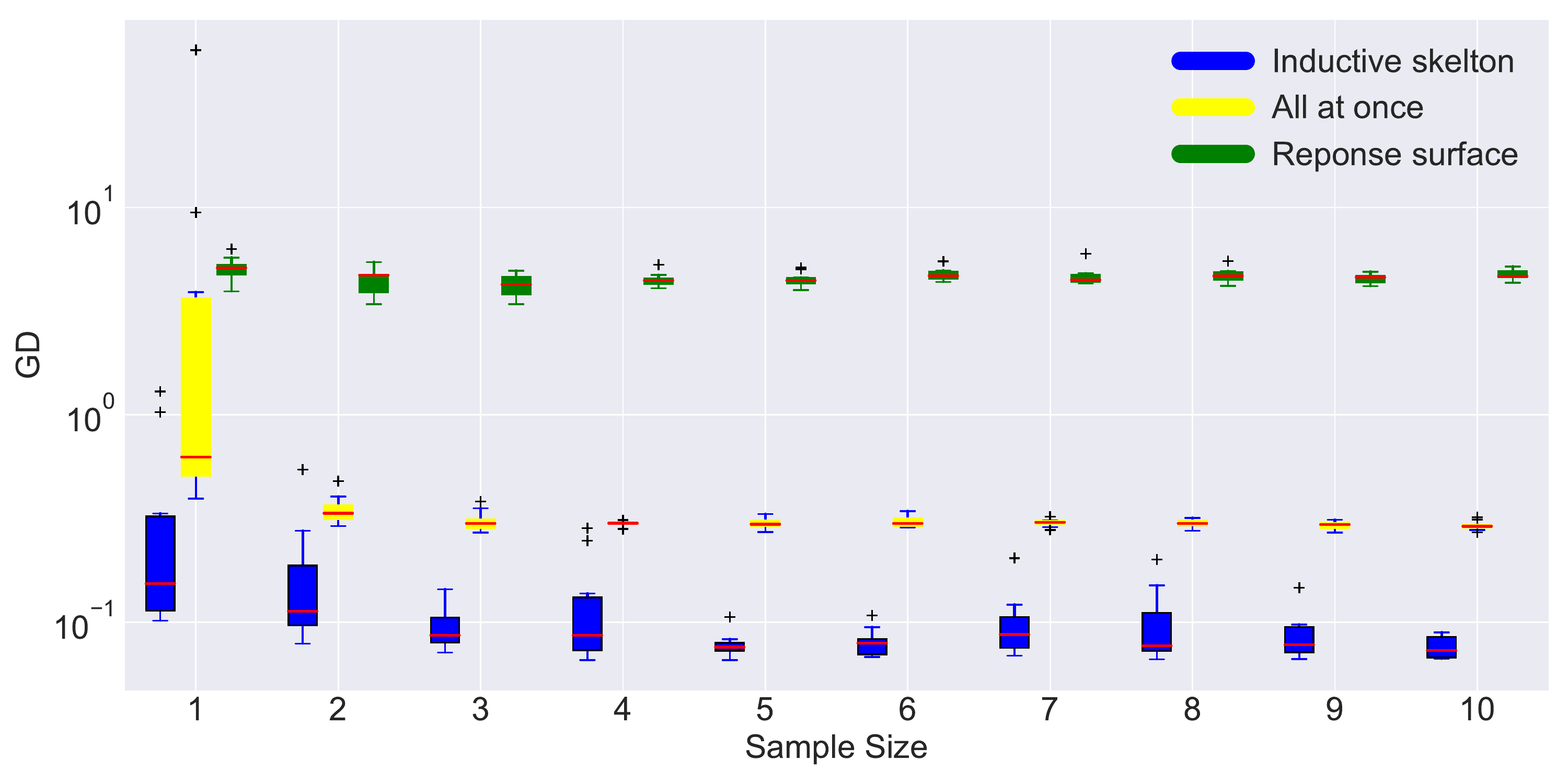}
\includegraphics[width=0.49\hsize]{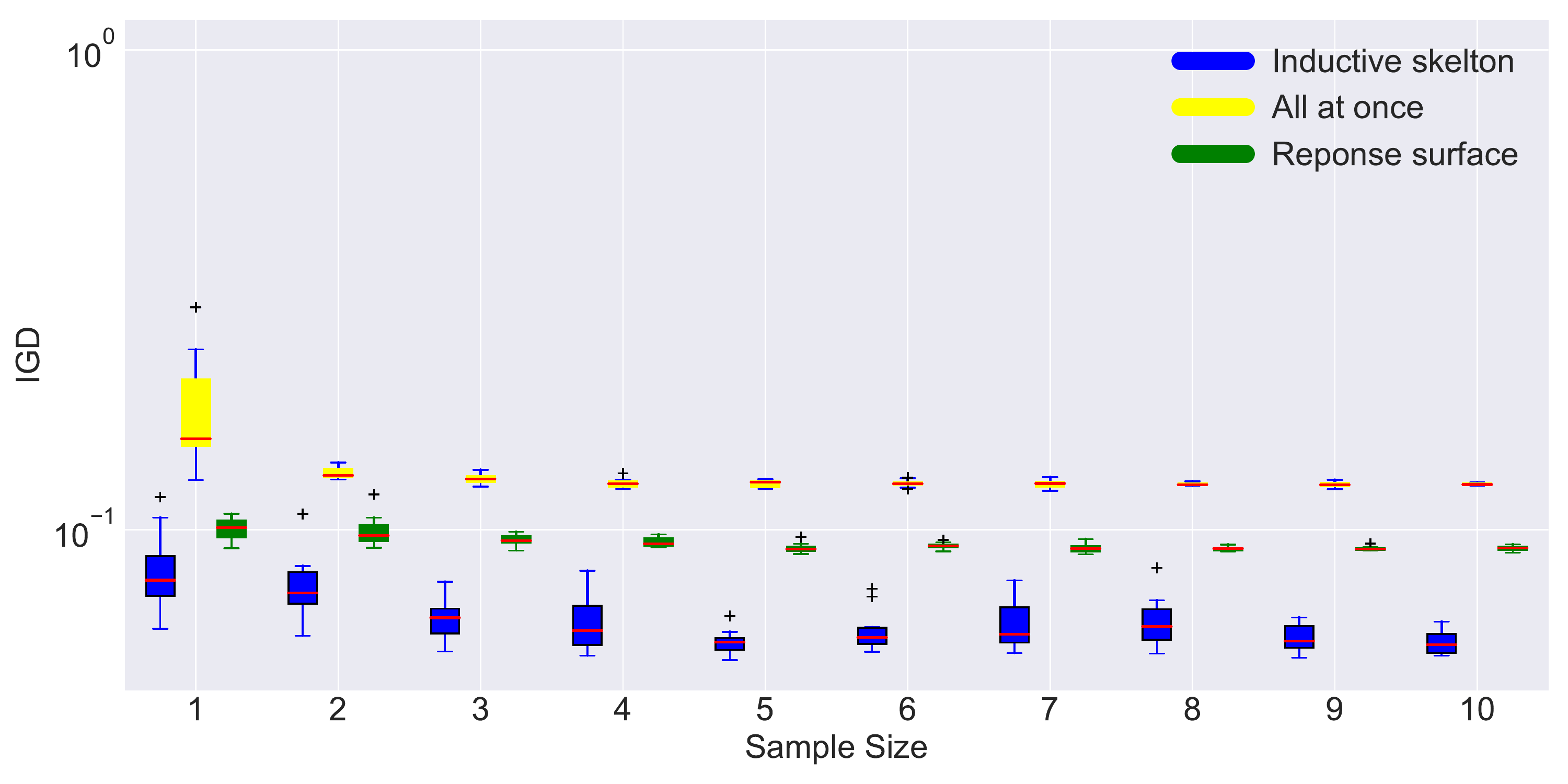}}\\
\subfloat[Sample size $(1,3,N_3)$]{%
\includegraphics[width=0.49\hsize]{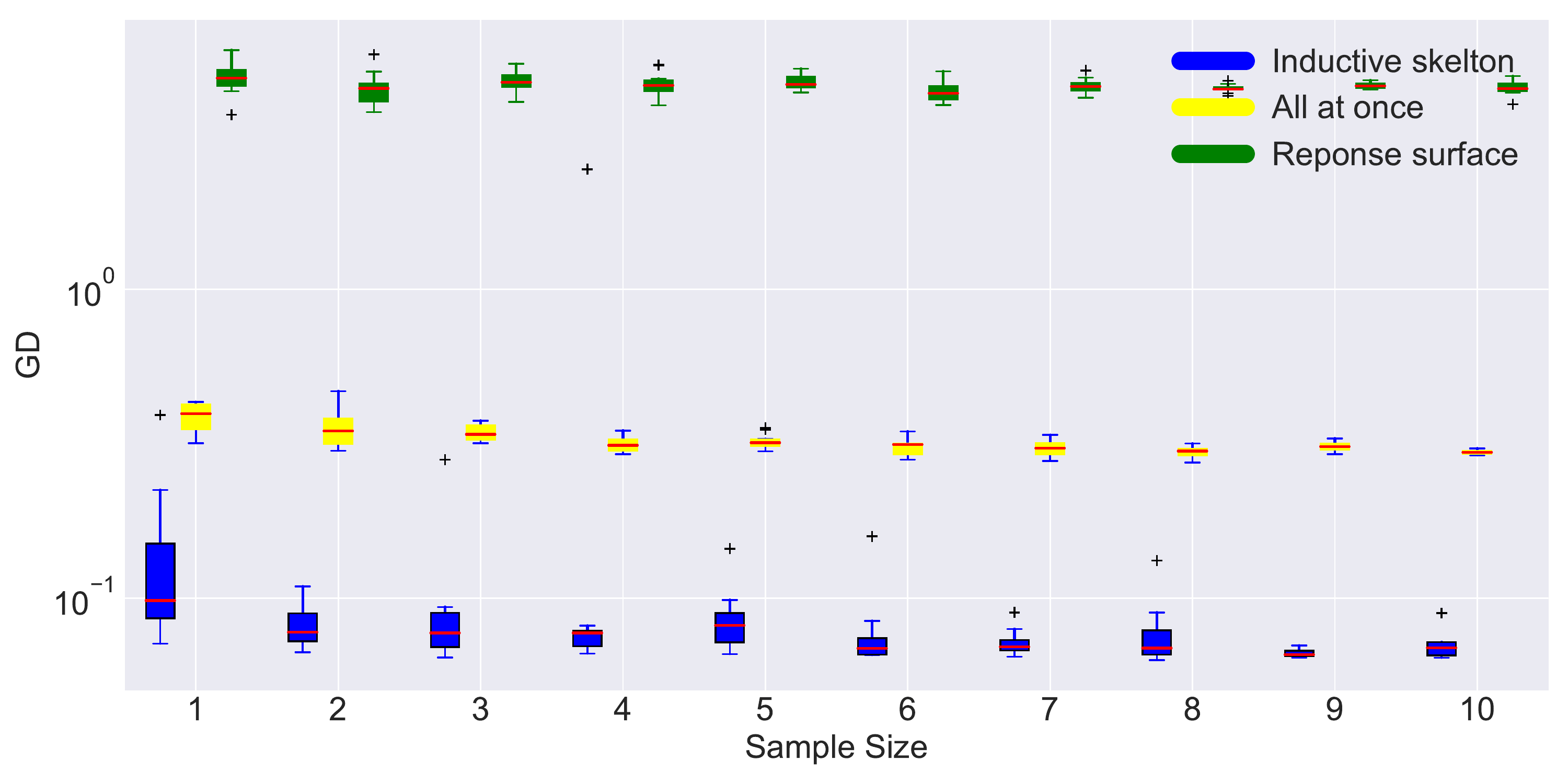}
\includegraphics[width=0.49\hsize]{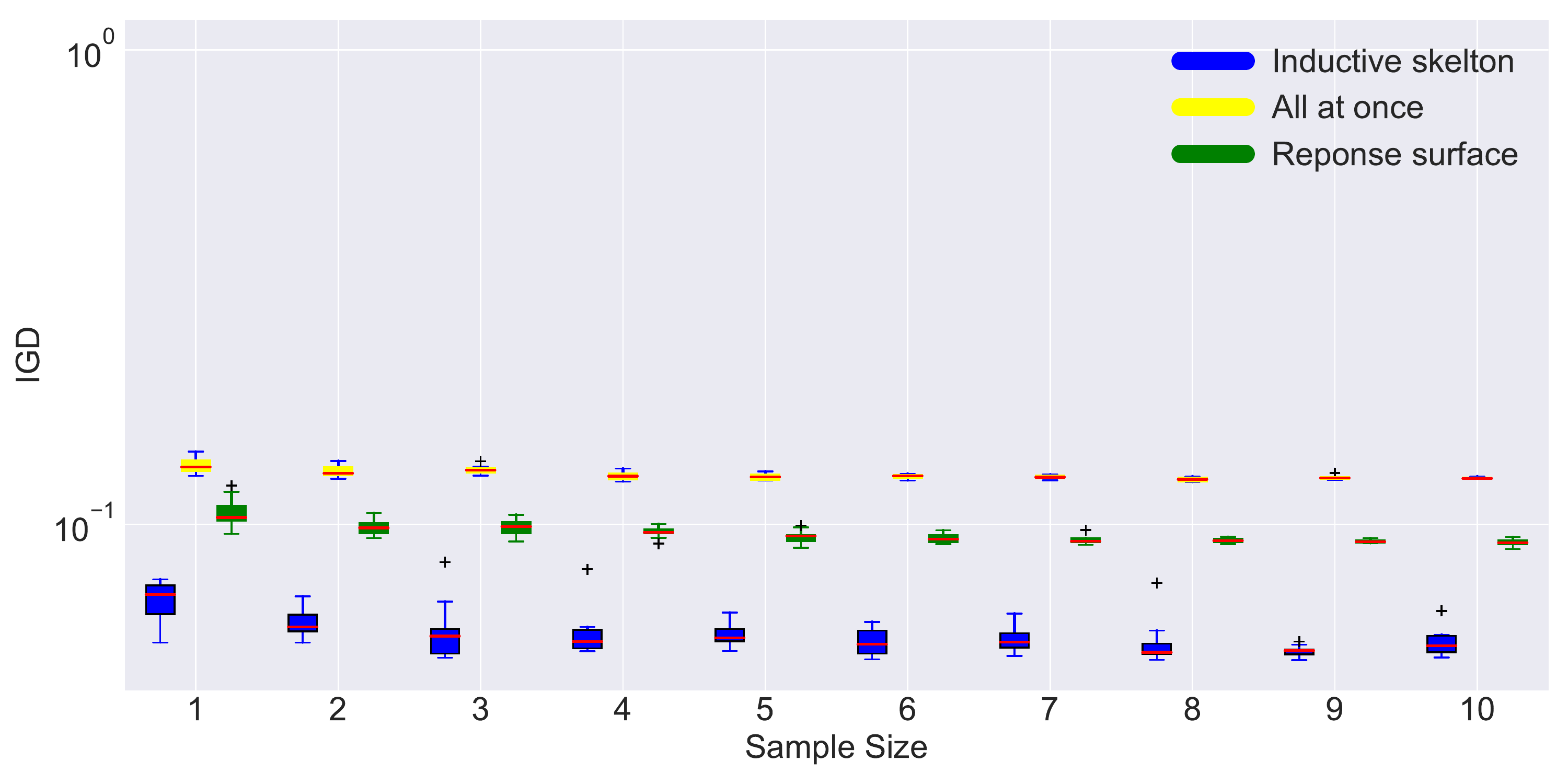}}\\
\subfloat[Sample size $(1,4,N_3)$]{%
\includegraphics[width=0.49\hsize]{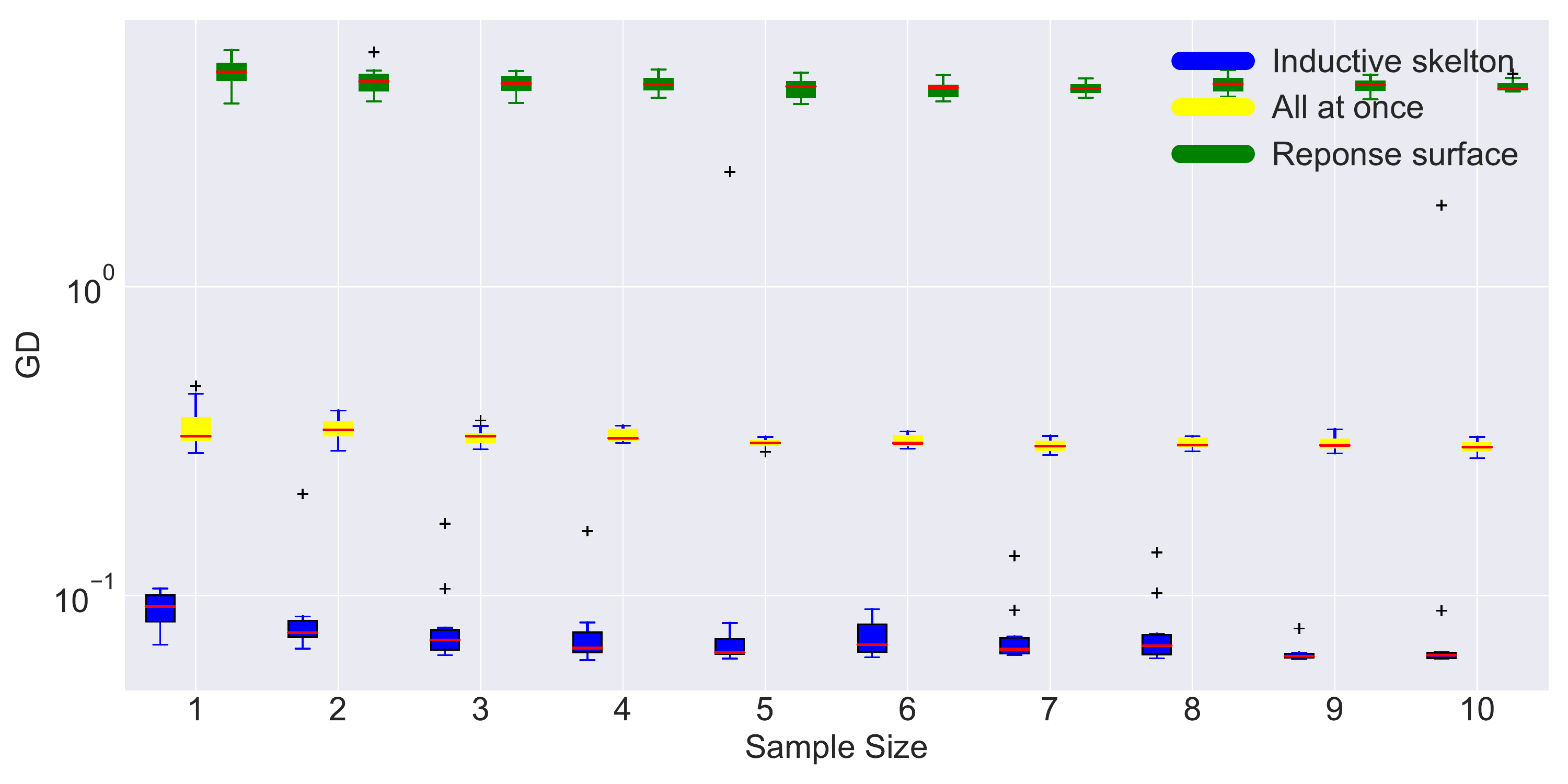}
\includegraphics[width=0.49\hsize]{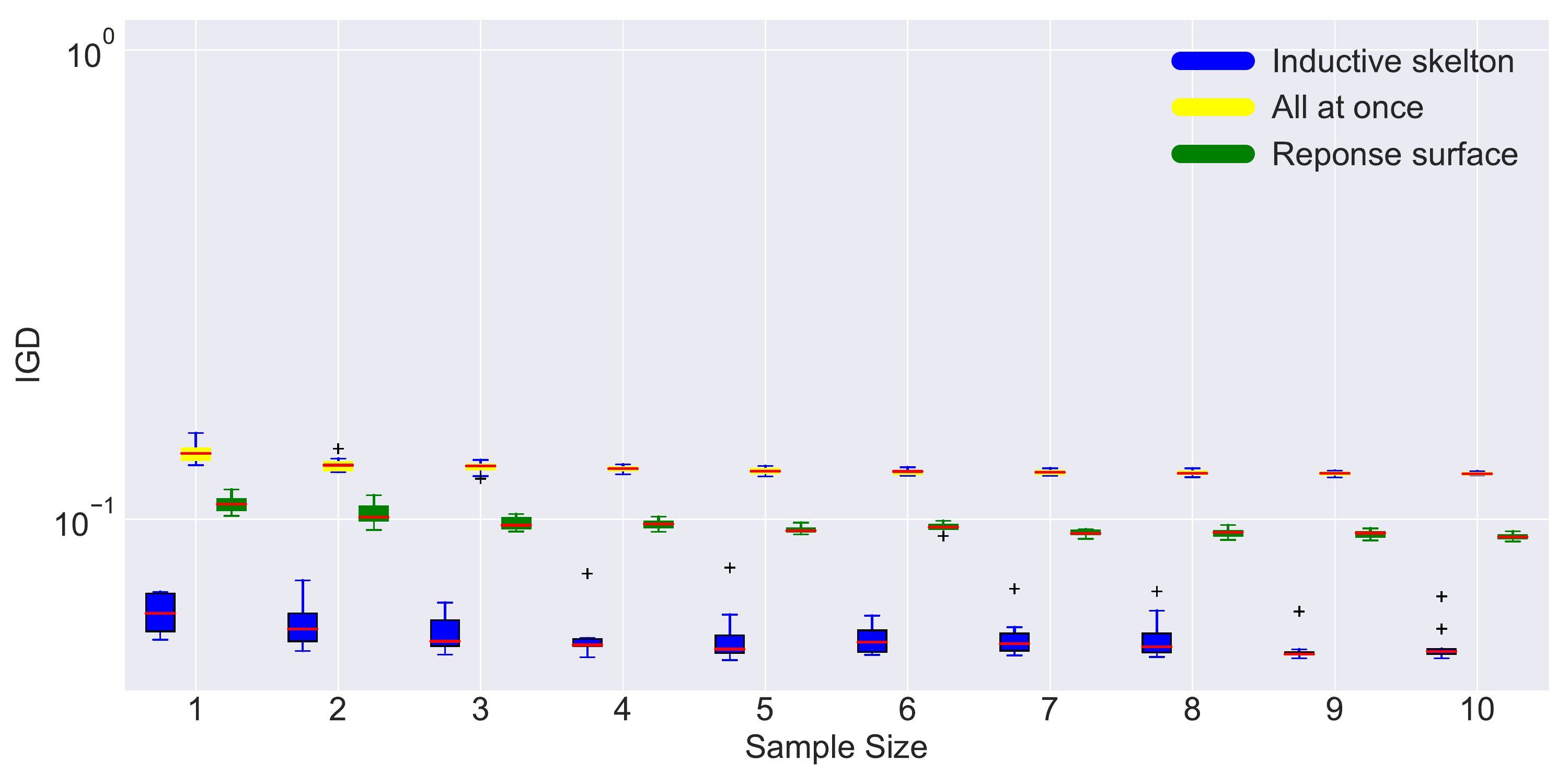}}
\caption{Sample size $N_3$ vs.\ GD/IGD on 5-MED (boxplot over ten trials).}\label{fig:sample-size-MED5_appendix}
\end{figure*}
\end{document}